\documentclass[11pt]{amsart}

\usepackage[normalem]{ulem}
\usepackage[svgnames,x11names]{xcolor}
\usepackage{amssymb,amsfonts}
\usepackage[pagebackref,colorlinks,urlcolor={DarkOliveGreen4},
linkcolor={DarkSlateBlue},
citecolor={Chocolate4}]{hyperref}
\usepackage[capitalize]{cleveref}
\usepackage{amssymb,textcomp}
\usepackage{graphicx}
\numberwithin{equation}{section}
\usepackage{bbm}
\usepackage{enumerate}
\usepackage[utf8]{inputenc}
\usepackage[T1]{fontenc}
\usepackage[mathscr]{eucal}
\usepackage{etoolbox}
\usepackage {tikz}
\usetikzlibrary {positioning}
\usepackage[a4paper, twoside=false, vmargin={2cm,3cm}, includehead]{geometry}
\usepackage{nameref}
\newtheorem{lemma}{Lemma}[section]
\newtheorem{theorem}[lemma]{Theorem}
\newtheorem*{theorem*}{Theorem}
\newtheorem{corollary}[lemma]{Corollary}
\newtheorem{question}{Open question}
\newtheorem*{question*}{Open question}
\newtheorem{proposition}[lemma]{Proposition}
\newtheorem*{proposition*}{Proposition}

\newtheorem*{problem*}{Problem}
\theoremstyle{definition}
\newtheorem{definition}[lemma]{Definition}
\newtheorem*{claim*}{Claim}
\newtheorem*{notation}{Notation}

\newtheorem*{Conjecture}{Conjecture}
\newtheorem{example}{Example}

\newtheorem{remark}{Remark}

\newcommand{\Mod}[1]{\ \mathrm{mod}\ #1}

\makeatletter
\theoremstyle{plain}
\newtheorem*{namedthm}{\namedthmname}
\newcounter{namedthm}
	

\makeatletter
\renewcommand\section{\@startsection{section}{1}%
  \z@{.7\linespacing\@plus\linespacing}{.5\linespacing}%
  {\normalfont\Large\centering\bfseries}}

\makeatother

\makeatletter

\newcommand{\N}{{\mathbb N}}

\newcommand{\Z}{{\mathbb Z}}

\DeclareMathOperator{\id}{id}

\newtheorem{theoremL}{Theorem}

\newtheorem{CorollaryL}{Corollary}

\patchcmd{\subsection}{\normalfont}{\large}{}{}
\patchcmd{\subsection}{-.5em}{.5em}{}{}

\begin{document}
\title{Distribution of integers with digit restrictions via Markov chains }

\author{Vicente Saavedra-Araya}
\email{vicente.saavedra-araya@warwick.ac.uk}
\address{University of Warwick, Department of Mathematics, Coventry, UK}	

\begin{abstract}
In this paper, we introduce a new technique to study the distribution in residue classes of sets of integers with digit and sum-of-digits restrictions. From our main theorem, we derive a necessary and sufficient condition for integers with missing digits to be uniformly distributed in arithmetic progressions, extending previous results going back to the work of   Erd\H{o}s, Mauduit and S\'ark\"ozy. 
Our approach utilizes Markov chains and does not rely on Fourier analysis as many results of this nature do. 

Our results apply more generally to the class of multiplicatively invariant sets of integers. This class, defined by Glasscock, Moreira and Richter using symbolic dynamics, is an integer analogue to fractal sets and includes all missing digits sets. We address uniform distribution in this setting, partially answering an open question posed by the same authors.

\end{abstract}

\maketitle
\small
\tableofcontents
\normalsize

\section{Introduction}
 In this paper, we investigate the distribution of sets of integers formed by imposing restrictions on the digits in a given base. Throughout this paper, we denote by $\N_0$ the set of natural numbers including zero. For an integer $g\geq 2$ and a set of digits $\mathcal{D}\subseteq \{0,\cdots,g-1\}$, let
\begin{equation}
 \mathcal{C}_{g,\mathcal{D}}:=\left\{\sum_{i=0}^m w_ig^i:\ m\in \N_0, \  w_i\in \mathcal{D} \right\}\label{def:missing_sets}   
\end{equation}
be the set of non-negative integers that can be written in base $g$ only using digits of $\mathcal{D}$.
Such sets have been widely studied (see, for example, \cite{EMS1,Kempner01021914,Mahler,MAY1}), and we will refer to them as \emph{missing digits sets.} A broader class of integers was introduced by Glasscock, Moreira, and Richter in \cite{GMR}.
\begin{definition}
    Let $g\geq 2$ be an integer. The set $A\subseteq \N_0$ is said to be \emph{$\times g$-invariant} if for every $n\in A$ with base-$g$ expansion given by $n=w_mg^m+\cdots+w_1g+w_0$, where $m\in \N_0$ and $w_m\neq 0$, it holds
    \[w_{m-1}g^{m-1}+\cdots+w_1g+w_0\in A \quad \text{and}\quad w_{m}g^{m-1}+\cdots+w_{2}g+w_1\in A.\]
    We say a set of integers is \emph{multiplicatively invariant} if it is $\times g$-invariant for some $g\geq 2$.\label{def:invariant_sets}
    \end{definition}
    Notice that the first condition in \cref{def:invariant_sets} indicates that the set $A$ is invariant under the operation of deleting the most significant digit in base $g$, while the second condition ensures that it is invariant under the operation of deleting the least significant digit. Clearly, integers with missing digits belong to this class, but so do many other sets described by broader digit restrictions (for example, by forbidding specific combinations of digits). Unlike missing digits sets, much less is known about multiplicatively invariant sets.

    Motivated by the study of the distribution of integers with missing digits in arithmetic progressions by Erdős, Mauduit and Sárközy in \cite{EMS1}, and by an open question posed by Glasscock, Moreira, and Richter in \cite[Question 5.6]{GMR} concerning the interaction between multiplicatively invariant sets and arithmetic progressions, this paper investigates distribution properties of multiplicatively invariant sets.  A central idea is to understand when a set of integers visits different congruence classes  with the same frequency. In this sense, $A\subseteq \N_0$ is said to be \emph{uniformly distributed$\Mod{a}$} (or in arithmetic progressions) if
\[\lim_{N\to \infty}\dfrac{|\{n\in A: n\equiv b\Mod{a}\}\cap [0,N)|}{|A\cap [0,N)|}=\dfrac{1}{a}\] for every $b\in \Z_a:=\Z/a\Z$.
To explore this idea in more generality, we use  the notion of \emph{$g$-additive} functions (firstly considered by Bellman and Shapiro in \cite{BS}). 
We say $f:\N_0\to \Z$ is \emph{$g$-additive} if 
\[f\left(\sum_{i=0}^m w_ig^i\right)=\sum_{i=0}^m f(w_ig^i)\]
for every $w_0,\ldots,w_m\in \{0,\ldots,g-1\}.$  
Let $A\subseteq \N_0$ be $\times g$-invariant and $r\in \N$. For all $i\in \{1,\ldots,r\}$, let  $f_i$ be  a $g$-additive function, $a_i\in \N$ and $b_i\in \Z_{a_i}$.  The main objective in this work is to study the limiting behaviour of  \begin{equation}
    \frac{|\left\{n\in A:\ f_i(n)\equiv b_i\Mod{a_i} \text{ for $i=1,\ldots,r$}\right\}\cap [0,N)|}{|A\cap [0,N)|}\label{eq:distribution}\end{equation}
    when $N\to \infty$, and determine conditions to guarantee the limit equals to $1/(a_1\cdot a_2\cdots a_r)$ for all $b_i\in \Z_{a_i}$. When this happens, we will say that $A$ is \emph{uniformly distributed $\mathbf{f}\Mod{\mathbf{a}}$}, where $\mathbf{f}=(f_1,\ldots,f_r)$ and $\mathbf{a}=(a_1,\ldots,a_r)$. Our main idea lies in the construction of suitable Markov chains that capture the behaviour of \eqref{eq:distribution}, and then study it using classical convergence results for Markov chains. Before stating our main results, we will briefly review some previous results and introduce the context of the problem to be investigated.

    \subsection{History and context}
The study of the distribution of sets of integers with digit restrictions dates back to the works of Fine \cite{Fine} and Gelfond \cite{Gelfond}. For an integer $g\geq 2$, they considered the function sum-of-digits in base $g$, denoted by $S_g$ and defined as follows: If $w_0,\cdots,w_{m}\in \{0,\cdots,g-1\}$,  \begin{equation}S_g\left(\sum_{i=0}^m w_ig^i\right)=w_0+\cdots+w_m.\label{sum_digits}\end{equation}
In this context, Gelfond  studied the interplay between $n$ and $S_g(n)$ for positive integers in terms of their distribution in residue classes, providing sufficient conditions for uniform distribution.  A straightforward consequence of \cite[Theorem 1]{Gelfond} is presented below.
\begin{corollary} Let $a,a'\in \N$ such that $\gcd(g-1,a')=1$. Then,  
\begin{equation}
    \lim_{N\to \infty}\dfrac{|\{n\leq N:\ n\equiv b\Mod{a},\ S_g(n)\equiv b'\Mod{a'}\}|}{N}=\dfrac{1}{aa'}\label{Gelfond}
\end{equation}
for any $b\in \Z_a$ and $b'\in \Z_{a'}$.\label{Gelfond_thm}
\end{corollary} 
 It is worth mentioning that Gelfond's result also gives quantitative information about the convergence. The \cref{Gelfond_thm} shows that the sequences $n\Mod{a}$ and $S_g(n)\Mod{a'}$ are independent; in other words, integers with a fixed sum of digits$\Mod{a'}$ are uniformly distributed in arithmetic progressions (a related result can be found in \cite{MS2}). As part of our work, we provide an equivalent condition to the uniform distribution exhibited in \eqref{Gelfond}, thereby extending previous works.
Results concerning  the independence between sum-of-digits functions in different bases can be found in \cite{Besineau} and \cite{KDH}. For an overview of results of this type, see \cite[Section 4.3]{Sandor_Crstici}.

In a related direction, the distribution of missing digits sets in residue classes was initially studied by Erd\H{o}s, Mauduit and S\'ark\"ozy. Let  $\mathcal{D}\subset\{0,\ldots,g-1\}$ be a subset of digits (with at least 2 elements). Regarding the set $\mathcal{C}_{g,\mathcal{D}}$ defined in \eqref{def:missing_sets}, they showed the following result in \cite[Theorem 1]{EMS1}.

\begin{theorem} Assume $\gcd\left(g(g-1),a\right)=1$, $0\in \mathcal{D}$ and $\gcd\{d\in \mathcal{D}:\ d\neq 0\}=1.$ Then, there exists $c_1,c_2$ and $c_3$ (all dependent only on $g$ and $|\mathcal{D}|)$ such that 
\[\left|\dfrac{|\left\{n\in \mathcal{C}_{g,\mathcal{D}}:\ n\equiv b\Mod{a}\right\}\cap [0,N)|}{|\mathcal{C}_{g,\mathcal{D}}\cap[0,N)|}-\dfrac{1}{a}\right|\leq \dfrac{c_2}{m}\exp \left(-c_3\dfrac{\log(N)}{\log(a)}\right)\]
    for any $a<\exp\left(c_1(\log(N))^{1/2}\right)$ and $b\in \Z_a$.\label{EMS}
\end{theorem}

Further generalization of \cref{EMS} were provided by Konyagin \cite[Corollary 1]{SK} and  Aloui \cite[Corollary 2.2]{Aloui}, with the latter also generalizing \cref{Gelfond_thm} to the case of missing digits sets. A qualitative version of \cref{EMS} can be 
presented as follows.
\begin{corollary}
    Assume $\gcd(g(g-1),a)=1$, $0\in \mathcal{D}$ and $\gcd\{d\in \mathcal{D}:\ d\neq 0\}=1.$ Then, $\mathcal{C}_{g,\mathcal{D}}$ is uniformly distributed$\Mod{a}$.\label{EMS_Qualitative}
\end{corollary}

As a consequence of our work, we strengthen \cref{EMS_Qualitative}, providing an equivalent condition to the uniform distribution when $\gcd(g,a)=1$.

The study of arithmetic properties for sets with digit restrictions has been an active area in number theory and is not limited to the problems covered in this paper. Different directions that have been explored include, for instance, studies on the number of prime factors for elements of sets with digit restrictions \cite{ DM1, DM2,EMS2, KMS}, as well as properties of divisibility \cite{BWS, SC}. It is also worth mentioning the breakthrough work of
Mauduit and Rivat \cite{MR} about the distribution of the function sum-of-digits for primes in residue classes and the works of Maynard \cite{MAY1, MAY2} about primes with missing digits, showing the existence of infinitely many primes with one missing digit. Additional references include \cite{AMM,BCS, DMS, Dartyge_Sarkozy, DMC,DMR,MS, CS,Thus, WW, Yu}.

 On the other hand,  a famous conjecture in fractal geometry and dynamical systems was posed by Furstenberg in \cite{Furst}, which  stated that $\times p \mod 1$ and $\times q\mod 1$ invariant sets of $[0,1]$ are transverse\footnote{For a given notion of dimension, the sets $A,B\subseteq [0,1]$ are said to be transverse if $\dim(A\cap B)\leq \max\{0,\dim(A)+\dim(B)-1\}$.} when $\log(p)/\log(q)\notin \mathbb{Q}$. This conjecture was independently proven by Shmerkin \cite{Shm} and Wu \cite{Wu} (see \cite{Austin} for an alternative proof). Motivated by an integer version of this conjecture, Glasscock, Moreira and Richter  introduced multiplicatively invariant sets of integers (\cref{def:invariant_sets}), proving in  \cite[Theorem B]{GMR} an  analogous result to Furstenberg's conjecture in the setting of integer fractals. Seeking a better understanding of invariant sets and their interaction with other sets from a fractal point of view, they raised the following question \cite[Question 5.6]{GMR}:
\begin{question}
    Let   $A\subseteq \N_0$ be a multiplicatively invariant set and $P$ be an arithmetic progression. Is it true that $\dim(A\cap P)$ is either 0 or $\dim(A)$?\label{OpenQuestion}
\end{question}

Particularly, they used the notion of mass dimension (see \cref{Mass_dimension}). The answer to this question is only known for
certain missing digits sets (those covered by previous results about uniform distribution), but has not been explored in generality. In this paper, we provide a negative answer to this question in the general case, and describe a class of sets where the answer is affirmative. Finally, we propose a conjecture identifying the precise class of multiplicatively invariant sets that satisfy this property.

\subsection{Overview of main results}
The main contribution of this paper lies in the introduction of a new approach to study the distribution of certain sets of integers via Markov chains. From this technique, we can retrieve several results, some of which are presented below.

In the direction of the works of Fine \cite{Fine}, Gelfond \cite{Gelfond}, and  Erd\H{o}s, Mauduit and S\'ark\"ozy \cite{EMS1}, we study the distribution in residue classes of set of integers with digits restrictions. A first consequence of the \cref{MainTheorem}, our main technical result, establishes conditions for a set with missing digits  to be uniformly distributed in residue classes. In contrast to \cref{EMS_Qualitative} and Aloui's result \cite[Theorem 2.9]{Aloui}, which provide only sufficient conditions, our result offers conditions that are both necessary and sufficient for uniform distribution.

\begin{theoremL}
Let $g\geq 2$ be an integer and $a,a'\in \N$ such that $\gcd(g,a)=\gcd(a,a')=1$. Let $\mathcal{D}=\{d_1,\cdots,d_t\}\subseteq \{0,\cdots,g-1\}$ be a subset of digits (where $d_1<\ldots<d_t$) and $\mathcal{C}_{g,\mathcal{D}}$ be the missing digits set defined in \eqref{def:missing_sets}. Then, 
\[\lim_{N\to\infty}\dfrac{|\{n\in \mathcal{C}_{g,\mathcal{D}}:\ n\equiv b\Mod{a},\ S_g(n)\equiv b'\Mod{a'}\}\cap [0,N)|}{|\mathcal{C}_{g,\mathcal{D}}\cap [0,N)|}=\dfrac{1}{aa'}\]
for every $(b,b')\in \Z_a\times \Z_{a'}$ if and only if
$\gcd(aa',d_2-d_1,\cdots,d_t-d_1)=1.$
    \label{thmA}
\end{theoremL}
Indeed, we provide a comprehensive description of the limiting distribution even when  $\gcd(aa',d_2-d_1,\cdots,d_t-d_1)\neq 1$ (\cref{Missing_Theorem}), and  offer insight into the rate of convergence. In \cref{thm:General}, we also present an equivalent condition without the assumption $\gcd(a,a')=1$.

By setting $a'=1$, we generalize \cref{EMS_Qualitative} by providing an equivalent condition. Furthermore, we are able to recover the same convergence speed as that obtained by Erd\H{o}s, Mauduit and S\'ark\"ozy, albeit with a difference in the constants involved. 

\begin{CorollaryL} Let $g\geq2$ be an integer, and $a\in \N$ such that $\gcd(g,a)=1$.  Let $\mathcal{D}=\{d_1,\cdots,d_t\}\subseteq \{0,\cdots,g-1\}$ be a subset of digits, where $d_1<\ldots<d_t.$ Then, $\mathcal{C}_{g,\mathcal{D}}$ is uniformly distributed $\Mod{a}$ if and only if $\gcd(a,d_2-d_1,\cdots,d_t-d_1)=1.$\label{CorA}
\end{CorollaryL}
For a fully description of the distribution $\Mod{a}$, see \cref{cor:Missing_Theorem}. Our results are not limited to the distribution of integers with missing digits, but more general digit restrictions are allowed. In \cref{Section5:SFT} and \cref{Section5:sofic}, we address more general restrictions and provide examples of sets with missing combinations of digits that exhibit uniform distribution. Additionally, our method allows us to revisit the result of Gelfond, yielding an equivalent condition to the uniform distribution exhibited in \eqref{Gelfond}. In particular, this shows that the sufficient condition $\gcd(g-1,a')=1$ provided in \cref{Gelfond_thm} is not necessary.

\begin{theoremL} Let $g\geq 2$ be an integer, and $a,a'\in \N$. Then,
\[\lim_{N\to\infty}\dfrac{|\{n\leq N:\ n\equiv b\Mod{a}, \ S_g(n)\equiv b'\Mod{a'}\}|}{N}=\dfrac{1}{aa'}\]
for every $(b,b')\in \Z_a\times \Z_{a'}$  if and only if there exists $n\in \N$ such that $n\equiv 0\Mod{a}$ and $S_g(n)\equiv 1\Mod{a'}.$
    \label{thmB}
\end{theoremL}

Notice that the existence of $n\in\N$ that satisfy $n\equiv 0\Mod{a}$ and $S_g(n)\equiv 1\Mod{a'}$ can be established from an elementary construction\footnote{The author thanks for the useful construction provided in \url{https://mathoverflow.net/q/482398}. } under the condition $\gcd(g-1,a')=1$.

Finally, we provide a negative answer in the general case to the open question posed by Glasscock, Moreira and Richter \cite[Question 5.6]{GMR}, and we provide a description using symbolic dynamics of sets of integers satisfying affirmatively the statement. An important idea is that every multiplicatively invariant set of integers can be characterized by a dynamical structure known as a subshift (see \cref{languageset}). Two important classes of subshifts are sofic subshifts and transitive subshifts (see \cref{Section2:Symbolic} for details). Using these notions, we establish the following partial answer to the \cref{OpenQuestion}.

\begin{theoremL} Let $A\subseteq \N_0$ be a $\times g$-invariant set that can be represented by a transitive sofic subshift. Then, for every infinite arithmetic progression $P$, $\dim_{\text{M}}(A\cap P)$ is either 0 or $\dim_{\text{M}}(A)$. Moreover, if the subshift is only transitive, or only sofic, the result is not always true.
    \label{thmC}
\end{theoremL}
Refer to \cref{thm:Transversality} for a more general version of this result. Furthermore, we show that the sofic and transitive setting is not optimal, as the affirmative answer can be extended to more general classes of subshifts. Thus, using the notion of topological entropy (see \cref{entropy} and \cref{EntropyMinimal}), we propose the following conjecture.

\begin{Conjecture}
    Let $A\subseteq \N_0$ be a $\times g$-invariant set. There exists an entropy minimal subshift that represents $A$ if and only if, for every arithmetic progression $P$, $\dim_{\text{M}}(A\cap P)$ is either 0 or $\dim_{\text{M}}(A)$.
\end{Conjecture}

\subsection{Outline of the proof}\label{Section1:Outline}

We present an outline of the proof of \cref{CorA}. Despite being a consequence of more general results, the proof contains the key ideas of this paper while avoiding some obstructions that arise when dealing with additional digit restrictions.

Consider an integer $g\geq 2$, an integer $a\in \N$ coprime with $g$, and a set of digits $\mathcal{D}=\{d_1,\cdots,d_t\}\subseteq \{0,\ldots,g-1\}$.  We aim to construct Markov chains that describe the 
 behaviour of\begin{equation}
     \dfrac{|\{n\in \mathcal{C}_{g,\mathcal{D}}:\ n\equiv b \Mod{a}\}\cap [0,N)|}{|\mathcal{C}_{g,\mathcal{D}}\cap [0,N)|}\label{Outline1}
 \end{equation}
 when $N\to \infty$. For a finite sequence of digits $w=w_0w_1\cdots w_m$ (we call it a word of length $m+1$), we denote
\[(w)_g:=w_mg^m+\cdots+w_1g+w_0.\]
We identify $\mathcal{D}^i$ with the set of words of length $i$, and we say $w\in \mathcal{D}^i$ represents $n\in \N_0$ if  $(w)_g=n$. Clearly,
\[\mathcal{C}_{g,\mathcal{D}}=\{(w)_g:\ w\in \mathcal{D}^n \text{ for $n\in \N_0$}\},\]
so equivalently, we can study the behavior of \eqref{Outline1} in terms of words. 
For every $i$, we define the measure $\mu_{i}$ on $\Z_a$ as 
 \[\mu_{i}(\{b\}):=\dfrac{|\{w\in \mathcal{D}^{i}:\ (w)_g\equiv b\Mod{a}\}|}{|\mathcal{D}|^{i}}.\]
 For simplicity, we will denote $\mu_i(b)=\mu_i(\{b\})$. This measure represents the distribution in  $\Z_a$ of elements in $\mathcal{C}_{g,\mathcal{D}}$ that can be expressed using $i$ digits of $\mathcal{D}$ in base $g$. 
 Our goal is to understand the evolution of $\mu_i$, and to do this, we will study how a word of length $i$ is extended\footnote{We say $x'=x'_0\ldots x'_k$ is an extension of $x=x_0\ldots x_m$ if $k>m$ and $x'_j=x_j$ for all $j\leq m$.}. 
For this purpose, we consider $p\in \N$ such that $g^p\equiv 1\Mod{a(g-1)}$ (whose existence is guaranteed by the condition $\gcd(g,a)=1$). Let us define the sets $E_{i}$ for  $i\in \N$ as
\begin{equation}E_{i}(b):=\Big\{w\in \mathcal{D}^{p}:\ g^i(w)_g\equiv b\Mod{a}\Big\},\label{extensions_outline}\end{equation}
and we also define define the matrix on $\Z_a\times \Z_a$ given by
$$M_{i}(b,b'):=\dfrac{|E_{i}(b'-b)|}{|\mathcal{D}|^p}.$$
It is easy to see that $M_i$ is a double stochastic matrix.\\
Let $n\in \N$, and let $x\in \mathcal{D}^{i+np}$  such that $(x)_g\equiv b\Mod{a}$. We can ask the following: For $b'\in \Z_a$, how many extensions  $x'\in \mathcal{D}^{i+(n+1)p}$  of $x$ satisfy $(x')_g\equiv b'\Mod{a}$? 
The answer is the cardinality of $E_{i+np}(b'-b)$, and we can  observe that $E_{i}(b-b')=E_{i+np}(b-b')$ since $g^{np}\equiv 1\mod a$. From this, we can obtain the relation $$\mu_{i+np}=\mu_{i}M_{i}^n.$$

If $X^i$ is a Markov chain with initial distribution $\mu_{i}$ and transition matrix $M_{i}$, it holds
\[\mathbb{P}\Big(X_n^i=b\Big)=\mu_{i+np}(b).\]

Moreover, if $X^i$ is irreducible and aperiodic, the Markov chain converges to its invariant distribution, hence
\begin{equation}
    \lim_{n\to \infty}\mathbb{P}\Big(X^{i}_n=b\Big)=\lim_{n\to \infty}\dfrac{|\{w\in \mathcal{D}^{i+np}:\ (w)_g\equiv b\mod a\}|}{|\mathcal{D}|^{i+np}}=\dfrac{1}{a}.\label{Limit_outline}
\end{equation}

If $0\notin \mathcal{D}$, there is a one-to-one correspondence between elements of $\mathcal{C}_{g,\mathcal{D}}$ and words. On the other hand, since $(w0\cdots 0)_g=(w)_g$,  every element of $\mathcal{C}_{g,\mathcal{D}}$ can be represented by infinitely many words when $0\in \mathcal{D}$. By making this distinction and under the assumption that $X^1,\cdots,X^{p}$ are irreducible and aperiodic, we can elucidate \eqref{Limit_outline} in terms of the set of integers with missing digits:
\begin{itemize}
    \item If $0\in \mathcal{D}$, 
    \[\lim_{N\to \infty}\dfrac{|\{n\in \mathcal{C}_{g,\mathcal{D}}:\ n\equiv b\Mod{a}\}\cap [0,g^N)|}{|\mathcal{C}_{g,\mathcal{D}}\cap [0,g^N)|}=\dfrac{1}{a}.\]
    \item If $0\notin \mathcal{D}$, 
    \[\lim_{N\to \infty}\dfrac{|\{n\in \mathcal{C}_{g,\mathcal{D}}:\ n\equiv b\Mod{a}\}\cap  [g^{N-1},g^N)|}{|\mathcal{C}_{g,\mathcal{D}}\cap [g^{N-1},g^N)|}=\dfrac{1}{a}.\]
\end{itemize}

For an arbitrary set of integers, this does not guarantee the existence of the limit when considering intervals of the form $[0,N)$. However, as we will prove in \cref{Section4}, in this case it happens. Therefore, we can show that
\begin{equation}\lim_{N\to \infty}\dfrac{|\{n\in \mathcal{C}_{g,\mathcal{D}}:\ n\equiv b\Mod{a}\}\cap [0,N)|}{|\mathcal{C}_{g,\mathcal{D}}\cap [0,N)|}=\dfrac{1}{a}\label{Outline3}\end{equation}
for every $b\in \Z_a$. Certainly,  we must still ensure that each Markov chain $X^i$ converges to the equilibrium. This property will hold if we choose $p$ as above and $\gcd(a,d_2-d_1,\cdots,d_t-d_1)=1$, where $d_1$ is the smallest digit of $\mathcal{D}$.

On the other hand, if $\gcd(a,d_2-d_1,\cdots,d_t-d_1)\neq 1$, it is possible to show that there exist $b\in \Z_a$ and $j\in \{1,\cdots,p\}$ such that \[\mathbb{P}\Big(X_n^j=b\Big)=0\] 
for every $n\in \N$. From this fact, we will conclude that \eqref{Outline3} does not hold for such a $b$.

Some difficulties arise when dealing with additional constraints on the digits, but we will extend the construction presented above to more general cases by leveraging ideas from symbolic dynamics. We are going to discuss this in detail in the following sections.

\vspace{1cm}

\textbf{Acknowledgements.} The author thanks his advisor, Joel Moreira, for introducing him to the problem and for his continuous support and guidance throughout this project. The author also expresses gratitude to Vaughn Climenhaga for providing useful references, and to the anonymous referee for their valuable feedback and comments.

\section{Preliminaries}
In this section we introduce some concepts of Markov chains, symbolic dynamics and integer fractals that will be used throughout this paper.

\subsection{Markov chains}
Let $\mathcal{S}$ be a finite set,  and let $M$ be a square matrix indexed by $\mathcal{S}\times \mathcal{S}$ with entries in $[0,1]$.  
The matrix $M$ is called \emph{stochastic} if the sum of each row is 1, i.e.,  for every $s\in \mathcal{S}$, $$\displaystyle \sum_{s'\in \mathcal{S}} M(s,s')=1.$$ 
Given a stochastic matrix $M$ and a probability measure $\mu$ on $\mathcal{S}$, a \emph{discrete-time Markov chain} with state space $\mathcal{S}$, transition matrix $M$ and initial distribution $\mu$ is a sequence of random variables $(X_n)_{n\in \N_0}$ on a probability space (with probability $\mathbb{P}$) such that $\mathbb{P}(X_0=s)=\mu(\{s\})=:\mu(s)$, and
\[\mathbb{P}(X_{n+1}=s_{n+1}|X_n=s_n,\cdots,X_0=s_0)=\mathbb{P}(X_1=s_{n+1}|X_0=s_n)=M(s_n,s_{n+1}) \ \]

 for every $n\in \N_0$ and $s_0,\cdots,s_{n+1}\in \mathcal{S}$. In this case, we have $$\mathbb{P}(X_n=s')=\mu \cdot M^n(\cdot,s')=\sum_{s\in\mathcal{S}}\mu(s)M^n(s,s').$$ For $s,s'\in \mathcal{S}$, we say that $s'$ is \emph{accessible from $s$} if $M^n(s,s')>0$ for some $n\in \N$. The matrix $M$ is called \emph{irreducible} if for every $s,s'\in \mathcal{S}$, there exists $n\in \N$ such that $M^n(s,s')>0$ (i.e., all states are accessible from each other).  We define the \emph{period of the state $s$} as \[Per(s):=\gcd \Big\{n\in \N:\ M^n(s,s)>0\Big\}.\]
If every state has period 1, we say $M$ is \emph{aperiodic}. When $M$ is irreducible, all states share the same period, so aperiodicity can be checked by finding one state with period 1. 

The following classical result on the convergence of Markov chains will be a key tool in this paper.
\begin{theorem}\cite[Theorem 4.9]{L_P}
    Let $X$ be a Markov chain with state space $\mathcal{S}$ (not necessarily finite),  transition matrix $M$ and initial distribution $\mu$. Suppose $M$ is irreducible and aperiodic, and there exists an invariant distribution $\lambda$ for $M$ (i.e., $\lambda=\lambda M$). Then, there exist $\rho\in(0,1)$ and $C>0$ such that
\[|\mathbb{P}(X_n=s)-\lambda(s)|\leq C\rho^n\]
for every $n\in \N$ and $s\in \mathcal{S}$.\label{thm:ConvergenceMC}
\end{theorem}

Notice that the invariant distribution for an irreducible and aperiodic transition matrix with a finite state space always exists (and it is unique) as a consequence of the Perron-Frobenius theorem. In addition, the \cref{thm:ConvergenceMC} is only dependent on the transition matrix, so that $\lambda$, $C$ and $\rho$ are independent of the initial distribution. A particularly interesting case occurs when the sum of each column of $M$ is also 1, in which case $M$ is called \emph{doubly stochastic}. In this case, the distribution $\lambda(s) = 1/|\mathcal{S}|$ is invariant for $M$. For further discussion on the convergence of Markov chains, see \cite[Section 1.8]{MarkovChain2} and \cite[Section 3]{MarkovChain1}.

\subsection{Elements of symbolic dynamics}\label{Section2:Symbolic}

Let $\mathcal{D}$ be a finite set (in this paper, we assume
$\mathcal{D}\subseteq\{0,\cdots,g-1\}$ for some $g\in \N$), and consider the \emph{left-shift} operator $\sigma$ over $\mathcal{D}^{\N_0}$, defined by
\[\sigma: (w_n)_{n\in \N_0}\mapsto (w_{n+1})_{n\in \N_0} .\]

Endowing $\mathcal{D}$ with the discrete topology and considering the product topology on $\mathcal{D}^{\N_0}$, $\sigma$ is a continuous function, and the topological system $(\mathcal{D}^{\N_0},\sigma)$ is called a \emph{full shift}.  If $\Sigma\subseteq \mathcal{D}^{\N_0}$ is closed and $\sigma(\Sigma)\subseteq \Sigma$, the dynamical system $(\Sigma,\sigma)$ is called a \emph{subshift}.
 Although a subshift is a set of infinite sequences of symbols, it can be fully characterized by the finite sequences that appear in $\Sigma$, referred to as \emph{words}. 
 
 We define the \emph{language} of $\Sigma$, denoted by $\mathcal{L}(\Sigma)$, as the set of all the words in $\Sigma$,
\[\mathcal{L}(\Sigma):=\Big\{w_0\cdots w_{n-1}: \ n\in \N_0, \  \exists x\in \Sigma \text{ s.t. } x_0=w_0, \ x_1=w_1,\ldots,\ x_{n-1}=w_{n-1} \Big\}.\]

When there is no confusion, we will simply denote the language by $\mathcal{L}$. For any word $w=w_0\cdots w_{n-1}\in \mathcal{L}(\Sigma)$, we say $w$ has \emph{length} $n$ (and we denote the length by $|w|$). The set $\mathcal{L}^n(\Sigma)$ denotes all words in the language of length $n$. Also, we will commonly use the notions of concatenation of words, and subwords. For $u=u_0\ldots u_{m-1}$ and $v=v_0\ldots v_{n-1}$, the \emph{concatenation} of $u$ with $v$ is written as $uv=u_0\ldots u_{m-1}v_0\ldots v_{n-1}$  and we say that $uv$ is an \emph{extension} of $u$. For a given word $u$ and $k\in \N$, we also denote by $u^k=u\cdots u$, the concatenation $k$-times of $u$. We say $w$ is a \emph{subword} of $x$ if the sequence of $x$ contains $w$. 

A subshift is a particular case of a topological dynamical system, and some classical notions of topological dynamics have a special characterization in the context of symbolic systems. 
\begin{definition} The topological dynamical system $(\Sigma,\sigma)$ is said to be \emph{transitive} (or \emph{irreducible}) if for every $u,v\in \mathcal{L}(\Sigma)$, there exists $w\in \mathcal{L}(\Sigma)$ such that $uwv\in \mathcal{L}(\Sigma).$ Also, the system is said to be \emph{mixing} if for every $u,v$ in the language, there exists $n\in \N$ such that for every $m\geq n$ there exists  $w\in \mathcal{L}^m(\Sigma)$ such that $uwv\in \mathcal{L}(\Sigma)$.
\end{definition}
On the other hand, the notion of \emph{topological entropy} is a measure of the complexity of a dynamical system. In symbolic systems, it quantifies the complexity of the language.
\begin{definition}
    The (topological) entropy of the subshift $\Sigma$ is defined as
    \[h(\Sigma):=\lim_{N\to\infty}\dfrac{\log |\mathcal{L}^N(\Sigma)|}{N}.\]\label{entropy}
\end{definition}

See Chapter 4 of \cite{Lind_Marcus_2021} for further discussion. Regarding the classification of subshifts, we will consider two important classes, \emph{shifts of finite type} and \emph{sofic shifts}. 
\begin{definition}
    Consider the full shift $X=\mathcal{D}^{\N_0}$, and let $\mathscr{F}\subseteq \mathcal{L}(X)$ be a finite set of \emph{forbidden words}. We define the subshift
\[X_{\mathscr{F}}:=\{x\in X:\ x \text{ does not contain any $w\in \mathscr{F}$ as subword} \}.\]
We say that $\Sigma\subseteq X$ is a \emph{shift of finite type} (SFT for short) if there exists a finite set $\mathscr{F}$ such that $\Sigma=X_\mathscr{F}$.\label{SFT}
\end{definition} If the longest forbidden word has length $M+1$, we say that $X_{\mathscr{F}}$ is a \emph{$M$-step} shift of finite type. In a $1$-step SFT, all forbidden words have length $2$, so it can be described by a matrix indexed by $\mathcal{D}\times \mathcal{D}$ with values in $\{0,1\}$. In this sense, $\Sigma$ is a $1$-step SFT if and only if there exists a matrix $T$ such that
\[\Sigma=\Big\{x\in \mathcal{D}^{\N_0}:\  T(x_i,x_{i+1})=1 \text{ for all }i\Big\}.\]

A more general way to construct subshifts is by reading infinite sequences in a \emph{labelled directed multi-graph}.  A labelled multi-graph is a triple $G=(V,E,\lambda)$, where $V$ is a finite set of nodes, $E$ is a set of edges and $\lambda:E\to \mathcal{A}$ is a labelling function for some finite alphabet $\mathcal{A}$. Each edge $e\in E$ has an associated starting node $s(e)\in V$ and a terminal node $t(e)\in V$. A sequence of edges $\pi=e_0\cdots e_{n}$ is called a \emph{path} in $G$ if $t(e_i)=s(e_{i+1})$ for every $i\in \{0,\cdots,n-1\}$, and we say that $\pi$ has length $|\pi|:=n+1$. Note that multiple edges may share the same starting and terminal nodes while having different labels.
\begin{definition}
    Let $G=(V,E,\lambda)$ be a labelled directed multi-graph. If $\mathcal{D}$ is the finite alphabet associated with the labelling function, we define the subshift $\Sigma_G\subset \mathcal{D}^{\mathbb{N}_0}$ as
    \[\Sigma_G=\left\{\left(\lambda(e_j)\right)_{j\in \mathbb{N}_0}\in \mathcal{D}^{\mathbb{N}_0}: \ s(e_{j+1})=
    t(e_j) \ \forall j\in \mathbb{N}_0 \right\}.\]\label{cover_def}
    
\end{definition}

Thus, $\Sigma_G$ are all infinite sequences that we can read in some infinite path of $G$. It is well known that $\Sigma_G$ is a subshift. 
\begin{definition}
    The subshift $\Sigma\subseteq \mathcal{D}^{\N_0}$ is said to be \emph{sofic} if there exists a labelled directed multi-graph such that $\Sigma=\Sigma_G$. The graph $G$ is called a \emph{cover} of $\Sigma$.
    \end{definition}
    Note that the cover of a sofic shift is not unique. Additionally, every shift of finite type is sofic.
    \begin{remark}
Let $\Sigma\subseteq \mathcal{D}^{\N_0}$ be a 1-step shift of finite type, and consider a matrix $T$ that represents $\Sigma$, i.e., $T(d,d')=1$ if and only if $dd'\in \mathcal{L}(\Sigma)$. We create a node $V(d)$ for every $d\in \mathcal{D}$, and place an edge from $V(d)$ to $V(d')$ labelled by $d'$ if and only if $T(d,d')=1$. The resulting labelled graph is a cover for $\Sigma$, and we denote it by $\mathcal{G}_T$.\label{graph:SFT}
    \end{remark}

To construct a cover for a sofic shift, we can rely on the notion of \emph{follower sets} \cite[Chapter 3]{Lind_Marcus_2021}.

\begin{definition}
   Let $\Sigma$ be a subshift of $\mathcal{D}^{\N_0}$. For any $w\in\mathcal{L}(\Sigma)$, the \emph{follower set of $w$}, denoted $F_\Sigma(w)$, is the set of all words that can extend $w$, i.e.,  
$$F_\Sigma(w):=\Big\{v\in \mathcal{L}(\Sigma):\ wv\in \mathcal{L}(\Sigma)\Big\}.$$  
We denote by  $\mathcal{F}_\Sigma:=\Big\{F_\Sigma(w): \ w\in \mathcal{L}(\Sigma)\Big\}$ the collection of all follower sets. When the underlying subshift $\Sigma$ is clear from the context, we will simply write $F(w)$ and $\mathcal{F}$.
\end{definition}
While each word has its own follower set, a folklore result states that $\Sigma$ is sofic if and only if there are finitely many distinct follower sets \cite[Theorem 3.2.10]{Lind_Marcus_2021}, i.e., $|\mathcal{F}_\Sigma|<\infty$. If $d\in F(u)=F(v)$, it follows that $F(ud)=F(vd)$. Using these ideas, a cover can be constructed for $\Sigma$.
\begin{definition}
    Let $\Sigma\subseteq \mathcal{D}^{\N_0}$ be a sofic shift. Consider the finite set of nodes $V=\mathcal{F}_\Sigma$, and place an edge labelled by $d\in \mathcal{D}$ from $F\in \mathcal{F}_\Sigma$ to $F'\in \mathcal{F}_\Sigma$ if $d\in F$, and $F(wd)=F'$ for any word $w$ such that $F=F(w).$ The labelled multi-graph generated is called the \emph{follower set graph of $\Sigma$}.
\end{definition}
The Follower set graph is a cover of $\Sigma$ \cite[Proposition 3.2.9]{Lind_Marcus_2021}. Notice that the follower set graph is \emph{right-resolving}, meaning that for any node, all the edges going out have different labels.  When the sofic shift is transitive, we would expect the cover to be an irreducible graph (i.e., there is a path between each pair of vertices), but this is not always true for the follower set graph. To address this problem, a subgraph can be extracted, which remains a cover and is irreducible.
\begin{definition}
    A word $w\in \mathcal{L}(\Sigma)$ is said to be \emph{synchronizing} if for every $u,v\in \mathcal{L}(\Sigma)$, the condition that both $uv$ and $wv$ are in the language implies that $uwv\in \mathcal{L}(\Sigma)$.
\end{definition}
While not every subshift has synchronizing words, it is well-known that transitive sofic shifts do. Moreover, every word can be extended to a synchronizing, and any extension of a synchronizing word is also synchronizing. In the case of a $M$-step SFT, every word with length at least  $M+1$ is synchronizing 

An important result due to Fischer \cite{Fischer1,Fischer2} establishes the existence and uniqueness (up to isomorphism) of a minimal (in terms of  the number of nodes), right-resolving and irreducible cover for a transitive sofic shift.
This unique representation is known as the \emph{Fischer cover}, and plays an important role in this paper. The following outlines the construction of this cover.

\textbf{Construction of the Fischer cover}: Let $\Sigma\subseteq \mathcal{D}^{\N_0}$ be a transitive sofic shift.
For each follower set $F\in \mathcal{F}_\Sigma$, we create a node associated to $F$ if there exists a synchronizing word $w$ such that $F(w)=F$; in other words, we consider the set of nodes
\[\mathcal{V}_\Sigma:=\{F(w):\ w\in \mathcal{L}(\Sigma) \text{ is synchronizing}\}\subseteq \mathcal{F}_\Sigma.\]
Let $F,F'\in \mathcal{V}_\Sigma$ and $d\in \mathcal{D}$. Then, there exists $w\in \mathcal{L}(\Sigma)$ such that $F=F(w)$. We create an edge with starting node $F$, terminal node $F'$  and labelled by $d$ if and only if $F'=F(wd)$. The resulting graph, denoted by $\mathcal{G}_\Sigma=(\mathcal{V}_\Sigma,\mathcal{E}_\Sigma,\lambda_\Sigma)$, is the Fischer cover for $\Sigma$ (see Chapter 3 in \cite{Lind_Marcus_2021} and Lemma 6.11 in \cite{Spandl} for more details). \\

Thus, the Fischer cover is a subgraph of the follower set graph, constructed by considering only the nodes corresponding to follower sets of synchronizing words. Notice that for any synchronizing word $w \in \mathcal{L}(\Sigma)$, every path in $\mathcal{G}_\Sigma$ labelled by $w$ must end at the node associated with $F(w)$. In the case $\Sigma$ is a mixing sofic subshift, it is possible to find $n\in \N$ such that for every pair of nodes $F,F'\in \mathcal{V}$ and $m\geq n$, there exists a path of length $m$ from $F$ to $F'$. In such a case, for every pair of words $u$ and $v$ in the language, there exists $w\in \mathcal{L}^m(\Sigma)$ such that $uwv\in \mathcal{L}(\Sigma)$

\subsection{Fractal sets of integers}
In \cite{GMR}, Glasscock, Moreira and Richter introduced the notion of multiplicativaly invariant sets of integers (recall \cref{def:invariant_sets}), and studied them from a fractal perspective. Here, we focus on the notion of \emph{mass dimension}, which is analogous to the box dimension in fractal geometry but adapted to the setting of integers.
\begin{definition} Let $A\subseteq \N_0$ be non-empty. We define the \emph{lower mass dimension} and the \emph{upper mass dimension} of $A$ as:
\begin{align*}
  &\underline{\dim}_\text{M}(A):=\liminf_{N\to \infty} \dfrac{\log |A\cap [0,N) |}{\log N}=\sup \left\{\gamma\geq 0: \ \liminf_{N\to \infty}\dfrac{|A\cap [0,N)|}{N^\gamma}>0 \right\},\\
  &\overline{\dim}_\text{M}(A):=\limsup_{N\to \infty} \dfrac{\log |A\cap [0,N) |}{\log N}=\sup \left\{\gamma\geq 0: \ \limsup_{N\to \infty}\dfrac{|A\cap [0,N)|}{N^\gamma}>0 \right\}.
\end{align*}
If $\underline{\dim}_\text{M}(A)=\overline{\dim}_\text{M}(A)$, we denote this value by $\dim_\text{M}(A)$ and say that the \emph{mass dimension} of $A$ exists.\label{Mass_dimension}    
\end{definition}

 From \cite[Proposition 3.6]{GMR}, the mass dimension exists for any multiplicatively invariant set of integers. A natural way to generate $\times g$-invariant sets is by considering subshifts of $X=\{0,\ldots,g-1\}^{\N_0}$.
\begin{definition} The \emph{g-language set} asociated with a subshift $\Sigma \subseteq \{0,\ldots,g-1\}^{\N_0}$ is the set $A_\Sigma\subseteq \N_0$ defined by
\begin{equation}
A_\Sigma:=\Big\{(w)_g:=w_0+w_1g+\cdots+w_{n-1}g^{n-1}\ : \ w=w_0\cdots w_{n-1}\in \mathcal{L}(\Sigma) \Big\}.    
\end{equation}\label{A_Sigma}
\end{definition}
In general, the expansion of a number in a certain base is read from left to right, where the leftmost digit represents the most significant digit. For a given word $w$, in this paper we will consider the rightmost digit to be the most significant digit of the integer $(w)_g$, this in order to be consistent with the theory of susbshifts. However, when it comes to applications this does not make a major distinction.

From the definition of $\times g$-invariant sets and the shift-invariance of $\Sigma$, it is evident that $A_\Sigma$ is $\times g$-invariant. It is also possible to see that every multiplicatively invariant set can be written as $ A_\Sigma$ for some subshift $\Sigma$, and the mass dimension is related to the entropy of the subshift.  This is summarized in the following proposition (see  \cite[Proposition 3.10]{GMR}).

\begin{proposition} The $g$-language set $A_\Sigma\subseteq \N_0$ corresponding to any non-empty subshift $\Sigma\subseteq \{0,\cdots,g-1\}^{\N_0}$ is a $\times g$-invariant set, and the mass dimension of $A_\Sigma$ is equal to the normalized topological entropy of the subshift $(\Sigma,\sigma)$, i.e.,
\begin{equation}
   \dim_\text{M}(A_\Sigma)=h(\Sigma)/\log(g).\label{entropy2}
\end{equation}
    Moreover, for any $\times g$-invariant set $A\subseteq \N_0$, there exists a subshift $\Sigma\subseteq \{0,\cdots,g-1\}^{\N_0}$ such that $A=A_\Sigma$. In such a case, we say that $A$ is represented by $\Sigma$.\label{languageset}
\end{proposition}
\begin{corollary} Let $g\geq 2$ be an integer and let $\mathcal{D}\subseteq \{0,\cdots,g-1\}$. If we consider the full shift $\Sigma=\mathcal{D}^{\N_0}$, $\mathcal{C}_{g,\mathcal{D}}=A_\Sigma$. In addition, \[\dim_{\text{M}}(\mathcal{C}_{g,\mathcal{D}})=\log(|\mathcal{D}|)/\log(g).\]
    \label{dim_missingdigitsets}
\end{corollary}
Finally, the following result allows us to simplify the calculation of the dimension to simpler intervals.
\begin{lemma}\cite[Lemma 3.2. (IV)]{GMR}.\label{dimM2}
    For $r\in \mathbb{N}$, $r\geq 2$, \begin{equation}
    \overline{\dim}_\text{M}(A)=\limsup_{N\to \infty}\dfrac{\log |A\cap [0,r^N) |}{N\log r}.\label{computation}\end{equation}
Analogously, we can compute $\underline{dim}_\text{M}(A)$ by using $\liminf$ instead of $\limsup$ in (\ref{computation}).\label{lemma:DIM}
\end{lemma}

\section{Construction of Markov chains}\label{Section3}

In \cref{Section1:Outline},
 we provided an overview of the construction of Markov chains in the context of integers with missing digits and their distribution in residue classes, but now we extend this approach to a more general framework. Let $f_1,\ldots,f_r$ be $g$-additive functions, $a_1,\ldots,a_r\in \N$ and $A\subseteq \N_0$ a $\times g$-invariant set. We aim to find conditions to ensure
 \begin{equation}\lim_{N\to \infty}\dfrac{|\{n\in A: f_i(n)\equiv b_i\Mod{a_i} \text{ for $i=1,\ldots,r$}\}\cap [0,N)|}{|A\cap [0,N)|}=\dfrac{1}{\Pi_{i=1}^r a_i} \label{limit_distribution}\end{equation}
  for any $(b_1,\ldots,b_r)\in \Z_{a_1}\times \cdots\times \Z_{a_r}.$ 

\begin{notation} Throughout this paper, we denote $\mathbf{a} := (a_1, \ldots, a_r)$, $|\mathbf{a}|:=|\Z_{a_1}\times \ldots\times \Z_{a_r}|=\Pi_{i=1}^r a_i$, and   $\mathbf{f} := (f_1, \ldots, f_r)$. For any $\mathbf{b} \in \mathbb{Z}_{\mathbf{a}} := \mathbb{Z}_{a_1} \times \cdots \times \mathbb{Z}_{a_r}$ and $n \in \mathbb{N}_0$, we will write $\mathbf{f}(n) \equiv \mathbf{b}\Mod{\mathbf{a}}$ if $f_j(n) \equiv b_j\Mod{a_j}$ for every $j \in \{1, \ldots, r\}$. \end{notation}

From \cref{languageset}, there exists a subshift $\Sigma\subseteq \{0,\cdots,g-1\}^{\N_0}$ such that $$A=A_\Sigma=\Big\{(w)_g:\ w\in \mathcal{L}(\Sigma)\Big\}.$$
 Therefore, we will focus on studying the distribution of words of a certain length and then extend the result to the set of integers $A$. Similarly to \cref{Section1:Outline}, we need to find an analogue to the value $p$ for periodicity and properly understand the concatenation of words in the language.

\begin{definition}
   Let $f_1,\cdots,f_r$ be $g$-additive functions and $\mathbf{a}\in \N^{r}$. We say $\mathbf{f}$ is \emph{eventually periodic with respect to $\mathbf{a}$} if there exists $\ell\in \N_0$ and $p\in \N$ such that \[\mathbf{f}(dg^{i+np})\equiv \mathbf{f}(dg^i)\Mod{\mathbf{a}}\]
for every $d\in \{0,\ldots,g-1\}$, $n\in \N$ and $i\geq \ell$. Any pair $(p,\ell)$ satisfying this condition is called an \emph{eventual period}.\label{def:periodic}
    \end{definition}

    When $\ell'\geq \ell$ and $p\mid p'$, the pair $(p',\ell')$ is also an eventual period. Notice that $\mathbf{f}$ is eventually periodic if each $f_i$ is eventually periodic.  Clearly, the function sum-of-digits in base $g$ satisfies this condition with $p=1$ and $\ell=0$. Also, the identity is always $g$-invariant and eventually periodic.

\begin{proposition}
 For any $a,g\in \N$, there exist $p\in \N$ and $\ell\leq a-1$ such that 
 \[dg^{i}\equiv dg^{i+np}\Mod{a}\]
 for every $d\in \{0,\cdots,g-1\}$, $n\in \N$ and $i\geq \ell$.\label{prop:periodicity} 
\end{proposition}
 \begin{proof}
    At least two elements of $\{g^j\Mod{a}:\ j\in\{0,\cdots,a\}\}$ agree. Then, there exist $j_0<j_1\leq a$ such that $g^{j_0}\equiv g^{j_1}\Mod{a}$. We conclude by considering $\ell=j_0$ and $p=j_1-j_0$.
\end{proof}
If $A$ is a missing digits set, then $A=A_\Sigma$ for $\Sigma=\mathcal{D}^{\N_0}$, where $\mathcal{D}$ is the respective set of digits. In this context, the concatenation of two words is in $\mathcal{L}(\Sigma)$ if and only if both words are in $\mathcal{L}(\Sigma)$.  Unlike full shifts, this property generally does not hold. Nevertheless, it is possible to adapt the construction when concatenation is governed by a suitable rule, as occurs in transitive sofic shifts through the notions of follower sets and the Fischer cover.

As discussed in \cref{Section2:Symbolic}, we can identify the nodes of the Fischer cover $\mathcal{G}_{\Sigma}=(\mathcal{V}_\Sigma,\mathcal{E}_{\Sigma},\lambda_{\Sigma})$ with the collection of follower sets associated to synchronizing words of the language. Thus, we can identify $$\mathcal{V}_{\Sigma}=\{F_1,F_2,\ldots,F_q\}=\{F(w):\ w\in \mathcal{L}(\Sigma)\text{ is a synchronizing word}\}.$$
The concatenation of words can be understood from $\mathcal{G}_\Sigma$. For $x,y\in \mathcal{L}(\Sigma)$, $$xy\in \mathcal{L}(\Sigma)\iff y\in F(x) \iff \text{There exists a path of $\mathcal{G}_\Sigma$ labelled by $xy$.}$$ When $F(x)\in \mathcal{V}_{\Sigma}$, it is also equivalent to the existence of a path $\pi$ in the Fischer cover starting at the node associated to $F(x)$, and it is labelled by $y$. 
\begin{notation}  For a word $w\in \mathcal{L}(\Sigma)$ and  $F,F'\in \mathcal{V}_\Sigma$, we write $F\overset{w}{\to} F'$ if there is a path $\pi$ in $\mathcal{G}_\Sigma$ from $F$ to $F'$ labelled by $w$. Since $\mathcal{G}_\Sigma$ is right-resolving, any other path from $F$ to $F'$ has a different label.
  \end{notation}
  Notice that  $F\overset{w}{\to} F'$ if and only if $w\in F$, and for every word $u$ such that $F=F(u)$, $F(uw)=F'$.

 \begin{remark}
For the rest of this paper,  we will consider a base  $g\geq 2$,  $r\in \N$, $\mathbf{a}\in \N^{r}$, a collection of  $g$-additive functions $\mathbf{f}=(f_1,\ldots,f_r)$ eventually periodic with respect to $\mathbf{a}$, and an eventual period $(p,\ell)$. In addition, we consider a transitive sofic subshift $\Sigma\subseteq \{0,\ldots,g-1\}^{\N_0}$, and $\mathcal{G}_{\Sigma}=(\mathcal{V}_\Sigma,\mathcal{E}_{\Sigma},\lambda_{\Sigma})$ its Fischer cover (if there is no confusion, we omit the subscript). For technical reasons, we will assume that $\ell$ is large enough such that there exists some synchronizing word of length equal to $\ell$.\label{Convention}
\end{remark}

\subsection{Transition matrix}

In order to define the transition matrix, we will group the elements of $\mathcal{L}(\Sigma)$ according to how they extend other words. Recalling that $(w_0\ldots w_n)_g:=w_0+w_1g+\cdots+w_ng^n$, we introduce the following subsets of $\mathcal{L}^p(\Sigma)$.

  \begin{definition}
Let $i\in \N$. For any $\mathbf{b}\in \Z_{\mathbf{a}}=\Z_{a_1}\times \cdots\times \Z_{a_r}$ and $F,F'\in \mathcal{V}$, we define the set
       $$E_{i}(\mathbf{b},F,F'):=\Big\{w\in \mathcal{L}^p(\Sigma): \mathbf{f}(g^i(w)_g)\equiv \mathbf{b}\Mod{\mathbf{a}},\  F\overset{w}{\to}F'\Big\}.$$\label{def:sets_E}
\end{definition}
Notice that $E_i$ depends on $g$, $\mathbf{f}$, $\mathbf{a}$, $\Sigma$ and $p$, but as in \cref{Convention} we set these elements, we will not make the dependency explicit in order to lighten the notation. The following lemma clarifies the motivation for defining these sets.

\begin{lemma}  Let $x\in \mathcal{L}^i(\Sigma)$ and $w\in E_i(\mathbf{b},F(x),F')$. Then, $xw\in \mathcal{L}(\Sigma)$, $F(xw)=F'$ and $\mathbf{f}\left((xw)_g\right)\equiv \mathbf{f}((x)_g)+\mathbf{b}\Mod{\mathbf{a}}.$ If $i\geq \ell$, it also holds $E_{i}=E_{i+np}$ for all $n\in \N$.\label{lemma:sets_E}
\end{lemma}
\begin{proof}
    Since $F(x)\overset{w}{\to}F'$,  $xw\in \mathcal{L}(\Sigma)$ and $F(xw)=F'$. Also, $$\mathbf{f}((xw)_g)=\mathbf{f}\left(\sum_{j=0}^{i-1}x_jg^j+\sum_{j=i}^{i+p-1}w_jg^j\right)=\mathbf{f}((x)_g)+\mathbf{f}(g^i(w)_g)\equiv \mathbf{f}((x)_g)+\mathbf{b}\Mod{\mathbf{a}}.$$
    If $i\geq \ell$, $\mathbf{f}(g^i(y)_g)=\mathbf{f}(g^{i+np}(y)_g)$ for every $n\in \N$ and $y\in \mathcal{L}^p(\Sigma)$, concluding that $E_i=E_{i+np}.$
\end{proof}

Despite that the collection of sets $E_{i}$ is not a partition of $\mathcal{L}^p$ (since a word might be read along different paths in $\mathcal{G}$), we still can understand it as a partition of the paths of length $p$ in $\mathcal{G}$. 
\begin{proposition}  Let $i\in\N$.  For every $F \in \mathcal{V}$,
 \begin{equation}\ \sum_{(\mathbf{b},F')\in \Z_{\mathbf{a}}\times \mathcal{V}}|E_{i}(\mathbf{b},F,F')|=|\{\pi=(\pi_0\cdots\pi_{p-1}) \text{ is a path of $\mathcal{G}$ and }s(\pi)=F\}|.\label{eq:source}\end{equation}
Similarly, for every $F'\in \mathcal{V}$, 
     \begin{equation} \sum_{(\mathbf{b},F)\in \Z_{\mathbf{a}}\times \mathcal{V}}|E_{i}(\mathbf{b},F,F')|=|\{\pi=(\pi_0\cdots\pi_{p-1}) \text{ is a path of $\mathcal{G}$ and }t(\pi)=F'\}|.\label{eq:target}\end{equation}
     \label{prop:E_size}
\end{proposition}
\begin{proof}
    Let $F,F'\in \mathcal{V}$. Notice that the sets $E_{i}(\mathbf{b},F,F')$ and $E_{i}(\mathbf{b}',F,F')$ are disjoint when $\mathbf{b}\neq \mathbf{b}'\in \Z_{\mathbf{a}}$. For any $\mathbf{b}\in \Z_{\mathbf{a}}$, if $w\in E_{i}(\mathbf{b},F,F')$, there exists a path $\pi$ in $\mathcal{G}$ of length $p$ such that $\lambda(\pi)=w$ from $F$ to $F'$. On the other hand, any path $\pi$ of length $p$ from $F$ to $F'$ represents a word $x:=\lambda(\pi)$ of $E_{i}(\mathbf{b},F,F')$, where $\mathbf{b}:=\mathbf{f}(g^i(x)_g)$. Since $F$ and $F'$ are fixed, each of these paths represents a different word. Then,
    \begin{align*}
   |\{\pi=(\pi_0\cdots\pi_{p-1}) \text{ is a path in $\mathcal{G}$ from $F$ to $F'$}|&=\left|\bigcup_{\mathbf{b}\in \Z_{\mathbf{a}}}E_{i}(\mathbf{b},F,F')\right|\\
    &=\sum_{\mathbf{b}\in \Z_{\mathbf{a}}}  |E_{i}(\mathbf{b},F,F')|.
    \end{align*}
    By summing over $F'$  the equality \eqref{eq:source} is obtained, and similarly, \eqref{eq:target} follows by summing over $F$.
\end{proof}

Usually, \eqref{eq:source} and \eqref{eq:target} do not agree. For this work, we require an extra regularity of $\mathcal{G}$ to ensure that the number of paths of length $p$ starting at some $F\in \mathcal{V}$ is equal to the number of paths of length $p$ with terminal node $F'\in \mathcal{V}$. From graph theory, a directed multi-graph is said to be \emph{$k$-regular} if each vertex has $k$ edges going in and $k$ edges going out (two edges are different if they differ in the starting node, terminal node or label).  We can pose this definition in subshifts.
\begin{definition}
 We say a transitive sofic shift is \emph{$k$-regular} if the graph representation given by the Fischer cover is a $k$-regular graph. We say that a transitive sofic shift is regular if it is $k$-regular for some $k\geq 2$. \label{def:regular}
\end{definition}

If a graph is $1$-regular, it must represent a full shift in one symbol. Thus, we are interested in the case $k\geq 2$, so we include it directly in the definition. Trivially, any full shift satisfies this condition, with $k$ being the number of symbols. When $\Sigma$ is $k$-regular and $F \in \mathcal{V}$, there are $k^p$ paths of length $p$ starting from $F$, and $k^p$ paths of length $p$ ending at $F$. From this idea, we can define a transition matrix for the Markov chains.

\begin{definition}
In the framework of \cref{Convention}, and assuming $\Sigma$ is $k$-regular, we define the state space $\mathcal{S}:=\Z_{\mathbf{a}}\times \mathcal{V}$ and the matrix $M_{i}$ indexed by $\mathcal{S}\times \mathcal{S}$ as \[M_{i}\Big((\mathbf{b},F),(\mathbf{b}',F')\Big):=k^{-p}\cdot |E_{i}(\mathbf{b}'-\mathbf{b},F,F')|\]
  for any $(\mathbf{b},F), (\mathbf{b}',F')\in \mathcal{S}.$
 \label{def:matrix}
\end{definition}

Clearly, the matrix $M_i$ depends on several elements, including the choice of $p$, but we will not make the dependence explicit. The regularity condition allows us to show that this matrix $M_i$ is suitable for defining a Markov chain.
\begin{proposition} For any  $i\in \N$, the matrix $M_{i}$ is a doubly stochastic matrix. \label{prop:stochastic}
 \end{proposition}
 \begin{proof}
     Let $(\mathbf{b},F)\in \mathcal{S}$. From \eqref{eq:source},
     \begin{align*}
        \sum_{s\in \mathcal{S}} M_{i}\Big((\mathbf{b},F),s\Big)&=k^{-p}\sum_{(\mathbf{b}',F')\in \mathcal{S}}|E_{i}(\mathbf{b}'-\mathbf{b},F,F')| \\
        &=k^{-p} \sum_{(\hat{b},F')\in \mathcal{S}} |E_{i}(\hat{\mathbf{b}},F,F')|\\
        &=k^{-p}|\{\pi=(\pi_0\pi_1\cdots\pi_{p-1}) \text{ is a path in $\mathcal{G}$ and }s(\pi)=F\}|.\end{align*}
        We conclude $\displaystyle \sum_{s\in \mathcal{S}} M_{i}\Big((\mathbf{b},F),s\Big)=1$ from $k$-regularity.
        Analogously, from \eqref{eq:target}  we obtain $\displaystyle\sum_{s\in \mathcal{S}} M_{i}\Big(s,(\mathbf{b},F)\Big)=1.$
        \end{proof}

  Note that the regularity is important to ensure that the normalization factor to make the matrix stochastic is the same for all rows. Without this, while a stochastic matrix could still be defined, the link between $M_i$ and the distribution of integers becomes unclear, and our method no longer seems to work.

\subsection{Initial distribution}
In our attempt to study $A_\Sigma$, the idea is to extend all the $\ell$-long words that have a follower set in $\mathcal{V}$, as we can fully describe those extensions using the Fischer cover. However, for some $w\in \mathcal{L}^{\ell}$, it might be that $F(w)\notin \mathcal{V}$. This leads to the following definition.
 \begin{definition} Let $\Sigma$ be a transitive sofic shift. We define the \emph{$\ell$-restricted language of $\Sigma$}  as all words of $\mathcal{L}(\Sigma)$ that are extension of a word of length $\ell$ with follower set in $\mathcal{V}$,
\[\mathcal{L}_{\mathcal{V},\ell}(\Sigma):=\Big\{w\in \mathcal{L}(\Sigma): |w|\geq \ell \text{ and } F(w_0\cdots w_{\ell-1})\in \mathcal{V}\Big\}.\]
  For every $i\geq \ell$, we also define $\mathcal{L}_{\mathcal{V},\ell}^i(\Sigma):=\{w\in \mathcal{L}_{\mathcal{V},\ell}(\Sigma):\ |w|=i\}$.
\label{def:partial_language}
\end{definition}

The reason for considering $\ell$ large enough in \cref{Convention} is to ensure that this set is not empty. If $\Sigma$ is a $M$-step shift of finite type, any word $w$ such that $|w|\geq M+1$ is synchronizing, therefore $\mathcal{L}_{\mathcal{V},\ell}^{i}(\Sigma)=\mathcal{L}^i(\Sigma)$ for all $i\geq M+1$. Under the condition of regularity, it is easy to describe the cardinality of $\mathcal{L}_{\mathcal{V},\ell}^i$. For the upcoming results, if $w = w_0 \cdots w_n$ is a word, we will identify it with a pair of words $(x, y)$ such that $w = xy$.
\begin{proposition}
       Let $\Sigma$ be a transitive sofic shift and $k$-regular. Then,
       \[|\mathcal{L}^i_{\mathcal{V},\ell}(\Sigma)|=|\mathcal{L}^{\ell}_{\mathcal{V},\ell}(\Sigma)|\cdot k^{i-\ell}\]\label{propCardinality}
       for every $i\geq \ell$.\label{prop:cardinality}
\end{proposition}
\begin{proof}
    Notice  $x\in \mathcal{L}^i_{\mathcal{V},\ell}$ if and only if $F:=F(x_0\cdots x_{\ell-1})\in \mathcal{V}$  and there is a path $\pi$ in the Fischer cover of length $i-\ell$  from the vertex associated to $F$, and $\pi$ is labelled by $x_\ell\cdots x_{i-1}$. Using the identification, 
\begin{equation}
\mathcal{L}^i_{\mathcal{V},\ell}=\bigcup_{F\in \mathcal{V}}\{w\in \mathcal{L}^\ell_{\mathcal{V},\ell}: F(w)=F\}\times \{\lambda(\pi):\ \pi=\pi_0\cdots \pi_{i-\ell-1}\in \mathcal{G},\ s(\pi)=F \}.\label{proof:UnionLanguage} 
\end{equation}

Since the Fischer cover is right-resolving, all paths in $\{ \pi=\pi_0\cdots \pi_{i-\ell-1}\in \mathcal{G},\ s(\pi)=F \}$ have different labels. Then, the  conclusion follows by using the $k$-regularity in \eqref{proof:UnionLanguage}.
\end{proof}

We introduce a probability measure on $\mathcal{S}$.
\begin{definition} For every $i\geq \ell$, the distribution $\mu_{i}$
   on $\mathcal{S}=\Z_{\mathbf{a}}\times \mathcal{V}$ is defined as
\[\mu_{i}(\{(\mathbf{b},F)\}):=\dfrac{|\{w\in \mathcal{L}_{\mathcal{V},\ell}^{i}(\Sigma):\ \mathbf{f}((w)_g)\equiv \mathbf{b}\Mod{\mathbf{a}},\ F(w)=F\}|}{|\mathcal{L}_{\mathcal{V},\ell}^{i}(\Sigma)|}\]
for any $(\mathbf{b},F)\in \mathcal{S}$.
\label{def:distribution}
\end{definition}

We will denote $\mu_{i}(\mathbf{b},F):=\mu_{i}(\{(\mathbf{b},F)\})$. By understanding the powers of the matrix $M_i$, we can describe the evolution of the measure $\mu_{i+np}$ with respect to $n\in \N$.
\begin{proposition} For any $i,n\in \N$ and $s=(\mathbf{b}, F),s'=(\mathbf{b}',F')\in \mathcal{S}$,  \begin{equation}
    M_{i}^n(s,s')=k^{-pn}|\{w\in \mathcal{L}^{pn}(\Sigma):\ \mathbf{f}(g^{i}(w)_g)\equiv \textbf{b}'-\textbf{b}\Mod{\textbf{a}}, \ F\overset{w}{\to} F'\}|.\label{eq:keymarkov_1}
\end{equation}
    Moreover, for every $i\geq \ell$ and $n\in \N$,
\begin{equation}
    \mu_{i+np}=\mu_{i} M_{i}^n.\label{eq:keymarkov_2}
\end{equation}\label{prop:keymarkov}
\end{proposition}
\begin{proof}
 We proceed by induction to show \eqref{eq:keymarkov_1}. Let  $s=(\mathbf{b}, F), s'=(\mathbf{b}',F')\in \mathcal{S}$. The case $n=1$ holds by definition. Suppose that the result holds for some $n\in \N$, so we have to prove the result for $n+1$. Since
 \[M_i^{n+1}(s,s')=\sum_{\Tilde{s}\in \mathcal{S}}M_i^{n}(s,\Tilde{s})M_i(\Tilde{s},s'),\]
 from the induction hypotheses, it will be sufficient to prove that 
 \begin{equation}
     \begin{split}
        &\{w\in \mathcal{L}^{p(n+1)}(\Sigma):\ \mathbf{f}(g^{i}(w)_g)\equiv \textbf{b}'-\textbf{b}\Mod{\textbf{a}}, \ F\overset{w}{\to} F'\}\\
        &=\bigcup_{(\Tilde{\mathbf{b}},\Tilde{F})\in \mathcal{S}} \Big( \{x\in \mathcal{L}^{pn}(\Sigma):\ \mathbf{f}(g^{i}(x)_g)\equiv \Tilde{\textbf{b}}-\textbf{b}\Mod{\textbf{a}}, \ F\overset{x}{\to} \Tilde{F}\}\\
        &\hspace{2cm}\times\{y\in \mathcal{L}^{p}(\Sigma):\ \mathbf{f}(g^{i}(w)_g)\equiv \textbf{b}'-\Tilde{\textbf{b}}\Mod{\textbf{a}}, \ \Tilde{F}\overset{y}{\to} F'\}\Big).
     \end{split}\label{proof_keymarkov:1}
 \end{equation}
 Let $w\in \mathcal{L}^{p(n+1)}$ in the left hand side of \eqref{proof_keymarkov:1}, and consider the decomposition $w=(x,y)$, where $x=w_0\ldots w_{pn-1}$ and $y=w_{pn}\ldots w_{n(p+1)-1}$. Since $\mathbf{f}(g^i(w)_g)=\mathbf{f}(g^i(x)_g)+\mathbf{f}(g^{i+pn}(y)_g),$ the eventual periodicity leads to
 \[\mathbf{f}(g^i(w)_g)\equiv_{\mathbf{a}}\mathbf{b}'-\mathbf{b} \text{ if and only if }\mathbf{f}(g^i(x)_g)\equiv_{\mathbf{a}}\Tilde{\mathbf{b}}-\mathbf{b}\text{ and }\mathbf{f}(g^i(y)_g)\equiv_{\mathbf{a}}\mathbf{b}'-\Tilde{\mathbf{b}}, \]
 where $\Tilde{\mathbf{b}}:=\mathbf{f}(g^i(x)_g)+\mathbf{b}\Mod{\mathbf{a}}$. In addition, $F\overset{w}{\to}F'$ is equivalent to the existence of a path $\pi$ in $\mathcal{G}$ of length $(n+1)p$ from $F$ to $F'$ labelled by $w$. If we denote by $\Tilde{F}$ the terminal node for the subpath $\pi_0\ldots \pi_{np-1}$, $F\overset{w}{\to}F' \text{ if and only if }F\overset{x}{\to}\Tilde{F}\text{ and }\Tilde{F}\overset{y}{\to}F'$. Thus, we conclude \eqref{proof_keymarkov:1}. On the other hand,  $$\Big(\mu_iM^n_i\Big)(s)=\sum_{s'\in \mathcal{S}}\mu_i(s')M_i^n(s',s).$$ By \cref{prop:cardinality}, $k^{pn}|\mathcal{L}^i_{\mathcal{V},\ell}(\Sigma)|=|\mathcal{L}^{i+np}_{\mathcal{V},\ell}(\Sigma)|$. Hence, using \eqref{eq:keymarkov_1}, we can conclude \eqref{eq:keymarkov_2} by showing that
\begin{equation}
    \begin{split}
        &\{w\in \mathcal{L}_{\mathcal{V},\ell}^{i+(n+1)p}(\Sigma):\ \mathbf{f}((w)_g)\equiv \mathbf{b}\Mod{\mathbf{a}},\ F(w)=F\}\\
        &=\bigcup_{s'\in \mathcal{S}}\Big(\{x\in \mathcal{L}_{\mathcal{V},\ell}^{i}(\Sigma):\ \mathbf{f}((w)_g)\equiv \mathbf{b}'\Mod{\mathbf{a}},\ F(w)=F'\}\\
        &\hspace{1cm}\times \{y\in \mathcal{L}^{pn}(\Sigma):\ \mathbf{f}(g^{i}(y)_g)\equiv \textbf{b}-\textbf{b}'\Mod{\textbf{a}}, \ F'\overset{y}{\to} F\}\Big).
    \end{split}
\end{equation}
If $w\in \mathcal{L}_{\mathcal{V},\ell}^{i+(n+1)p}(\Sigma)$, it is straightforward that $w_0\ldots w_{i-1}\in \mathcal{L}^i_{\mathcal{V},\ell}(\Sigma)$. The rest of the proof follows by replicating the argument used to prove \eqref{proof_keymarkov:1}. 
 \end{proof}

From the previous results, we can characterize the distribution using Markov chains.

 \begin{definition} Let $g, \mathbf{a}, \mathbf{f}, p, \ell$ and $\Sigma$ as defined in \cref{Convention}, and assume that the subshift $\Sigma$ is $k$-regular (where $k\geq 2$). For $i\geq \ell$, we denote by $X^i=(X^{i}_n)_{n\in \N_0}$ a Markov chain on the state space $\mathcal{S}$ with transition matrix $M_{i}$ (\cref{def:matrix}) and initial distribution $\mu_{i}$ (\cref{def:distribution}). We denote by $\mathbb{P}$ the probability measure associated to the process. Notice that $X^i= X^i(g, \mathbf{a}, \mathbf{f}, p, \ell,\Sigma)$, but we omit this dependency in the notation for simplicity.\label{def:Markov}
\end{definition}

\begin{remark}
   Let $\Sigma\subseteq \mathcal{D}^{\N_0}$ be a 1-step shift of finite type. For any $w=w_0\ldots w_n\in \mathcal{L}(\Sigma)$, $F(w)=F(w_n)$. Thus, $\mathcal{V}_\Sigma=\{F(d): d\in \mathcal{D}\}.$ Note that the cover presented in \cref{graph:SFT} is the Fischer cover for $\Sigma$, except if there exist $d\neq d'\in \mathcal{D}$ such that $F(d)=F(d')$. In such a case, the symbols $d$ and $d'$ are associated to the same node in the Fischer cover.

    If the graph of \cref{graph:SFT} is irreducible and regular, we can use this cover to make the construction of the Markov chains instead of the Fischer cover. The only distinction is that different symbols are forced to be associated with different nodes, but the rest of the construction follows immediately. It is easy to see that the results that we will present in \cref{Section4} also hold when using this cover.\label{Remark_SFT}
\end{remark}

\section{From Markov chains to uniform distribution}\label{Section4}

Throughout this section, we will work within the framework outlined in \cref{Convention}, and assuming that $\Sigma$ is a $k$-regular subshift for some $k\geq 2$.
The connection between the Markov chains introduced in \cref{def:Markov} and the distribution of words in the $\ell$-restricted language $\mathcal{L}_{\mathcal{V},\ell}(\Sigma)$ is straightforward from \cref{prop:keymarkov}.  
\begin{corollary} For any $i\geq \ell$, $(\mathbf{b},F)\in \mathcal{S}$ and $n\in \N$,
 \[\mathbb{P}\Big(X^{i}_n=(\mathbf{b},F)\Big)=\dfrac{|\{w\in \mathcal{L}_{\mathcal{V},\ell}^{i+np}(\Sigma): \ \mathbf{f}((w)_g)\equiv\mathbf{b}\Mod{\mathbf{a}}, \ F(w)=F\}|}{|\mathcal{L}_{\mathcal{V},\ell}^{i+np}(\Sigma)|}.\]\label{cor:Fundamental}
 \end{corollary}
 \begin{proof}
     Given the Markov chain $X^i$ with initial distribution $\mu_i$ and transition matrix $M_i$, for every $(\mathbf{b},F)\in \mathcal{S}$ and $n\in \N$, $$\mathbb{P}\Big(X^i_n=(\mathbf{b},F)\Big)=\sum_{s\in \mathcal{S}}\mu_i(s)M_i\Big(s, (\mathbf{b},F)\Big).$$
     The conclusion follows from \cref{def:distribution} and  \eqref{eq:keymarkov_2} of \cref{prop:keymarkov}.
 \end{proof}

 Since we are restricted to words in $\mathcal{L}_{\mathcal{V},\ell}(\Sigma)\subseteq \mathcal{L}(\Sigma)$, it is natural to consider the subset of $A_\Sigma$ defined by \[A_{\Sigma,\ell}:=\left\{\sum_{i=0}^n w_ig^i:\ n\geq \ell-1, \  w_0\cdots w_n\in \mathcal{L}_{\mathcal{V},\ell}(\Sigma)\right\}.\]
 Through $X_n^{i}$, the distribution of integers represented by words in $\mathcal{L}_{\mathcal{V},\ell}^{i+np}$ can be understood. By running different Markov chains for $i=\ell,\ell+1,\ldots,\ell+(p-1)$, we can recover the distribution in residue classes for $A_{\Sigma,\ell}$. Furthermore, as $\ell$ increases, the set difference $A_{\Sigma}\setminus A_{\Sigma,\ell}$ becomes negligible in terms of relative density, allowing us to recover the distribution for $A_\Sigma$.
 
In order to ensure the convergence of the Markov chains, we introduce the following definition.

  \begin{definition}
      Let $g, \mathbf{a}, \mathbf{f}, p, \ell$ and $\Sigma$ as defined in \cref{Convention}, and assume that the subshift $\Sigma$ is regular. We say that the \emph{Markov condition} holds if $M_i$ (\cref{def:matrix}) is irreducible and aperiodic for every $i\geq \ell$.
  \end{definition}
  
 If $i\geq \ell$, $M_i=M_{i+np}$ for all $n\in\N$, so the Markov condition it is equivalent to show that $M_i$ is irreducible and aperiodic for $i\in \{\ell,\cdots,\ell+(p-1)\}$. Notice that the Markov condition implicitly assumes that $\Sigma$ is mixing (while the construction only requires transitivity). Assuming the convergence of the Markov chains and using \cref{cor:Fundamental}, we devote the rest of the section to obtain the uniform distribution $\mathbf{f}\Mod{\mathbf{a}}$ for the $\times g$-invariant set $A_\Sigma$.

  \begin{lemma}
      If the Markov condition holds, there exist $C>0$ and $\rho\in (0,1)$ such that
   \begin{equation}
\left|\dfrac{|\{w\in \mathcal{L}_{\mathcal{V},\ell}^{n}(\Sigma): \ \mathbf{f}((w)_g)\equiv\mathbf{b}\Mod{\mathbf{a}},\ F(w)=F\}|}{|\mathcal{L}_{\mathcal{V},\ell}^{n}(\Sigma)|}-\dfrac{1}{
|\mathcal{S}|}\right|\leq C\rho^n\label{eq:Estimation1_1}
\end{equation}
 for every $n\geq \ell$ and $(\mathbf{b},F)\in \mathcal{S}$. In addition, 
 
  \begin{equation}
\left|\dfrac{|\{w\in \mathcal{L}_{\mathcal{V},\ell}^{n}(\Sigma): \ \mathbf{f}((w)_g)\equiv\mathbf{b}\Mod{\mathbf{a}},\ F(w)=F\}|}{|\{w\in \mathcal{L}_{\mathcal{V},\ell}^{n}(\Sigma):\ F(w)=F\}|}-\dfrac{1}{
|\mathbf{a}|}\right|\leq C\rho^n.\label{eq:Estimation1_2}
\end{equation}\label{lemma:Estimation1}
 \end{lemma}
\begin{proof} Notice that the uniform measure $1/|\mathcal{S}|$ is invariant for $M_i$ by \cref{prop:stochastic}. From \cref{thm:ConvergenceMC}, for every $i\in \{\ell,\cdots,\ell+p-1\}$ there exist $C_i>0$ and $\rho_i\in (0,1)$ such that 

 \begin{equation} \Big|\mathbb{P}\Big(X_n^{i}=(\mathbf{b},F)\Big)-\dfrac{1}{|\mathcal{S}|}\Big|\leq C_i \rho_i^n\label{Prob1}
 \end{equation} 
 for every $(\mathbf{b},F)\in \mathcal{S}$ and $n\in \N.$   By choosing $C':=\max_{i}\{C_i\rho_i^{-i/p }\}$ and $\rho:=\max_i\{\rho_i^{1/p}\}$, 
      \begin{equation}
     \Big|\mathbb{P}\Big(X_n^{i}=(\mathbf{b},F)\Big)-\dfrac{1}{|\mathcal{S}|}\Big|\leq C' \rho^{np+i}\label{Prob2}
 \end{equation} 
 for every  $i\in \{\ell,\cdots,\ell+(p-1)\}$.   From \cref{cor:Fundamental}, we can write \[\left|\dfrac{|\{w\in \mathcal{L}_{\mathcal{V},\ell}^{np+i}(\Sigma): \ \mathbf{f}((w)_g)\equiv \mathbf{b}\Mod{\mathbf{a}};\ F(w)=F\}|}{|\mathcal{L}_{\mathcal{V},\ell}^{np+i}(\Sigma)|}-\dfrac{1}{|\mathcal{S}|}\right|\leq C'\rho^{np+i}.\]
Since it holds for every $i\in \{\ell,\cdots,\ell+(p-1)\}$ and $n\in \N$, \eqref{eq:Estimation1_1} is obtained. Equivalently, for every $n\in \N$,
\begin{equation}
    \begin{split}
      |\mathcal{L}_{\mathcal{V},\ell}^{n}(\Sigma)|\left(\dfrac{1}{|\mathcal{S}|}-C'\rho^n\right)  &\leq |\{w\in \mathcal{L}_{\mathcal{V},\ell}^{n}(\Sigma): \ \mathbf{f}((w)_g)\equiv \mathbf{b}\Mod{\mathbf{a}};\ F(w)=F\}|\\
        &\leq |\mathcal{L}_{\mathcal{V},\ell}^{n}(\Sigma)|\left(\dfrac{1}{|\mathcal{S}|}+C'\rho^n\right).
    \end{split}\label{proof:Estimation1_1}
\end{equation}
By summing \eqref{proof:Estimation1_1} over all $\mathbf{b}\in \Z_{\mathbf{a}}$, we obtain
\begin{equation}|\mathcal{L}_{\mathcal{V},\ell}^{n}(\Sigma)|\left(\dfrac{1}{|\mathcal{V}|}-C'|\mathbf{a}|\rho^n\right)  \leq |\{w\in \mathcal{L}_{\mathcal{V},\ell}^{n}(\Sigma): F(w)=F\}|\label{proof:Estimation1_2}.\end{equation}
Let $n_0\in \N$ large enough such that $1-C'|\mathcal{S}|\rho^{n_0}>1/2$. Combining \eqref{proof:Estimation1_1} and \eqref{proof:Estimation1_2}, 
\begin{equation}
\begin{split}
    &|\{w\in \mathcal{L}_{\mathcal{V},\ell}^{n}(\Sigma): \ \mathbf{f}((w)_g)\equiv\mathbf{b}\Mod{\mathbf{a}};\ F(w)=F\}|\\
    &\leq \left(\dfrac{1}{|\mathcal{S}|}+C'\rho^{n}\right) \left(\dfrac{1}{|\mathcal{V}|}-C'|\mathbf{a}|\rho^n\right)^{-1} |\{w\in \mathcal{L}_{\mathcal{V},\ell}^{n}(\Sigma): F(w)=F\}|\label{proof:Estimation1_3}
\end{split}
\end{equation}
for every $n\geq n_0$. A quick computation leads to 
$$\left(\dfrac{1}{|\mathcal{S}|}+C'\rho^{n}\right) \left(\dfrac{1}{|\mathcal{V}|}-C'|\mathbf{a}|\rho^n\right)^{-1}=\dfrac{1}{|\mathbf{a}|}+2C'|\mathcal{V}|\dfrac{\rho^n}{1-C'|\mathcal{S}|\rho^n}\leq \dfrac{1}{|\mathbf{a}|}+4C'|\mathcal{V}| \rho^n.$$

Using the previous estimation in \eqref{proof:Estimation1_3}, 
\begin{equation*}
 \dfrac{|\{w\in \mathcal{L}_{\mathcal{V},\ell}^{n}(\Sigma): \ \mathbf{f}((w)_g)\equiv_{\mathbf{a}} \mathbf{b};\ F(w)=F\}|}{ |\{w\in \mathcal{L}_{\mathcal{V},\ell}^{n}(\Sigma): F(w)=F\}|}\leq \dfrac{1}{|\mathbf{a}|}+4C'|\mathcal{V}|\rho^n
\end{equation*}
for every $n\geq n_0$.
With the obvious modifications, we can obtain in an analogous way that
\begin{equation*}
 \dfrac{|\{w\in \mathcal{L}_{\mathcal{V},\ell}^{n}(\Sigma): \ \mathbf{f}((w)_g)\equiv_{\mathbf{a}} \mathbf{b};\ F(w)=F\}|}{ |\{w\in \mathcal{L}_{\mathcal{V},\ell}^{n}(\Sigma): F(w)=F\}|}\geq \dfrac{1}{|\mathbf{a}|}-\dfrac{2C'|\mathcal{V}|\rho^n}{1+C'|\mathcal{S}|\rho^n}\geq  \dfrac{1}{|\mathbf{a}|}-4C'|\mathcal{V}|\rho^n.
\end{equation*}
The conclusion follows by choosing $C>0$ large enough.
\end{proof}

 When $0\notin \mathcal{L}(\Sigma)$, each element of $A_{\Sigma,\ell}$ has a unique representation in the language, which is no longer true if $0\in \mathcal{L}(\Sigma)$. This lack of uniqueness can introduce complications when analysing the distribution of integers, as multiple representations may correspond to the same integer. 
We deal with this situation in the following result by using follower sets and the estimation \eqref{eq:Estimation1_2}.

\begin{lemma}
If the Markov condition holds, there exist $C>0$ and $\rho\in (0,1)$ such that
\[\left|\dfrac{|\{n\in A_{\Sigma,\ell}:\ \mathbf{f}(n)\equiv \mathbf{b}\Mod{\mathbf{a}} \}\cap [0,g^m)|}{|A_{\Sigma,\ell}\cap [0,g^m)|}-\dfrac{1}{|\mathbf{a}|}\right|\leq C\rho^m\]\label{lemma:Estimation2}
 for every $m\geq \ell$ and $(\mathbf{b},F)\in \mathcal{S}$.
\end{lemma}
\begin{proof}
      For $\mathbf{b}\in \Z_{\mathbf{a}}$, $F\in \mathcal{V}$ and $t\geq \ell$, we denote \begin{align*}
        \mathcal{L}_{\mathcal{V},\ell}^t(\mathbf{b},F)&:=\{w\in \mathcal{L}_{\mathcal{V},\ell}^{t}(\Sigma): \ \mathbf{f}((w)_g)\equiv\mathbf{b}\Mod{\mathbf{a}},\ F(w)=F\},\\
        \mathcal{L}_{\mathcal{V},\ell}^t(\mathbf{b})&:=\{w\in \mathcal{L}_{\mathcal{V},\ell}^{t}(\Sigma): \ \mathbf{f}((w)_g)\equiv\mathbf{b}\Mod{\mathbf{a}}\}, \text{ and }
        \\
        \mathcal{L}_{\mathcal{V},\ell}^t(F)&:=\{w\in \mathcal{L}_{\mathcal{V},\ell}^{t}(\Sigma): \ F(w)=F\}.
    \end{align*}
    If two different words represent the same integer, they must have different length. Additionally, any word of length $t$ ending with the digit $0$ represents an integer that can also be represented by a word of length $t-1$. Therefore, 
    \begin{equation}
    \begin{split}
      &|\{n\in A_{\Sigma,\ell}:\ \mathbf{f}(n)\equiv \mathbf{b}\Mod{\mathbf{a}} \}\cap [0,g^m)| \\
      &=|\mathcal{L}_{\mathcal{V},\ell}^{\ell}(\mathbf{b})|+\sum_{t=\ell+1}^{m}|\{w\in \mathcal{L}_{\mathcal{V},\ell}^t:\ \mathbf{f}((w)_g)\equiv \mathbf{b}\Mod{\mathbf{a}},\ w_{t-1}\neq 0\}|.
      \end{split} \label{proof:Estimation2_1}
\end{equation}
Any word of length $t$ that ends at $0$ must be an extension by $0$ of a word $w$ of length $t-1$. Recalling that $\mathbf{f}(0)=0$, and since $0$ can extend $w$ if and only if $0\in F(w)$, it follows
 \begin{equation}
       |\{w\in \mathcal{L}_{\mathcal{V},\ell}^t:\ \mathbf{f}((w)_g)\equiv \mathbf{b}\Mod{\mathbf{a}};\ w_{t-1}\neq 0\}|=|\mathcal{L}_{\mathcal{V},\ell}^t(\mathbf{b})|-\sum_{F\in \mathcal{V}}|\mathcal{L}_{\mathcal{V},\ell}^{t-1}(\mathbf{b},F)|\cdot \mathbbm{1}_{F}(0).\label{proof:Estimation2_2} 
    \end{equation}

 Let $C>0$ and $\rho\in (0,1)$ given by \cref{lemma:Estimation1}. From \eqref{eq:Estimation1_1},
 \begin{equation}
     |\mathcal{L}_{\mathcal{V},\ell}^t(\mathbf{b})|\leq |\mathcal{L}^t_{\mathcal{V},\ell}|\left(\dfrac{1}{|\mathbf{a}|}+C|\mathcal{V}|\rho^n\right).\label{proof:Estimation2_extra}
 \end{equation}
From \eqref{proof:Estimation2_2}, \eqref{proof:Estimation2_extra} and the estimation provided by \eqref{eq:Estimation1_2} of \cref{lemma:Estimation1} for $|\mathcal{L}_{\mathcal{V},\ell}^{t-1}(\mathbf{b},F)|$, 
    \begin{align*}
     &|\{w\in \mathcal{L}_{\mathcal{V},\ell}^t:\ \mathbf{f}((w)_g)\equiv \mathbf{b}\Mod{\mathbf{a}};\ w_{t-1}\neq 0\}|\\
     &\leq \dfrac{1}{|\mathbf{a}|}\left(|\mathcal{L}_{\mathcal{V},\ell}^t|-\sum_{F\in \mathcal{V}}|\mathcal{L}_{\mathcal{V},\ell}^{t-1}(F)|\mathbbm{1}_F(0)\right) +C|\mathcal{V}|\rho^t|\mathcal{L}_{\mathcal{V},\ell}^t|+ C\sum_{F\in \mathcal{V}}\rho^{t-1}|\mathcal{L}_{\mathcal{V},\ell}^{t-1}(F)|\mathbbm{1}_F(0)\\
     &\leq \dfrac{1}{|\mathbf{a}|}|\{w\in \mathcal{L}_{\mathcal{V},\ell}^t:\ w_{t-1}\neq 0\}|+C'\rho^{t}|\mathcal{L}_{\mathcal{V},\ell}^{t}|,
   \end{align*}
    for some constant $C'>C$. Using the previous estimation and \eqref{proof:Estimation2_extra} in \eqref{proof:Estimation2_1},
\begin{equation}
    \begin{split}
            &|\{n\in A_{\Sigma,\ell}:\ \mathbf{f}(n)\equiv \mathbf{b}\Mod{\mathbf{a}} \}\cap [0,g^m)|\\
            &\leq \dfrac{1}{|\mathbf{a}|}\left(|\mathcal{L}_{\mathcal{V},\ell}^\ell|+\sum_{t=\ell+1}^m|\{w\in \mathcal{L}_{\mathcal{V},\ell}^t:\ w_{t-1}\neq 0\}|\right)+C\rho^\ell|\mathcal{L}^\ell_{\mathcal{V},\ell}|+C'\sum_{t=\ell+1}^m\rho^t|\mathcal{L}_{\mathcal{V},\ell}^t|\\
       &=\dfrac{1}{|\mathbf{a}|}|A_{\Sigma,\ell}\cap [0,g^m)|+C'\sum_{t=\ell}^m\rho^t|\mathcal{L}_{\mathcal{V},\ell}^t|. 
    \end{split} \label{proof:Estimation2_3}
\end{equation}
 Since $|A_{\Sigma,\ell}\cap [0,g^m)|\geq |\mathcal{L}_{\mathcal{V},\ell}^m|$,  \cref{prop:cardinality} leads to
   \[|A_{\Sigma,\ell}\cap [0,g^m)|^{-1}\sum_{t=\ell}^m\rho^t|\mathcal{L}_{\mathcal{V},\ell}^t|\leq \sum_{t=\ell}^m \rho^t\dfrac{|\mathcal{L}_{\mathcal{V},\ell}^t|}{|\mathcal{L}_{\mathcal{V},\ell}^m|} \leq \sum_{t=\ell}^m\rho^t k^{-(m-t)}\leq m\rho^m,\]
where in the last inequality we assumed that $k^{-1}\leq \rho$ (otherwise, we change $\rho$ by $k^{-1}$). Hence, it follows from \eqref{proof:Estimation2_3} that
   \[\dfrac{|\{n\in A_{\Sigma,\ell}:\ \mathbf{f}(n)\equiv \mathbf{b}\Mod{\mathbf{a}} \}\cap [0,g^m)|}{|A_{\Sigma,\ell}\cap [0,g^m)|}-\dfrac{1}{|\mathbf{a}|}\leq C'm\rho^m.\]
   Following the same approach, the analogue lower inequality can be obtained. If $r=\rho^{1/2}$ and $C''>0$ is large enough, it holds $C'm\rho^m\leq C'' r^m$ for all $m\in \N$, concluding the result.
   \end{proof}

   Let $N\in \N$, and consider its expansion in base $g$, $N=d_0+d_1g+\ldots+w_mg^m$, where $w_m\neq 0$. We can write $[0,N)=[0,g^m)\cup [g^m,N)$, and using \cref{lemma:Estimation2}, the distribution can be recovered when restricted to the interval $[0,g^m)$. The next lemma focuses on addressing the distribution in the interval  $[g^m,N)$.

\begin{lemma}
 If the Markov condition holds, there exist $C>0$ and $\rho\in (0,1)$  such that for all $t,m\in \N$ (with $t<m$), and for any sequence of digits $d_t,d_{t+1},\ldots,d_m\in \{0,\cdots,g-1\}$ where $d_m\neq 0$, defining $M:=d_mg^m+\cdots+d_{t+1}g^{t+1}$ and assuming $A_{\Sigma,\ell}\cap [M,M+d_tg^t)\neq \emptyset$, the following estimation holds for every $\mathbf{b}\in \Z_{\mathbf{a}}:$ 
 \begin{equation*}
 \left|\dfrac{ |  \{n\in A_{\Sigma,\ell}:\ \mathbf{f}(n)\equiv \mathbf{b}\Mod{\mathbf{a}}\}\cap [M,M+d_tg^{t})|
}{|A_{\Sigma,\ell}\cap [M,M+d_tg^t)|}-\dfrac{1}{|\mathbf{a}|}\right|
    \leq C\rho^t.
 \end{equation*}
 \label{lemma:Estimation3}
\end{lemma}
\begin{proof}
 Let $C$ and $\rho$ be as given by \cref{lemma:Estimation1}, and let $d\in \{0,\cdots,g-1\}$. Let
    \begin{equation}
        x\in \{n\in A_{\Sigma,\ell}:\ \mathbf{f}(n)\equiv\mathbf{b}\Mod{\mathbf{a}}\}\cap [M+dg^{t},M+(d+1)g^{t}).\label{proof:Estiomation3_1}
    \end{equation}
  Since $x\in[M+dg^{t},M+(d+1)g^{t})$, the representation of $x$ can be written as $x=(x_0\cdots x_{m})_g$, where $x_{t}=d$ and $x_{t+1}\cdots x_m=d_{t+1}\cdots d_m$. Therefore, using the $g$-additivity, we can see that $\mathbf{f}(x)\equiv\mathbf{b}\Mod{\mathbf{a}}$ is equivalent to $$\mathbf{f}((x_0\cdots x_{t-1})_g)\equiv\mathbf{b}-\mathbf{f}(dg^t+d_{t+1}g^{t+1}+\cdots+d_mg^m)\Mod{\mathbf{a}}=:\Tilde{\mathbf{b}}(d).$$
Additionally, as $d_m\neq 0$, $x=(x_0\cdots x_{t-1}dd_{t+1}\cdots d_m)_g\in A_{\Sigma,\ell}$ is equivalent to $$x_0\cdots x_{t-1}\in \mathcal{L}_{\mathcal{V},\ell}(\Sigma)\ \text{and}\ dd_{t+1}\cdots d_m\in F(x_0\cdots x_{t-1}).$$ Therefore, \eqref{proof:Estiomation3_1} is equivalent to
\[x_0\cdots x_{t-1}\in \{w\in \mathcal{L}_{\mathcal{V},\ell}^t(\Sigma):\ \mathbf{f}((w)_g)\equiv\Tilde{\mathbf{b}}(d)\Mod{\mathbf{a}}\} \text{ and } dd_{t+1}\cdots d_m\in F(x_0\cdots x_{t-1}). \]
Using this equivalence and the estimation \eqref{eq:Estimation1_2} of \cref{lemma:Estimation3},
\begin{equation}
\begin{split}
    &| \{n\in A_{\Sigma,\ell}:\ \mathbf{f}(n)\equiv\mathbf{b}\Mod{\mathbf{a}}\}\cap [M+dg^{t},M+(d+1)g^{t})|\\
    &=\sum_{F\in \mathcal{V}} |\{w\in \mathcal{L}_{\mathcal{V},\ell}^{t}(\Sigma): \ \mathbf{f}((w)_g)\equiv\Tilde{\mathbf{b}}(d)\Mod{\mathbf{a}};\  F(w)=F\} | \cdot \mathbbm{1}_F\left(dd_{t+1}\cdots d_m\right)\\
    &\leq \sum_{F\in \mathcal{V}}  \left(\dfrac{1}{|\mathbf{a}|}+C\rho^{t}\right)|\{w\in\mathcal{L}_{\mathcal{V},\ell}^{t}(\Sigma):\ F(w)=F\}| \cdot \mathbbm{1}_F(dd_{t+1}\cdots d_m).
\end{split}\label{proof:Estimation3_2}
\end{equation}

With an analogous argument as above, we can obtain that
$$\sum_{F\in \mathcal{V}}|\{w\in \mathcal{L}_{\mathcal{V},\ell}^{t}(\Sigma): F(w)=F\}| \cdot \mathbbm{1}_F(dd_{t+1}\cdots d_m)=|A_{\Sigma,\ell}\cap [M+dg^t,M+(d+1)g^t)|.$$ 
From the assumption $A_{\Sigma,\ell}\cap [M,M+d_tg^t)\neq \emptyset$, it is clear that $d_t\geq 1$. Thus, we can use the previous identity in \eqref{proof:Estimation3_2}, and then sum from  $d=0$ to $d_t-1$, obtaining
\begin{align*}
    | \{n\in A_{\Sigma,\ell}:\ \mathbf{f}(n)\equiv_{\mathbf{a}}\mathbf{b}\}\cap [M,M+d_tg^{t})|\leq \left(\dfrac{1}{|\mathbf{a}|}+C\rho^t\right)|A_{\Sigma,\ell}\cap [M,M+d_tg^t)|.
\end{align*}
With a similar argument, the lower inequality can be obtained, concluding the result.
\end{proof}

Putting the previous results together, the distribution for the set $A_{\Sigma,\ell}$ can be estimated for an arbitrary interval $[0,N)$.

\begin{proposition}
     Let $g\geq 2$, $\mathbf{a}\in \N^r$, and $\mathbf{f}=(f_1,\ldots,f_r)$ be a collection of $g$-additive functions eventually periodic with respect to $\mathbf{a}$.
      Let $\Sigma\subseteq \{0,\cdots,g-1\}^{\N_0}$ be a mixing, sofic  and regular subshift, and $(p,\ell)$ be an eventual period such that $\ell$ is at least the length of the shortest synchronizing word.

      If the Markov condition holds, there exists $\gamma>0$ such that \[\dfrac{| \{n\in A_{\Sigma,\ell}:\ \mathbf{f}(n)\equiv \mathbf{b}\Mod{\mathbf{a}}\}\cap [0,N)|}{|A_{\Sigma,\ell}\cap [0,N)|}=\frac{1}{a_1\cdot a_2\cdots a_r}+O(N^{-\gamma})\] for every $\mathbf{b}\in \Z_{\mathbf{a}}$.\label{prop:main}
 \end{proposition}
 \begin{proof}
    Since the Markov condition holds, we can apply \cref{lemma:Estimation2} and \cref{lemma:Estimation3} with constants $C>0$ and $\rho\in (0,1)$. Let $N\in \N$, and  its base-$g$ expansion $N=d_mg^m+\ldots+d_1g+d_0,$
    where $d_m\neq 0$. For every $t\in \{0,\cdots,m\}$, we define $N(t):=d_mg^m+\ldots+d_tg^t$. Noting
    \begin{equation}
        [0,N)=[0,g^m)\cup [g^m,d_mg^m)\cup \left(\bigcup_{t=0}^{m-1} [N(t+1),N(t))\right),\label{proof:main_decomposition}
    \end{equation}
we aim to estimate $|\{n\in A_{\Sigma,\ell}:\ \mathbf{f}(n)\equiv \mathbf{b}\Mod{\mathbf{a}}\}\cap I|$ for each interval $I$ in the previous decomposition of $[0,N)$. The interval $[0,g^m)$ can be estimated from \cref{lemma:Estimation2}, and  $[N(t+1),N(t))$ from \cref{lemma:Estimation3} as $N(t)=N(t+1)+d_tg^t$. On the other hand, the interval $[g^m,d_mg^m)$ can be handled in the same way as \cref{lemma:Estimation3} by writing $[g^m, g^m+(d_m-1)g^m)$. Putting these estimates together, we can obtain the upper bound
\begin{equation}
\begin{split}
    &\dfrac{| \{n\in A_{\Sigma,\ell}:\ \mathbf{f}(n)\equiv \mathbf{b}\Mod{\mathbf{a}}\}\cap [0,N)|}{|A_{\Sigma,\ell}\cap [0,N)|}
    \leq \dfrac{1}{|\mathbf{a}|}+C R(m),
\end{split}\label{proof:main_decomposition2}
\end{equation}
where
\[R(m):=\rho^m\dfrac{|A_{\Sigma,\ell}\cap [0,g^m)|}{|A_{\Sigma,\ell}\cap [0,N)|}+\rho^m\dfrac{|A_{\Sigma,\ell}\cap [g^m,d_mg^m)|}{|A_{\Sigma,\ell}\cap [0,N)|}
    +\sum_{t=0}^{m-1}\rho^t\dfrac{|A_{\Sigma,\ell}\cap [N(t+1),N(t))|}{|A_{\Sigma,\ell}\cap [0,N)|}.\]
  If $A_{\Sigma,\ell}\cap [N(t+1),N(t))\neq \emptyset$ (in particular, $d_t\neq 0$) each element of $A_{\Sigma,\ell}\cap [N(t+1),N(t))$ has a base-$g$ expansion of length $m+1$, where the $(m-t)$ most significant digits are $d_{t+1}\cdots d_m$. Then, they may differ only in the $(t+1)$ least significant digits. Therefore, $$|A_{\Sigma,\ell}\cap [N(t+1),N(t))|\leq |\mathcal{L}^{t+1}_{\mathcal{V},\ell}(\Sigma)|.$$  Let $k\in \N$ the regularity of $\Sigma$. Since $|A_{\Sigma,\ell}\cap [0,N)|\geq |\mathcal{L}^{m}_{\mathcal{V},\ell}(\Sigma)|$, we can get from the above estimation and \cref{prop:cardinality} that
  \[\dfrac{|A_{\Sigma,\ell}\cap [N(t+1),N(t))|}{|A_{\Sigma,\ell}\cap [0,N)|}\leq \dfrac{|\mathcal{L}_{\mathcal{V},\ell}^{\ell}(\Sigma)|k^{t+1-\ell}}{|\mathcal{L}_{\mathcal{V},\ell}^{\ell}(\Sigma)|k^{m-\ell}}=k^{t+1-m}\leq \rho^{m-t-1},\]
  where we assumed that $k^{-1}\leq \rho$ (otherwise, replace $\rho$ by $k^{-1}$).
  Therefore, using the previous estimation in \eqref{proof:main_decomposition2},
  \[\dfrac{| \{n\in A_{\Sigma,\ell}:\ \mathbf{f}(n)\equiv \mathbf{b}\Mod{\mathbf{a}}\}\cap [0,N)|}{|A_{\Sigma,\ell}\cap [0,N)|}
    -\dfrac{1}{|\textbf{a}|}\leq C(2+\rho^{-1}m)\rho^m.\]
The lower inequality can be obtained following the same argument. Considering $r=\rho^{1/2}$ and $C'$ large enough, $C'r^m>C(2+\rho^{-1}m)\rho^m$ for all $m\in \N$.  The conclusion follows as $m\log(g)\leq \log (N)\leq (m+1)\log(g).$
 \end{proof}
\begin{remark} When $\Sigma$ is a shift of finite type, $|A_\Sigma \setminus A_{\Sigma,\ell}|$ is finite for all $\ell$ large enough. In such a case,  we can replace $A_{\Sigma,\ell}$ by $A_\Sigma$ in \cref{prop:main}. Moreover, while the choice of the initial distribution for the Markov chains is important to recover $A_{\Sigma,\ell}$ in \cref{prop:main}, the convergence of the Markov chains remains independent of this choice. Therefore, different sets of integers can be considered by adjusting the initial distribution. \end{remark}

The next step is to transfer the distribution from $A_{\Sigma,\ell}$ to the multiplicatively invariant set $A_\Sigma$. To achieve this, it suffices to show that the difference between these two sets becomes asymptotically negligible in size. The main ingredient to conclude that is the following proposition.

\begin{proposition} Let $\Sigma$ be a transitive, sofic and regular subshift. Then, 
\begin{equation}
 \lim_{\ell\to \infty} \dfrac{|\mathcal{L}^\ell(\Sigma)\setminus\mathcal{L}^\ell_{\mathcal{V},\ell}(\Sigma)|}{|\mathcal{L}^\ell_{\mathcal{V},\ell}(\Sigma)|}=0.\label{eq:prop_estimation_language}  
\end{equation}\label{prop:estimation_language}
\end{proposition}
\begin{proof}
    Let $k\geq 2$ be the regularity of $\Sigma$. For any $y\in \mathcal{L}(\Sigma)$ such that $F(y)\in \mathcal{V}$, every extension $\Tilde{y}\in \mathcal{L}(\Sigma)$ of $y$ satisfies $F(\Tilde{y})\in \mathcal{V}$. Then,  for every $\ell\in \N$ and $i\geq \ell$, $\mathcal{L}^{i}_{\mathcal{V},\ell}\subseteq \mathcal{L}^i_{\mathcal{V},i}$.
  From \cref{prop:cardinality}, it follows
\begin{equation}
    |\mathcal{L}^i_{\mathcal{V},i}|\geq |\mathcal{L}^{i}_{\mathcal{V},\ell}|=|\mathcal{L}^{\ell}_{\mathcal{V},\ell}|k^{i-\ell}.\label{proof:prop_estimation_language_1}
\end{equation}
Since $\Sigma$ is transitive, the Fischer cover is an irreducible graph, so we can find $q\in \N$ large enough such that, for every node $F\in \mathcal{V}$, there is a path $\pi$ of length $q$ starting at $F$, where $\lambda(\pi)$ is a synchronizing word. To see this, notice that every word can be extended to a synchronizing word, and we can choose a $q$ common for every node as any extension of a synchronizing word is synchronizing.

Let $n\in \N$ and $x\in \mathcal{L}^{\ell+nq}\setminus\mathcal{L}_{\mathcal{V},\ell+nq}^{\ell+nq}$. Then, $x$ is in the language, but $F(x)\notin \mathcal{V}$.  In particular, $F(x_0\cdots x_{\ell-1})\notin \mathcal{V}$ and the word $x_{\ell}\cdots x_{\ell+nq-1}$ is not synchronizing. Thus,
\begin{equation}
    \left|\mathcal{L}^{\ell+nq}\setminus\mathcal{L}_{\mathcal{V},\ell+nq}^{\ell+nq}\right|\leq \left|\mathcal{L}^{\ell}\setminus\mathcal{L}_{\mathcal{V},\ell}^{\ell}\right|\times |\{w\in \mathcal{L}^{nq}(\Sigma):\ \text{$w$ is not synchronizing}\}|.\label{proof:prop_estimation_language_2}
\end{equation}
For each node $F\in \mathcal{V}$, there are $k^{qn}$ paths of length $qn$ in $\mathcal{G}$ starting at $F$. From those paths, there are at most $(k^{q}-1)^n$ paths labelled by non-synchronizing words, as for every node $F$ we can find at least one path starting at $F$ of length $q$ that represents a synchronizing word. From this idea and \eqref{proof:prop_estimation_language_2},
\begin{equation}
\left|\mathcal{L}^{\ell+nq}\setminus\mathcal{L}_{\mathcal{V},\ell+nq}^{\ell+nq}\right|\leq \left|\mathcal{L}^{\ell}\setminus\mathcal{L}_{\mathcal{V},\ell}^{\ell}\right|\times |\mathcal{V}|(k^q-1)^n.\label{proof:prop_estimation_language_3}\end{equation}
By defining $$C:=|\mathcal{V}|\max_{j=\ell,\cdots,\ell+(q-1)}\dfrac{\left|\mathcal{L}^{j}\setminus\mathcal{L}_{\mathcal{V},j}^{j}\right|}{|\mathcal{L}^j_{\mathcal{V},j}|},$$
and combining \eqref{proof:prop_estimation_language_1} with \eqref{proof:prop_estimation_language_3},
\[\dfrac{|\mathcal{L}^i\setminus\mathcal{L}^i_{\mathcal{V},i}|}{|\mathcal{L}^i_{\mathcal{V},i}|}\leq  C\dfrac{(k^q-1)^{\left \lfloor{i/q}\right \rfloor}}{k^{q{\left \lfloor{i/q}\right \rfloor}}}.\]
The result follows by taking $i\to \infty$.

\end{proof}
We can now state the main result of this section.
  \begin{theorem}
     Let $g\geq 2$, $\mathbf{a}=(a_1,\ldots,a_r)\in \N^r$, and $\mathbf{f}=(f_1,\ldots,f_r)$ be a collection of $g$-additive functions eventually periodic with respect to $\mathbf{a}$.
      Let $\Sigma\subseteq \{0,\cdots,g-1\}^{\N_0}$ be a mixing, sofic and regular subshift, and $p$ be an eventual period. If the Markov condition holds,  \[\lim_{N\to\infty}\dfrac{| \{n\in A_{\Sigma}:\ \mathbf{f}(n)\equiv \mathbf{b}\Mod{\mathbf{a}}\}\cap [0,N)|}{|A_{\Sigma}\cap [0,N)|}=\frac{1}{a_1\cdot a_2\cdots a_r}\] for every $\mathbf{b}=(b_1,\ldots,b_r)\in \Z_{\mathbf{a}}$.\label{MainTheorem}
 \end{theorem}
 \begin{proof}
     Let $(p,\ell)$ be an eventual period and $N\in \N$. For every $\mathbf{b}\in \Z_{\mathbf{a}}$, define
     \[\alpha(\ell,N):=\dfrac{|(A_\Sigma\setminus A_{\Sigma,\ell})\cap [0,N)|}{|A_{\Sigma,\ell}\cap [0,N)|},\ L(\mathbf{b},\ell,N):=\dfrac{| \{n\in A_{\Sigma,\ell}:\ \mathbf{f}(n)\equiv \mathbf{b}\Mod{\mathbf{a}}\}\cap [0,N)|}{|A_{\Sigma,\ell}\cap [0,N)|}.\]
A simple computation leads to the following inequality:
\[\dfrac{1}{1+\alpha(\ell,N)}L(\mathbf{b},\ell,N)\leq \dfrac{| \{n\in A_{\Sigma}:\ \mathbf{f}(n)\equiv \mathbf{b}\Mod{\mathbf{a}}\}\cap [0,N)|}{|A_\Sigma\cap [0,N)|}\leq L(\mathbf{b},\ell,N)+\alpha(\ell,N).\]
The \cref{prop:main} gives that $L(\mathbf{b},\ell,N)\to 1/|\mathbf{a}|$ when $N\to \infty$. Therefore, to conclude the theorem, it will be enough to show that
\begin{equation}
    \lim_{\ell\to\infty}\limsup_{N\to \infty} \alpha(\ell,N)=0.\label{proof:MainTheorem_1}
\end{equation}

For this purpose,  we will estimate $|(A_\Sigma\setminus A_{\Sigma,\ell})\cap [0,N)|$.  For $N\in \N$, let $m$ be such that $g^m\leq N<g^{m+1}$. If $n\in (A_\Sigma\setminus A_{\Sigma,\ell})\cap [g^t,g^{t+1})$ for some $\ell\leq t\leq m$,  $n$ can be represented by a word in $\mathcal{L}^{t+1}(\Sigma)\setminus\mathcal{L}^{t+1}_{\mathcal{V},\ell}(\Sigma).$  Thus, we can estimate
\begin{equation}
  \dfrac{|(A_\Sigma\setminus A_{\Sigma,\ell})\cap [g^\ell,g^{m+1})|}{|A_{\Sigma,\ell}\cap [g^\ell,g^m)|}\leq \dfrac{\sum_{i=\ell}^{m+1}|\mathcal{L}^i\setminus\mathcal{L}^{i}_{\mathcal{V},\ell}|}{|\mathcal{L}^m_{\mathcal{V},\ell}|}.\label{proof:MainTheorem_2}  
\end{equation}
It is easy to see that each word in $\mathcal{L}^i\setminus\mathcal{L}^{i}_{\mathcal{V},\ell}$ is an extension of a word in $\mathcal{L}^\ell\setminus\mathcal{L}^\ell_{\mathcal{V},\ell}$. Let $w\in\mathcal{L}^\ell\setminus\mathcal{L}^\ell_{\mathcal{V},\ell}$. Despite $F(w)\notin \mathcal{V}$, $F(w)$ is union of elements in $\mathcal{V}$. To see this, just consider the union of all $F\in \mathcal{V}$ such that the word $w$ can be read in some path of the Fischer cover ending at $F$. Since the subshift is $k$-regular for some $k\geq 2$, from the Fischer cover we can deduce that there are at most $|\mathcal{V}|k^{i-\ell}$ words of length $i$ that are extension of $w$. In particular, each of those extensions is in $\mathcal{L}^i\setminus\mathcal{L}^i_{\mathcal{V},\ell}$. Therefore, each word in $\mathcal{L}^\ell\setminus\mathcal{L}^\ell_{\mathcal{V},\ell}$ is extended to at most $|\mathcal{V}|k^{i-\ell}$ words of length $i$, concluding  from \eqref{proof:MainTheorem_2} that
\[\dfrac{|(A_\Sigma\setminus A_{\Sigma,\ell})\cap [g^\ell,g^{m+1})|}{|A_{\Sigma,\ell}\cap [g^\ell,g^m)|}\leq \dfrac{|\mathcal{V}||\mathcal{L}^\ell\setminus\mathcal{L}^\ell_{\mathcal{V},\ell}|\sum_{i=\ell}^{m+1}k^{i-\ell}}{|\mathcal{L}^\ell_{\mathcal{V},\ell}|k^{m-\ell}}\leq |\mathcal{V}|\sum_{i=-1}^{\infty}k^{-i} \dfrac{|\mathcal{L}^\ell\setminus\mathcal{L}^\ell_{\mathcal{V},\ell}|}{|\mathcal{L}^\ell_{\mathcal{V},\ell}|}.\]
Therefore,
\[0\leq\limsup_{N\to \infty}\alpha(\ell,N)\leq\limsup_{N\to\infty}\dfrac{g^\ell+|(A_\Sigma\setminus A_{\Sigma,\ell})\cap [g^\ell,g^{m+1})|}{|A_{\Sigma,\ell}\cap [g^\ell,g^m)|}\leq \dfrac{|\mathcal{L}^\ell\setminus\mathcal{L}^\ell_{\mathcal{V},\ell}|}{|\mathcal{L}^\ell_{\mathcal{V},\ell}|} |\mathcal{V}|\sum_{i=-1}^{\infty}k^{-i}.\]
It follows that  $\displaystyle \lim_{\ell\to\infty}\limsup_{N\to \infty} \alpha(\ell,N)=0$ from the previous inequality and \cref{prop:estimation_language}, concluding the result.
\end{proof}

\begin{remark}
    Although we focus on subsets of integers, the nature of the construction allows some of the ideas presented in this paper to be immediately extended to more general contexts where elements can be uniquely represented on a certain basis (canonical numerical systems). Relevant references to these systems include \cite{A-A,Kovacs,MM}. A particular case are the Gaussian integers, where unique expansion of elements can be provided in certain bases (see \cite{G_Thus, Katai,Kovacs2}).
\end{remark}

\section{Applications to the distribution of multiplicatively invariant sets of integers}

Let $A$ be a multiplicatively invariant set of integers that admits a representation $A=A_\Sigma$, where $\Sigma$ is a mixing, sofic and regular subshift. From \cref{MainTheorem}, uniform distribution is obtained if the Markov chains are irreducible and aperiodic, which might be challenging to verify. In order to check the Markov condition, we can conveniently choose the period $p$ and study the matrices $M_i$ using the characterization provided by \cref{prop:keymarkov}: For any $n\in \N$ and $s,s'\in \mathcal{S}$,
 $$M_{i}^n(s,s')=k^{-pn}|\{w\in \mathcal{L}^{pn}(\Sigma):\ \mathbf{f}(g^{i}(w)_g)\equiv \textbf{b}'-\textbf{b}\Mod{\textbf{a}}, \ F\overset{w}{\to} F'\}|.$$
  For $M_i$ to be irreducible, we need to check that for all $\mathbf{b}\in \Z_{\mathbf{a}}$ and $F,F'\in \mathcal{V}$,
\begin{equation}
    \text{There exists }w\in \mathcal{L}(\Sigma) \text{ such that }p\mid |w|,\  \mathbf{f}(g^i(w)_g)\equiv \mathbf{b} \Mod{\mathbf{a}} \text{ and } F\overset{w}{\to} F'. \label{irreducible}
\end{equation}

Assuming that $M_i$ is irreducible, aperiodicity holds if there exists $F\in \mathcal{V}$ such that 
\begin{equation}
     \gcd \Big\{n\in \N: \exists w\in \mathcal{L}^{pn}(\Sigma) \text{ such that } \mathbf{f}(g^{i}(w)_g)\equiv 0\Mod{\mathbf{a}} \text{ and }F\overset{w}{\to} F \Big\}=1.\label{aperiodic}
\end{equation}
In this section, we will study the function $\textbf{f}=(\id,S_g)$, where $S_g$ denote the function sum of digits in base $g$. 

\subsection{Uniform distribution of integers with missing digits}\label{Section5:Missing_digits}

Let $g\geq 2$ be an integer, and let $\mathcal{D}\subseteq \{0,\cdots,g-1\}$ be a subset of digits with at least two elements. We will write $\mathcal{D}=\{d_1,\cdots,d_t\}$, where $d_1<d_2<\cdots<d_t$.  Recalling that $\mathcal{C}_{g,\mathcal{D}}=A_\Sigma$ for $\Sigma=\mathcal{D}^{\N_0}$, it is easy to see that the Fischer cover for $\Sigma$ is the labelled graph with a unique node $F$, and for each $d\in \mathcal{D}$, there is a path from $F$ to $F$ labelled by $d$. Thus, $\Sigma$ is mixing, sofic and regular, so we can apply \cref{MainTheorem}. Moreover, the existence of a single follower set allows us to obviate this notion in the following discussion.

\begin{definition}
    Let  $A\subseteq \N_0$ be a $\times g$-invariant set, and $a,a'\in \N$. We say $A$ is \emph{jointly uniformly distributed $\Mod{(a,a')}$} if
\[\lim_{N\to\infty}\dfrac{|\{n\in A:\ n\equiv b\Mod{a},\ S_g(n)\equiv b'\Mod{a'}\}\cap[0,N)|}{|A\cap [0,N)|}=\dfrac{1}{aa'}\]
for every $b\in \Z_a,b'\in \Z_{a'}$.\label{joint_distribution}
\end{definition} 
From \cref{prop:periodicity}, we can find an eventual period of the form $(p,\ell)$  for $\mathbf{f}=(\id,S_g)$. Thus, checking that $M_i$ is irreducible reduces to show that for all $(b,b')\in \Z_a\times \Z_{a'}$,
\begin{equation}
 \text{there is $n\in \N$ and } w\in \mathcal{D}^{np} \text{ such that }g^i(w)_g\equiv b\Mod{a} \text{ and }S_g(w)\equiv b'\Mod{a'}.\label{irreducible_missingdigits}
\end{equation}
Similarly, $M_i$ is aperiodic if
\begin{equation}
     \gcd \Big\{n\in \N: \exists w\in \mathcal{D}^{pn} \text{ such that } (w)_g\equiv 0\Mod{a} \text{ and }  S_g(w)\equiv 0\Mod{a'}\Big\}=1.\label{aperiodic_missingdigits}
\end{equation}

Aperiodicity in the case $\gcd(g,a)=1$ follows immediately from the following result.
\begin{lemma}
Let $g\geq 2$ be an integer, and $a,a'\in \N$ such that $\gcd(g,a)=1$. Let $p:=a'\phi(a(g-1))$, where $\phi$ is the Euler's totient function. Then, for every $n\in \N$ and $w=d\cdots d\in \mathcal{D}^{np}$, $(w)_g\equiv 0\Mod{a}$ and $S_g(w)\equiv 0\Mod{a'}$.
    \label{aperiodicity}
\end{lemma}
\begin{proof}
   Let $n\in \N$. Notice that $\gcd(g,a(g-1))=1$, so the Euler's theorem implies that $g^p\equiv 1\Mod{a(g-1)}$. If $w=d\cdots d\in \mathcal{D}^{np}$,
       \begin{equation}
          (w)_g=d+dg+\cdots+dg^{np-1}=d\dfrac{g^{np}-1}{g-1}.
       \end{equation} 
       Since $g^{np}-1$ is divisible by $a(g-1)$, $\dfrac{g^{np}-1}{g-1}\equiv 0\Mod{a}$, concluding that $(w)_g\equiv 0\Mod{a}$. On the other hand, $S_g(w)=npd=a'(n\phi(a(g-1))d)\equiv 0\Mod{a'}$.
\end{proof}

The following application of Bézout's identity is fundamental to check the irreducibility of $M_i$.

\begin{lemma}
      Let $g\geq 2$ be an integer, and let $\mathcal{D}=\{d_1,\cdots,d_t\}$ be a set of digits, where $d_1$ is the smallest digit. Consider $a,a'\in\N$ such that $\gcd(g,a)=1$, and define $p:=a'\phi(a(g-1))$, where $\phi$ is the Euler's totient function. Then, there exists a word $w\in \mathcal{L}(\mathcal{D}^{\N_0})$ such that $$|w|\equiv 0\Mod{p}, \ (w)_g\equiv \delta\Mod{a},  \text{ and } S_g(w)\equiv \delta\Mod{a'},$$ where $\delta:=\gcd(aa',d_2-d_1,\cdots,d_t-d_1).$\label{Bezout}
\end{lemma}

\begin{proof}
From Bézout's identity, there exist $n_1,n_2,\cdots,n_t\in \Z$ such that \[\delta=n_1(aa')+n_2(d_2-d_1)+\cdots+n_t(d_t-d_1).\]
For all $j\in \{2,\cdots,t\}$, choose $k_j\in \N$ large enough such that $m_j:=k_j(aa')+n_j>0$. Then, 
\begin{equation}
      m_2(d_2-d_1)+\ldots+m_t(d_t-d_1)\equiv \delta\Mod{aa'}.\label{eq:Bezout}
\end{equation}
Let us define the word $w(j):=d_jd_1\ldots d_1\in \mathcal{D}^p$ for $j\in \{2,\cdots,t\}$. From \cref{aperiodicity},
\[(w(j))_g=(d_j-d_1)+d_1+d_1g+\cdots+d_1g^{p-1}\equiv d_j-d_1\Mod{a}.\]
Similarly, $S_g(w(j))=pd_1+(d_j-d_1)\equiv d_j-d_1 \Mod{a'}$.
Let us define  $$w:=\underbrace{w(2)\cdots w(2)}_\text{$m_2$ times}\cdots \underbrace{w(t)\cdots w(t)}_\text{$m_t$ times}.$$
Notice that $|w|=p(m_2+\ldots+m_t)$. Also, the periodicity $g^p\equiv 1 \Mod{a}$ and \eqref{eq:Bezout} imply that
 $$(w)_g\equiv \sum_{j=2}^t(w(j))_gm_j\equiv \sum_{j=2}^t(d_j-d_1)m_j\equiv \delta\Mod{a}.$$
 Using \eqref{eq:Bezout} again, we can conclude $$S_g(w)=\sum_{j=2}^t S_g(w(j))m_j\equiv \sum_{j=2}^t (d_j-d_1)m_j\equiv \delta\Mod{a'}.$$ 
\end{proof}

Leveraging the \cref{MainTheorem} and \cref{Bezout}, we can establish a complete description of the phenomenon of joint distribution (\cref{joint_distribution}) for missing digits sets when $\gcd(g,a)=\gcd(a,a')=1$. In particular, this shows \cref{thmA}.
\begin{theorem}
Let $g\geq 2$ be an integer, and let $\mathcal{D}=\{d_1,\cdots,d_t\}$ be a set of digits, where $d_1$ is the smallest digit.  Suppose that $\gcd(g,a)=\gcd(a,a')=1$, and recall the missing digits set $\mathcal{C}_{g,\mathcal{D}}$ defined in \eqref{def:missing_sets}. Defining $\delta:=\gcd(aa',d_2-d_1,\cdots,d_t-d_1)$,  
\begin{enumerate}
    \item $\mathcal{C}_{g,\mathcal{D}}$ is jointly uniformly distributed $\Mod{(a,a')}$ if and only if $\delta=1$.\label{I_Applications_thm1}
    \end{enumerate}
    More generally, concerning the limit 
    \begin{equation}L(b,b'):=\lim_{N\to\infty}\dfrac{|\{n\in \mathcal{C}_{g,\mathcal{D}}:\ n\equiv b\Mod{a},\ S_g(n)\equiv b'\Mod{a'}\}\cap[0,N)|}{|\mathcal{C}_{g,\mathcal{D}}\cap [0,N)|}\label{Missing_Theorem_limit},\end{equation}
it follows that:
    \begin{enumerate}
\setcounter{enumi}{1}
    \item Let $\delta_a:=\gcd(\delta,a)$ and $\delta_{a'}:=\gcd(\delta,a')$. If $d_1\equiv 0\Mod{\delta_a}$ and $d_1\equiv 0\Mod{\delta_{a'}}$, then $L(b,b')=\delta/aa'$ for every $(b,b')$ in the subgroup $\langle \delta_{a}\rangle\times \langle\delta_{a'}\rangle\subseteq \Z_a\times \Z_{a'}$. On the other hand, if $(b,b')\notin\langle \delta_{a}\rangle\times \langle\delta_{a'}\rangle$, 
    $$\Big\{n\in \mathcal{C}_{g,\mathcal{D}}:\ n\equiv b\Mod{a}, \ S_g(n)\equiv b'\Mod{a'}\Big\}=\emptyset.$$\label{II_Applications_thm1}
    \item In any other case, for every $(b,b')\in \Z_a\times \Z_{a'}$, the limit $L(b,b')$ either equals zero or does not exist.\label{III_Applications_thm1}
\end{enumerate}\label{Missing_Theorem}
\end{theorem} 
\begin{proof}
   Let us define $p=a'\phi(a(g-1))$, where $\phi$ is the Euler's totient function, and note that $(p,1)$ is an eventual period for $\mathbf{f}=(\id,S_g)$ with respect to $\mathbf{a}=(a,a')$. For $i\in \N$, let $X^i$ be the Markov chain associated to $g, \mathbf{a}, \mathbf{f}, p, \ell$ and the subshift $\Sigma=\mathcal{D}^{\N_0}$, as defined in \cref{def:Markov}. We denote by $M_i$ its transition matrix and $\mu_i$ its initial distribution.
   
  The aperiodicity of $M_i$ follows from \cref{aperiodicity} and \eqref{aperiodic_missingdigits}.

     Let us suppose $\delta=1$. From \cref{Bezout}, there exists a word $w$ using symbols of $\mathcal{D}$ such that $|w|\equiv 0\Mod{p}$, $(w)_g\equiv 1\Mod{a}$ and $S_g(w)\equiv 1\Mod{a'}$. Since $\gcd(a,a')=1$, the element $(1,1)$ is a generator for $\Z_a\times \Z_{a'}$. Then, for every $(b,b')\in \Z_a\times \Z_{a'}$, there exists $m\in \N$ such that $m\equiv (g^i)^{-1}b\Mod{a}$ and $m\equiv b'\Mod{a'}$. Let $y\in \mathcal{L}(\mathcal{D}^{\N_0})$ be constructed by concatenating $w$ with itself $m$ times. The periodicity $g^p\equiv 1\Mod{a}$ implies that $g^i(y)_g\equiv g^i m(w)_g\equiv b\Mod{a}$. On the other hand, $S_g(y)\equiv mS_g(w)\equiv b'\Mod{a'}$. Therefore, $M_i$ is irreducible for every $i\in\N$. This concludes that the Markov condition holds, then the \cref{MainTheorem} implies that $\mathcal{C}_{g,\mathcal{D}}$ is jointly uniformly distributed $\Mod{(a,a')}.$

     Now, let us suppose that $\delta\neq 1$. Since $\delta$ divides $aa'$ and $\gcd(a,a')=1$,  we can write $\delta=\delta_{a}\cdot \delta_{a'}$, where $\delta_{a}:=\gcd(\delta,a)$ and $\delta_{a'}:=\gcd(\delta,a')$. In particular, $\gcd(\delta_a,\delta_{a'})=1.$
     
     From the definition of $\delta$, note that $d_j\equiv d_1\Mod{\delta}$ for all $j\in\{1,\cdots,t\}$. Therefore, for any $m\in \N$ and $x=x_0\cdots x_{m-1}\in \mathcal{D}^m$,
     \begin{equation}
     (x)_g=\sum_{j=0}^{m-1}(x_j-d_1)g^j+d_1g^j\equiv d_1\sum_{j=0}^{m-1}g^j\Mod{\delta_a}, \text{ and } S_g(x)\equiv md_1\Mod{\delta_{a'}}.
         \label{proof:Missing_Theorem_1}
     \end{equation}
    In particular, for every $x\in \mathcal{L}(\mathcal{D}^{\N_0})$ such that $|x|\equiv 0\Mod{p}$, \cref{aperiodicity} implies that $(x)_g\equiv 0\Mod{\delta_a}$ and $S_g(x)\equiv 0\Mod{\delta_{a'}}$.  On the other hand, from \cref{Bezout} we can find a word $y$ such that $|y|\equiv 0\Mod{p}$, $(y)_g\equiv \delta\Mod{a}$, and $S_g(y)\equiv \delta\Mod{a'}.$ By the Chinese remainder theorem, $(\delta,\delta)$ is a generator for the subgroup $\mathcal{S}(\delta_a,\delta_{a'}):=\langle\delta_a\rangle\times \langle\delta_{a'}\rangle\subseteq \Z_a\times \Z_{a'}$. Therefore, recalling that $\gcd(g,\delta_a)=1$, 
 \begin{equation}
    \bigcup_{n\in\N}\{w\in \mathcal{D}^{np}:\ g^i(w)_g\equiv b\Mod{a},\ S_g(w)\equiv b'\Mod{a'}\}\neq \emptyset\label{proof:Missing_Theorem_2}
\end{equation} if and only if $(b,b')\in \mathcal{S}(\delta_a,\delta_{a'})$. In particular, this already ensures that $M_i$ is not irreducible on $\Z_a\times \Z_{a'}$, so the Markov condition does not hold. However, we can show that $X^i$ is an irreducible and aperiodic Markov chain when restricting the state space. Recall the initial distribution of $X^i$ given by
\[\mu_{i}(b,b')=\dfrac{|\{w\in \mathcal{D}^i: \ (w)_g\equiv b\Mod{a},\ S_g(w)\equiv b'\Mod{a'}\}|}{|\mathcal{D}|^i}.\]
Defining $\delta_a(i):= d_1\sum_{j=0}^{i-1}g^{i-1}\Mod{\delta_a}$ and $\delta_{a'}(i):= id_1\mod{\delta_{a'}}$ , \eqref{proof:Missing_Theorem_1} implies that $(x)_g\equiv \delta_a(i)\Mod{\delta_a}$ and $S_g(x)\equiv \delta_{a'}(i)\Mod{\delta_{a'}}$ for every $x\in \mathcal{D}^i$.  Therefore, $\mu_i(b,b')\neq 0$ implies $b\equiv\delta_a(i) \Mod{\delta_a}$ and $b'\equiv \delta_{a'}(i)\Mod{\delta_{a'}}$. If we define $$\mathcal{S}(i):=(\delta_a(i),\delta_{a'}(i))+\mathcal{S}(\delta_a,\delta_{a'}),$$ the equivalence in \eqref{proof:Missing_Theorem_2} indicates that $X^i$ never visits the states in $\Z_a\times \Z_{a'}\setminus \mathcal{S}(i)$, and also $X^i$ is irreducible and aperiodic when restricted to the state space $\mathcal{S}(i)$. Note that $M^i$ restricted to $\mathcal{S}(i)$ is still doubly stochastic, since we only ignore zero-value entries. Therefore, as $|\mathcal{S}(i)|=|\mathcal{S}(\delta_a,\delta_{a'})|=(a/\delta_a)\cdot (a'/\delta_{a'})$,
\begin{equation}
    \lim_{N\to \infty}\mu_{i+Np}(b,b')=\begin{cases}
\delta/aa' & \text{ if }(b,b')\in \mathcal{S}(i),\\
 0& \text{ if }(b,b')\notin \mathcal{S}(i).
\end{cases}\label{proof:Missing_Theorem_3}
\end{equation}

If $d_1\equiv 0\Mod{\delta_a}$ and $d_1\equiv 0\Mod{\delta_{a'}}$, $\mathcal{S}(i)=\mathcal{S}(\delta_a,\delta_{a'})$ for every $i\in \N$. This shows that each Markov chain $X^i$ is irreducible and aperiodic in a common state space $\mathcal{S}(\delta_a,\delta_{a'})$. Note that the Markov condition means that every $X^i$ is irreducible and aperiodic over the state space $\Z_a\times \Z_{a'}$, but it is easy to see that the procedure of \cref{Section4} is identical if we assume that each $X^i$ is irreducible and aperiodic over a subgroup of $\Z_a\times \Z_{a'}$ (and the Markov chain never visits the other states). Thus, by replicating the argument of \cref{Section4}, we can conclude \begin{equation*}
    \lim_{N\to \infty}\dfrac{|\{n\in \mathcal{C}_{g,\mathcal{D}}:\ n\equiv_a b,\ S_g(n)\equiv_{a'} b'\}\cap [0,N)|}{|\mathcal{C}_{g,\mathcal{D}}\cap [0,N)|}=\begin{cases}
 \delta/aa' & \text{ if } (b,b')\in \mathcal{S}(\delta_a,\delta_{a'}) \\
 0& \text{ if } (b,b')\notin \mathcal{S}(\delta_a,\delta_{a'}).
\end{cases}
\end{equation*}
This shows (\ref{II_Applications_thm1}). If either $d_1\not \equiv 0\Mod{\delta_a}$ or $d_1\not \equiv 0\Mod{\delta_{a'}}$, notice that  $\mathcal{S}(1)\cap \mathcal{S}(2)=\emptyset$. Then, for every $(b,b')\in \Z_a\times \Z_{a'}$, there exists $j\in \N$ such that the Markov chain $X^j$ does not visit the state $(b,b')$. From this fact, we want to show that \eqref{Missing_Theorem_limit} either equals zero or the limit does not exist. Let $(b,b')\in \Z_a\times \Z_{a'}$ and let $j\in \N$ such that $X^j$ never visits the state $(b,b')$. By contradiction, suppose that $L(b,b')>0$. Since $X^j$ does not visit the state $(b,b')$, $$\{n\in \mathcal{C}_{g,\mathcal{D}}:\ n\equiv b\Mod{a},\ S_g(n)\equiv b'\Mod{a'}\}\cap [g^{j+mp-1},g^{j+mp})=\emptyset$$ for every $m\in\N$. The contradiction follows by noting that 
\begin{align*}
        &L(b,b')=\lim_{m\to \infty}\dfrac{|\{n\in \mathcal{C}_{g,\mathcal{D}}:\ n\equiv b\Mod{a},\ S_g(n)\equiv b'\Mod{a'}\}\cap [0,g^{j+mp-1})|}{|\mathcal{C}_{g,\mathcal{D}}\cap [0,g^{j+mp-1})|}\\
            &>\liminf_{m\to \infty}\dfrac{|\{n\in \mathcal{C}_{g,\mathcal{D}}:\ n\equiv b\Mod{a},\ S_g(n)\equiv b'\Mod{a'}\}\cap [0,g^{j+mp})|}{|\mathcal{C}_{g,\mathcal{D}}\cap [0,g^{j+mp})|}.
\end{align*}
This concludes (\ref{III_Applications_thm1}). 
\end{proof}
Note that in the cases (\ref{I_Applications_thm1}) and (\ref{II_Applications_thm1}), we can conclude from \cref{prop:main} instead of using \cref{MainTheorem}, as $\Sigma$ is a full shift. Hence, the order of convergence $O(N^{-\gamma})$ for some $\gamma>0$ can be provided for the limit \eqref{Missing_Theorem_limit}. An immediate consequence follows by taking $a'=1$, which allows us to fully describe the phenomenon of uniform distribution $\Mod{a}$  when $\gcd(g,a)=1$. This shows \cref{CorA}.
\begin{corollary} Let $g\geq 2$ be an integer, and let $\mathcal{D}=\{d_1,\cdots,d_t\}$ be a set of digits, where $d_1$ is the smallest digit. 
Suppose $\gcd(g,a)=1$ and denote $\delta:=\gcd(a,d_2-d_1,\cdots,d_t-d_1).$ 
Then, 
\begin{enumerate}
    \item $\mathcal{C}_{g,\mathcal{D}}$ is uniformly distributed $\Mod{a}$ if and only if $\delta=1$.
    \item If $\delta\neq 1$ and $d_1\equiv0\Mod{\delta}$, $\mathcal{C}_{g,\mathcal{D}}$ is uniformly distributed in the subgroup $\langle \delta \rangle\subseteq \Z_a$. \label{II_Cor}
    \item In any other case, the limit of \[\dfrac{|\{n\in \mathcal{C}_{g,\mathcal{D}}:\ n\equiv b\Mod{a}\}\cap[0,N)|}{|\mathcal{C}_{g,\mathcal{D}}\cap [0,N)|}\] when $N\to \infty$ is either zero, or does not exist.
\end{enumerate}

    \label{cor:Missing_Theorem}
\end{corollary}

Note that the condition $\gcd(g,a)= 1$ is commonly assumed in results of this kind. In the framework of \cref{cor:Missing_Theorem}, when $\gcd(g,a)\neq 1$ the matrices $M_i$ are no longer irreducible over $\Z_a$. However, the Markov chains constructed can still capture the behaviour of the distribution by breaking $X^i$ into pieces determined by the communication classes of the Markov chain (see Section 1.2 in \cite{MarkovChain2} for details about this notion). The disadvantage of this situation is that the convergence of $X^i$ depends on the initial distribution. 

The following simple example shows that when $\gcd(a,g)\neq 1$, the limiting behavior depends heavily on the initial distribution.
\begin{example}
    Let $a=12$, $g=6$ and $\mathcal{D}=\{1,2,4\}$. Notice that $a$ divides $g^i$ for every $i\geq 2$, then $(w_0\cdots w_n)_g\equiv w_0+w_1g\Mod{a}$. Therefore, 
\[\lim_{N\to \infty}\dfrac{|\{n\in \mathcal{C}_{6,\mathcal{D}}:\ n\equiv b\Mod{12}\}\cap [0,N)|}{|\mathcal{C}_{6,\mathcal{D}}\cap [0,N)|}=\begin{cases}
2/9 & \text{ if } b\in \{1,2,4\}\\ 
1/9 & \text{ if } b\in\{7,8,10\}\\
0 & \text{ if } b\in\{0,3,5,6,9,11\}.
\end{cases}\]
Notice that this distribution agrees with the distribution when considering $[0,g^2)$. Also, we can see that the limit is different and non-zero for some classes, behaviour which is not observed under the assumption $\gcd(g,a)=1$.
\end{example}

Finally, we will present an equivalent condition for uniform distribution $\Mod{(a,a')}$ without the assumption $\gcd(a,a')=1$, however, at the cost of having a more complicated condition to check.
\begin{proposition}
    Let $g\geq 2$ be an integer, and let $\mathcal{D}=\{d_1,\cdots,d_t\}$ be a set of digits, where $d_1$ is the smallest digit. Consider $a,a'\in \N$ such that $\gcd(g,a)=1$, and define $p:=a'\phi(a(g-1))$. Therefore, $\mathcal{C}_{g,\mathcal{D}}$ is jointly uniformly distributed $\Mod{(a,a')}$  if and only if $\gcd(a,d_2-d_1,\cdots,d_t-d_1)=1$,
     and there exists $w\in \mathcal{L}(\mathcal{D}^{\N_0})$ such that $(w)_g\equiv 0\Mod{a}$, $S_g(w)\equiv 1\Mod{a'}$ and  $|w|\equiv 0 \Mod{p}$.\label{thm:General}
\end{proposition}
\begin{proof}
    If $\mathcal{C}_{g,\mathcal{D}}$ is jointly uniformly distributed $\Mod{(a,a')}$, in particular, it is uniformly distributed $\Mod{a}$. From \cref{cor:Missing_Theorem} follows that   $\gcd(a,d_2-d_1,\cdots,d_t-d_1)=1$. In addition, there exist infinitely many words $w$ such that $(w)_g\equiv 0\Mod{a}$, $S_g((w)_g)\equiv 1\Mod{a'}$ and  $|w|\equiv 0 \Mod{p}$. Otherwise, for every $m\in \N$, $$\Big\{n\in \mathcal{C}_{g,\mathcal{D}}:\ n\equiv 0\Mod{a},\ S_g(n)\equiv 1\Mod{1}\Big\}\cap [g^{mp-1},g^{mp})=\emptyset,$$ contradicting that 
    \[\lim_{N\to\infty}\dfrac{|\{n\in \mathcal{C}_{g,\mathcal{D}}:\ n\equiv 0\Mod{a},\ S_g(n)\equiv 1\Mod{a'}\}\cap[0,N)|}{|\mathcal{C}_{g,\mathcal{D}}\cap [0,N)|}=\dfrac{1}{aa'}>0.\]

    On the other hand, let us suppose that $\gcd(a,d_2-d_1,\cdots,d_t-d_1)=1$ and that there exists $w\in \mathcal{L}(\mathcal{D}^{\N_0})$ such that $(w)_g\equiv 0\Mod{a}$, $S_g(w)\equiv 1\Mod{a'}$ and $|w|\equiv 0 \Mod{p}$. From \cref{Bezout}, there exists $y\in \mathcal{L}(\mathcal{D}^{\N_0})$ such that $(y)_g\equiv 1\Mod{a}$ and  $|y|\equiv 0\Mod{p}$. To prove that $M_i$ is irreducible, we will check \eqref{irreducible_missingdigits}. Let $b\in\Z_a$ and $b'\in \Z_{a'}$. If $\Tilde{y}$ is the concatenation $\left((g^{i})^{-1}b\Mod{a}\right)$-times of the word $y$ (in case $b=0$, we concatenate $a$ times), the periodicity $g^p\equiv 1\Mod{a}$ implies that $g^i(\Tilde{y})_g\equiv g^i(g^{i})^{-1}b\equiv b \Mod{a}$. If $s:=S_g(\Tilde{y})$, we denote by $\Tilde{w}$ the concatenation $\left(b'-s\Mod{a'}\right)$-times of $w$. Thus, we can conclude that 
    $g^i(\Tilde{y}\Tilde{w})_g\equiv b\Mod{a}$ and $S_g(\Tilde{y}\Tilde{w})\equiv b'\Mod{a'}$. Clearly, $|\Tilde{y}\Tilde{w}|\equiv 0\Mod{p}$.  Therefore,  $M_i$ is irreducible. The matrix $M_i$ is also aperiodic as consequence of \cref{aperiodicity}. Thus, the Markov condition holds and the result follows by the \cref{MainTheorem}.
\end{proof}

\subsection{Revisiting Gelfond's theorem: An equivalent condition}

Let $g\geq 2$ be an integer, and $a,a'\in \N$. As presented in \cref{Gelfond_thm}, Gelfond \cite{Gelfond} studied the long-term behaviour of 
\begin{equation}
        |\{n\leq N:\ n\equiv b\Mod{a},\ S_g(n)\equiv b'\Mod{a'}\}|,\label{Fine_eq}
\end{equation}
 showing that, under the assumption $\gcd(g-1,a')=1$, the set $A=\N$ is jointly uniformly distributed $\Mod{(a,a')}$. In this section, we present an extension of \cref{Gelfond_thm}, obtaining an equivalent condition to uniform distribution $\Mod{(a,a')}$.

Notice that the obstruction in applying our method arises from the lack of the condition $\gcd(g,a)=1$, which implies that the Markov chains associated with this problem are not irreducible. However, we can still break each Markov chain $X^i$ into pieces, generating irreducible and aperiodic Markov chains. In that way, the behaviour of \eqref{Fine_eq} can be understood using our method.
\begin{theorem}
    Let $g\geq 2$ be an integer and $a,a'\in \N$. Then,
\begin{equation}\lim_{N\to\infty}\dfrac{ |\{n\leq N:\ n\equiv b\Mod{a},\ S_g(n)\equiv b'\Mod{a'}\}|}{N}=\dfrac{1}{aa'} \label{eq:Fine_Gelfond}
\end{equation} for every $(b,b')\in \Z_a\times \Z_{a'}$ if and only if there exists $n\in \N$ such that $n\equiv 0\Mod{a}$ and $S_g(n)\equiv 1\Mod{a'}.$\label{thm:Gelfond}
\end{theorem}
\begin{proof}
    If \eqref{eq:Fine_Gelfond} holds, there are infinitely many $n\in\N$ such that $n\equiv 0\Mod{a}$ and $S_g(n)\equiv 1\Mod{a'}$. 

For the other direction, let $n\in \N$ be such that $n\equiv 0\Mod{a}$ and $S_g(n)\equiv 1\Mod{a'}$. For $\mathbf{f}=(\id,S_g)$, let $p\in \N$ be  given by \cref{prop:periodicity}. Then, $(p,a-1)$ is an eventual period for $\mathbf{f}$ with respect to $\mathbf{(a,a')}$, and notice that $\Sigma=\{0,\cdots,g-1\}^{\N_0}$ is such that $\N=A_\Sigma$. For $i\geq a-1$, let $X^i$ be the Markov chain of \cref{def:Markov}, with transition matrix $M_i$ and initial distribution $\mu_i$. 

Let $w=10\cdots 0$ such that $|w|=p$, and clearly $g^i(w)_g\equiv g^i\Mod{a}$. If $y$ is the base-$g$ expansion of $n$ (by adding $0$'s, we can suppose that $y$ has length divisible by $p$), then $(y)_g\equiv 0\Mod{a}$ and $S_g(y)\equiv 1\Mod{a'}$. Hence, for any $(b,b')\in \Z_a\times \Z_{a'}$, there exists $k\in \N$ such that $(w^{b}y^{k})_g\equiv b\Mod{a}$ and $S_g(w^{b}y^{k})_g\equiv b'\Mod{a'}$ (where $z^j$ denotes the concatenation of $z$ $j$-times). Therefore,
\begin{equation}
 \bigcup_{m\in\N}\{w\in \mathcal{D}^{mp}:\ g^i(w)_g\equiv b\Mod{a},\ S_g(w)\equiv b'\Mod{a'}\}\neq \emptyset\label{proof:Fine_Gelfond_1}  
\end{equation}
if and only if $(b,b')\in \langle g^i \rangle\times \Z_{a'}\subseteq \Z_a\times \Z_{a'}.$ Denoting $\delta_i:=\gcd(g^i,a)$, 
it follows that $\sum_{m\in \N}M_i^m((b_0,b_1),(b_0',b_1'))>0$ if and only if $b_0\equiv b_0' \Mod{\delta_i}$. Thus, if the Markov chain $X^i$ starts at $X_0^i=(b,b')$, that trajectory will only visit classes in $ (b+\langle \delta_i \rangle)\times \Z_{a'}\subseteq \Z_a\times \Z_{a'}.$ In this way, we can split $X^i$ into pieces depending on the starting point. 

For each $j\in \{0,\cdots,\delta_i-1\}$, we define $\mathcal{S}(i,j):=j+\langle \delta_i \rangle$ and 

 $$\mu_{i,j}(b,b'):=\begin{cases}
 \dfrac{\mu_i(b,b')}{\Bar{\mu}_{i,j}}& \text{ if } b\equiv j\Mod{\delta_i} \\
 0& \text{ otherwise,}  
\end{cases}$$
where $\mu_i$ is the initial distribution of $X^i$ and $$\Bar{\mu}_{i,j}:=\sum_{b'\in \Z_{a'}}\sum_{b\in \mathcal{S}(j)}\mu_i(b,b')=\dfrac{|w\in \{0,\cdots,g-1\}^i:\ (w)_g\equiv j\Mod{\delta_i}\}|}{g^i}=\dfrac{1}{\delta_i}.$$ Let $X^{i,j}$ be a Markov chain on the state space $\mathcal{S}(i,j)$, with initial distribution $\mu_{i,j}$ and transition matrix $M_i$, both restricted to indexes of  $\mathcal{S}(i,j)$. The aperiodicity follows directly as $0\cdots 0$ is in the language of $\{0,\cdots,g-1\}^{\N_0}$, and $M_{i}$ restricted to $\mathcal{S}(i,j)$ is irreducible by \eqref{proof:Fine_Gelfond_1}. Therefore, the Markov chain $X^{i,j}$ converges to its invariant distribution given by $1/|\mathcal{S}(i,j)|=\delta_i/aa'.$ It is not difficult to see that
\begin{equation}
    \mu_{i+np}=\sum_{j=0}^{\delta_i-1}\Bar{\mu}_{i,j}\left(\mu_{i,j}M^n_i\right)=\dfrac{1}{\delta_i}\sum_{j=0}^{\delta_i-1}\mu_{i,j}M^n_i.\label{decomposition_mu}
\end{equation}
From \cref{thm:ConvergenceMC}, there exist $C_{i,j}>0$ and $\rho_{i,j}\in (0,1)$ such that
\begin{equation}
    \left|\left(\mu_{i,j}M_i^n\right)(b,b')-\dfrac{\delta_i}{aa'}\right|\leq C_{i,j}\rho_{i,j}^n\label{decomposition_convergence}
\end{equation}
for every $(b,b')\in \mathcal{S}(i,j).$ Clearly, $\left(\mu_{i,j}M_i^n\right)(b,b')=0$ if $(b,b')\notin \mathcal{S}(i,j)$. Since $\displaystyle\Z_a\times \Z_{a'}=\bigsqcup_{j\in \{0,\cdots,\delta_i-1\}}\mathcal{S}(i,j)$, from \eqref{decomposition_mu} and \eqref{decomposition_convergence} we can obtain that, for every $(b,b')\in \Z_a\times \Z_{a'}$,
\begin{equation}
    \left|\mu_{i+np}(b,b')-\dfrac{1}{aa'}\right|\leq C_i\rho_i^n,\label{mu_decomposition_MC}
\end{equation}
where $\displaystyle C_i:=\max_{j=0,\cdots,\delta_i-1}C_{i,j}$ and  $\displaystyle\rho_i:=\max_{j=0,\cdots,\delta_i-1}\rho_{i,j}$. Equivalently,
\begin{equation}
    \left|\mathbb{P}\Big(X^i_n=(b,b')\Big)-\dfrac{1}{aa'}\right|\leq C_i\rho_i^n.\label{mu_decomposition_MC2}
\end{equation}
Despite each $X^i$ is not irreducible, we still have convergence to the uniform measure $1/aa'$. Note that \eqref{mu_decomposition_MC2} is equivalent to the estimation \eqref{Prob1} in the case of $\Sigma$ being a full shift, and essentially, this is what is required to retrieve the results of \cref{Section4}, and in particular, \cref{MainTheorem}. Therefore, we can conclude that $\N$ is jointly uniformly distributed $\Mod{(a,a')}$ by replicating the arguments of \cref{Section4}.
\end{proof}
\begin{remark}
    The obstruction when trying to use the previous idea in the context of missing digits sets lies in the fact that for some $i\in \N$, the normalization factor $\Bar{\mu}_{i,j}$ could be not the same for each Markov chain $X^{i,j}$. 
\end{remark}

\subsection{Forbidding combinations of digits}\label{Section5:SFT}
So far, we have focused on applying \cref{MainTheorem} in restricted digit sets, where restrictions come from forbidding some digits (missing digits sets). Given the flexibility in choosing digits and restrictions in the general case, it does not seem possible to establish general conditions like those in \cref{Section5:Missing_digits} within the context of general $\times g$-invariant sets. In such a case, \cref{MainTheorem} is still a useful tool on a case-by-case basis. Seeking greater generality, we construct a class of sets with more flexible digit restrictions that exhibit joint uniform distribution. While we focus on restrictions of length 2, some of these ideas can be extended to longer lengths.

Let $g\geq 2$ be an integer and $a,a'\in \N$. Let $\mathcal{D}\subseteq \{0,\cdots,g-1\}$ be a set of digits, and let $T$ be a matrix on $\mathcal{D}$ such that $T(d,d')\in \{0,1\}$ for every $d,d'\in \mathcal{D}$. 
\begin{definition}
  We say that the matrix $T$ is \emph{$k$-regular} if for every $d\in \mathcal{D}$, $$\sum_{d'\in \mathcal{D}}M(d,d')=\sum_{d'\in \mathcal{D}}M(d',d)=k.$$ Also, we say that $T$ is \emph{irreducible} if for every $d,d'\in \mathcal{D}$, there exists $n\in \N$ such that $M^n(d,d')>0$.  
\end{definition}

As pointed out in \cref{Section2:Symbolic}, such a matrix $T$ induces a shift of finite type 
\[\Sigma_T:=\left\{x\in \mathcal{D}^{\N_0}:\ T(x_i,x_{i+1})=1 \ \forall i \in \N_0\right\}.\]
In fact, any 1-step subshift of finite type can be written in this form. Similarly, we can define the $\times g$-invariant set of integers generated by the subshift $\Sigma_T$ as
\[A_T:=A_{\Sigma_T}=\left\{\sum_{i=0}^nd_ig^i:\ T(d_i,d_{i+1})=1 \ \forall i\in \N_0\right\}.\]
As discussed in \cref{Remark_SFT}, for a 1-step SFT we can apply \cref{MainTheorem} when the Markov chains $X^i$ (\cref{def:Markov}) are constructed using the cover of \cref{graph:SFT} instead of the Fischer cover. It is also straightforward that the graph of \cref{graph:SFT} is regular and irreducible if the matrix $T$ is regular and irreducible.

\begin{theorem} Let $g\geq 2$ be an integer, and $a,a'\in \N$ such that $\gcd(g,a)=\gcd(a,a')=1$. Let $\mathcal{D}\subseteq \{0,\cdots,g-1\}$ be a set of digits, and $T$ be a matrix over $\mathcal{D}$ to values in $\{0,1\}$, and assume that $T$ is $k$-regular (for $k\geq 2$) and irreducible. If there exist $d,d'\in \mathcal{D}$ such that $T(d,d)=T(d,d')=T(d',d)=1$ and $\gcd(d'-d,aa')=1$, it follows
\[\lim_{N\to\infty}\dfrac{|\{n\in A_T:\ n\equiv b\Mod{a},\ S_g(n)\equiv b'\Mod{a'}\}\cap [0,N)|}{|A_T\cap [0,N)|}\]
for every $b\in \Z_a$ and $b'\in \Z_{a'}$.\label{Theorem_Matrix}
\end{theorem}
\begin{proof}
    We will check the conditions to apply \cref{MainTheorem}. Denote by $\Sigma\subseteq \mathcal{D}^{\N_0}$ the subshift of finite type associated to the matrix $T$. Since $T$ is regular and irreducible, the cover $\mathcal{G}_T$ constructed in \cref{graph:SFT} is regular and irreducible. Moreover, $\mathcal{G}_T$ is aperiodic, as there exist $d\in \mathcal{D}$ such that $T(d,d)$ (i.e., there is a cycle of length $1$ in $\mathcal{G}_T$). Let $p=a'\phi(a(g-1))$, and notice that $(p,1)$ is an eventual period for $\mathbf{f}=(\id,S_g)$ with respect to $(a,a')$. By considering the Markov chains $X^i$ of \cref{def:Markov} (with the consideration of \cref{Remark_SFT}), the transition matrix $M_i$ of $X^i$ has state space $\mathcal{S}=\Big(\Z_a\times \Z_{a'}\Big)\times \mathcal{V}$, where $\mathcal{V}=\{F(\Tilde{d}):\ \Tilde{d}\in \mathcal{D}\}.$ 
    
    Since $T(d,d)=1$, the word $x=d\cdots d\in \mathcal{D}^p$ satisfies $(w)_g\equiv 0\Mod{a}$ and $S_g(w)\equiv 0\Mod{a'}$ (\cref{aperiodicity}). In addition, $F(d)\overset{x}{\to} F(d)$ as the last symbol of $x$ is $d$. Thus, the state $(0,0,F(d))$ is aperiodic. 
    
     Let $(b_0,b_0',F_0),(b_1,b_1',F_1)\in \mathcal{S}$. To show that $M_i$ is irreducible for every $i$, we have to find $n(i)\in \N$ such that $$M^{n(i)}_i\Big((b_0,b_0',F_0),(b_1,b_1',F_1)\Big)>0.$$ Let $d_0,d_1\in \mathcal{D}$ such that $F_0=F(d_0)$ and $F_1=F(d_1)$. 
Since the cover $\mathcal{G}_T$ is aperiodic, there exists a path  (of length divisible by $p$) from $F_0$ to $F(d)$ labelled by some word $y$. Similarly, there exists a path labelled by $z$ (of length divisible by $p$) from $F(d)$ to $F_1$. 

Since $T(d,d)=T(d,d')=T(d',d)=1$, the word $w=d'd\cdots d\in \mathcal{D}^p$ can be read in a path from $F(d)$ to $F(d)$, $(w)_g\equiv d'-d\Mod{a}$ and $S_g(w)\equiv d'-d\Mod{a'}$. Since $\gcd(d'-d,aa')=1$ and $\gcd(a,a')=1$, $(d'-d,d'-d)$ generates $\Z_a\times \Z_{a'}$, so we can find $m\in \N$ such that $m(d'-d)\equiv b_1-b_0-(y)_g-(z)_g\Mod{a}$ and $m(d'-d)\equiv b_1'-b_0'-S_g(y)-S_g(z)\Mod{a'}$. Let $\Tilde{w}\in \mathcal{L}(\Sigma)$ be the concatenation $m$-times of $w$. It holds that $F_0\overset{y\Tilde{w}z}{\to} F_1$, the length of $y\Tilde{w}z$ is divisible by $p$, $(y\Tilde{w}z)_g\equiv b_1-b_0\Mod{a}$ and $S_g(y\Tilde{w}z)\equiv b_1'-b_0'\Mod{a'}$. Thus, if $n(i)=|y\Tilde{w}z|$, $M_i^{n(i)}((b_0,b_0',F_0),(b_1,b_1',F_1))>0$. Therefore, $M_i$ is irreducible and aperiodic for every $i\geq a-1$, so the conclusion follows from \cref{MainTheorem}.
\end{proof}
\begin{remark}
    Notice that the condition in the \cref{Theorem_Matrix} about the existence of $d,d'\in \mathcal{D}$ such that $T(d,d)=T(d,d')=T(d',d)=1$ and $\gcd(d'-d,aa')=1$ can be easily relaxed. However, we state it in that way because it allows us to check the Markov condition in a simple way, providing a tool to generate many examples of sets with digit constraints of length 2 that are jointly uniformly distributed.
\end{remark}
\begin{example}
    Let $g=10$, and consider the set of integers
    \[A:=\left\{\sum_{j=0}^m w_i10^i:\ w_i\in\{0,\cdots,9\},\ w_{i+1}\in \{w_i-1,w_i,w_i+1\}\Mod{10}\right\}.\]
    Notice that $A=A_T$, where $T(d,d')=1$ if and only if $d'\in \{d-1,d,d+1\}\Mod{10}$. Clearly, $T$ is $3$-regular and irreducible. Also, $T(0,0)=T(0,1)=T(1,0)=1$. 

Let $a,a'\in \N$ such that $\gcd(10,a)=\gcd(a,a')=1$. From the \cref{Theorem_Matrix},
\[\lim_{N\to\infty}\dfrac{|\{n\in A_T:\ n\equiv b\Mod{a},\ S_{10}(n)\equiv b'\Mod{a'}\}\cap [0,N)|}{|A_T\cap [0,N)|}=\dfrac{1}{aa'}\]
    for every $(b,b')\in \Z_a\times \Z_{a'}.$
\end{example}

\begin{example}
    Let $g=10$ and $\mathcal{D}=\{1,2,d,d'\}$ for any $d,d'\in \{0,3,4,\cdots,9\}$. Let $T$ a matrix indexed by $\mathcal{D}\times \mathcal{D}$ and defined by
    \[T=\begin{pmatrix}
 1&1  &0  &0  \\
 1&0  &1  &0  \\
 0&0  &1  &1  \\
 0&1  &0  &1  \\
\end{pmatrix}.\]
Since $T(1,1)=T(1,2)=T(2,1)$, the \cref{Theorem_Matrix} gives that $A_T$ is jointly uniformly distributed $\Mod{(a,a')}$ whenever $\gcd(10,a)=\gcd(a,a')=1$. This set can be understood as all the integers which base-10 expansion can be read in paths of the graph of \cref{Figure1}.\label{Example2}
    \end{example}

    \definecolor {processblue}{cmyk}{0,0,0,1}
\begin{figure}[h!]
\begin {center}

\begin {tikzpicture}[-latex ,auto ,node distance =3 cm and 3cm ,on grid ,
semithick ,
state/.style ={ circle ,top color =white , bottom color = processblue!5 ,
draw,processblue , text=black , minimum width =1 cm}]
\node[state] (C)
{};
\node[state] (D) [right =of C] {};
\node[state] (A) [above =of C] {};
\node[state] (B) [above=of D] {};

\path (A) edge [loop left] node[left] {$1$} (A);
\path (B) edge [bend left =15] node[below] {$1$} (A);
\path (A) edge [bend left =15] node[above] {$2$} (B);

\path (C) edge [bend left =15] node[below =0.15 cm] {$2$} (B);
\path (C) edge [loop left] node[below =0.15 cm] {$d'$} (C);
\path (B) edge [bend left=0] node[below right=0.15cm and 0.05 cm] {$d$} (D);
\path (D) edge [bend left=0] node[below right=0.15cm and 0.05 cm] {$d'$} (C);

\path (D) edge [loop right] node[right] {$d$} (D);
\end{tikzpicture}
\end{center}
\caption{Graph that represents the 2-regular and transitive shift of finite type associated to the Matrix $T$ in \cref{Example2}.}\label{Figure1}
\end{figure}
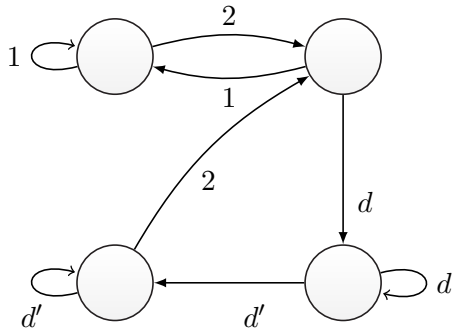

\subsection{Example involving a sofic subshift}\label{Section5:sofic}
When $\Sigma$ is a sofic subshift, more general restriction can be placed, but at the same time, it is harder to give general conditions to ensure uniform distribution. We will provide a sofic shift $\Sigma$, which is not of finite type, whose associated multiplicatively invariant set is uniformly distributed.

\begin{example}
    Let $\mathcal{G}$ be the graph presented in \cref{Figure2}, and let $\Sigma$ be the sofic shift $\Sigma:=\Sigma_\mathcal{G}$ (recall \cref{cover_def}). We can see that 
    $$\Sigma=\left\{x\in \{0,1,2\}^{\N_0}:\ \text{there is an even number of zeros between two 1's, or two 2's}\right\},$$
    and the cover $\mathcal{G}$ is the Fischer cover for $\Sigma$.\label{ExampleSofic}
\end{example}
\definecolor {processblue}{cmyk}{0,0,0,1}
\begin{figure}[h!]
\begin {center}
\begin {tikzpicture}[-latex ,auto ,node distance =3 cm and 3cm ,on grid ,
semithick ,
state/.style ={ circle ,top color =white , bottom color = processblue!5 ,
draw,processblue , text=black , minimum width =1 cm}]
\node[state] (A) [] {};
\node[state] (B) [right=of A] {};

\path (A) edge [loop left] node[left] {1} (A);
\path (B) edge [bend left =15] node[below] {0} (A);
\path (A) edge [bend left =15] node[above] {0} (B);
\path (B) edge [loop right] node[right] {2} (B);
\end{tikzpicture}
\end{center}
\caption{Fischer cover of the sofic shift $\Sigma$ presented in \cref{ExampleSofic}.}\label{Figure2}
\end{figure}
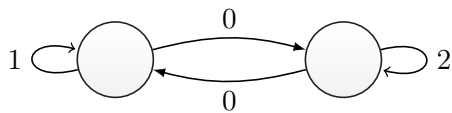
Notice that the graph $\mathcal{G}$ is $2$-regular and aperiodic, so that $\Sigma$ is a mixing, sofic and $2$-regular subshift. For $g=3$, we define the $\times 3$-invariant set $$A_\Sigma:=\left\{\sum_{j=0}^md_j3^j:\ \text{between two consecutive 1's or 2's there is an even number of 0's} \right\}.$$
Let $a,a'\in \N$ such that $\gcd(3,a)=\gcd(a,a')=1$. We can find a period $p\in \N$ such that $g^p\equiv 1\Mod{a(g-1)}$. Without loss of generality, we can suppose $p$ is even (otherwise, change $p$ by $2p$). The aperiodicity of the matrix $M_i$ follows because $0\cdots 0\in \mathcal{L}^p(\Sigma)$ is a cycle. Similarly,  the word $w=110\cdots 0\in \mathcal{L}^{p}(\Sigma)$ is a cycle in the graph $\mathcal{G}$. Also,  $(w)_3\equiv 4\Mod{a}$ and $S_3(w)\equiv 2\Mod{a'}$. By concatenating $w$ with itself, we can generate every state in $\langle 4\rangle\times \langle 2\rangle \subseteq \Z_a\times \Z_{a'}$. 

For instance, if $a=5$ and $a'=7$,  $\langle 4\rangle\times \langle 2\rangle=\Z_5\times \Z_{7}$. Considering that the word $w$ is read in a cycle of the Fischer cover, $\Big((w)_3,S_3(w)\Big)$ is a generator for $\Z_5\times \Z_7$ and $\mathcal{G}$ is aperiodic, it is possible to see that $M_i$ is irreducible for every $i\in\N$. The \cref{MainTheorem} implies that
\[\lim_{N\to\infty}\dfrac{|\{n\in A_\Sigma:\ n\equiv b\Mod{5},\ S_{3}(n)\equiv b'\Mod{7}\}\cap [0,N)|}{|A_\Sigma\cap [0,N)|}=\dfrac{1}{35}\]
    for every $(b,b')\in \Z_5\times \Z_{7}.$

\section{Transversality between multiplicatively invariant sets and arithmetic progressions}

In this section, we provide a partial answer to the open question \cite[Question 5.6]{GMR}. Recalling the notion of mass dimension provided in \cref{Mass_dimension}, the question is stated as follows: \emph{Let   $A\subseteq \N_0$ be a multiplicatively invariant set and $P$ be an arithmetic progression. Is it true that $\dim_{\text{M}}(A\cap P)$ is either zero or $\dim_{\text{M}}(A)$?}

Let $A$ be $\times g$-invariant, and recall the existence of a subshift $\Sigma$ such that $A=A_\Sigma$. If $A$ satisfies the conditions of the \cref{MainTheorem} with $f=\id$ and $a\in \N$ (including the Markov condition), it is possible to see that $\dim_{\text{M}}(A\cap (a\N+b))=\dim_{\text{M}}(A)$ for every $b\in \N$. Therefore, this provides a positive answer to the open question in those cases. As the notion of dimension is less sensitive to rough estimations than the relative density, we will modify the construction presented in \cref{Section3} to address this question for every $A=A_\Sigma$ when $\Sigma$ is any transitive sofic subshift.

\subsection{Counterexample for general cases}\label{Section6_1}
We will show that the answer is negative if we only suppose that the subshift $\Sigma$ is only transitive, or only sofic.

Let $g=a>4$, and  let $\mathcal{D}=\{d_0,d_1\}$ and $\mathcal{D}'=\{d_2,d_3,d_4\}$ be disjoint subsets of $\{0,\cdots,g-1\}$. We define the subshift $\Sigma=\mathcal{D}^{\N_0}\cup \mathcal{D}'^{\N_0}$. Notice that $\Sigma$ is a sofic subshift (in fact, it is a SFT). From \cref{languageset}, it is easy to see that $\dim_{\text{M}}(A_\Sigma)=\log(3)/\log(g).$ Since $a=g$, $n\in A_\Sigma\cap (a\N+d_0)$ if and only if $d_0$ is the least significant digit of $n$ in base $g$, so we can conclude that
\[0<\dim_{\text{M}}(A_\Sigma\cap (a\N+d_0))=\log(2)/\log(g)<\dim_{\text{M}}(A_\Sigma).\]

More generally, other examples can be constructed even when $\gcd(g,a)=1$. For this, we can use the same idea of considering a disjoint set of digits with different size, and choose the elements in order to apply (\ref{II_Cor}) of \cref{cor:Missing_Theorem}. This concludes that the sofic condition by itself is not enough to provide a positive answer to the open question. Now, we will see that assuming only transitivity is not enough either.

\begin{proposition}
    For any integer $g\geq 3$, there exists a transitive subshift $\Sigma$ and an arithmetic progression $P$ such that $0<\dim_\text{M}(A_\Sigma\cap P)<\dim_\text{M}(A_\Sigma).$\label{prop:ExampleTransitive}
\end{proposition}
\begin{proof}
      Let us choose a set of digits $\mathcal{D}\subsetneq \{0,\ldots,g-1\}$ with at least two elements, fixing some digits $h\notin \mathcal{D}$ and $d\in \mathcal{D}$.
    For every $k\in\N_0$, we define the set of words
    \[W_{k}:=\left\{hd^kz:\ z\in \mathcal{D}^k \right\}\subseteq \mathcal{L}(\{0,\cdots,g-1\}^{\N_0}),\]
where $d^k:=d\cdots d$ is the concatenation $k$ times of $d$. The length of any word in $W_k$ is $2k+1$ and $|W_k|=|\mathcal{D}|^k.$ Giving an enumeration to the $|\mathcal{D}|^k$ words of $W_k$, we can write
\[W_k=\left\{w_k^j: \ j=0,\cdots, |\mathcal{D}|^k-1 \right\}.\]
By concatenating all the words in the sets $W_k$, we define the infinite sequence \[w:=(w_0^0)(w_1^0w_1^1\cdots w_1^{|\mathcal{D}|-1})\cdots(w_k^0\cdots w_k^{|\mathcal{D}|^{k}-1})\cdots,\]
and we refer to the subword $w_k^0\cdots w_k^{|\mathcal{D}|^{k}-1}$ as the block induced by $W_k$. Considering the full shift $\left (\{0,\cdots,g-1\}^{\mathbb{N}_0},\sigma\right),$
we can define the transitive subshift
\[\Sigma=\overline{Orb_\sigma (w)}:=\overline{\left\{\sigma^n(w):\ n\in \N_0\right\}}.\] 
Also, it is straightforward that 
\[\mathcal{L}(\Sigma)=\Big\{x=x_0\cdots x_m:\ m\in \N_0,\ x \text{ appears in the sequence $w$}\Big\}.\]
We claim that $\dim_\text{M}(A_\Sigma)=\log(|\mathcal{D}|)/\log(g).$ To show it, notice  that $\mathcal{C}_{g,\mathcal{D}}\subseteq A_\Sigma$ as $\mathcal{D}^{\N_0} \subseteq \Sigma$.  Thus, \cref{dim_missingdigitsets} implies that \begin{equation}
\dim_\text{M}(A_\Sigma)\geq \dim_\text{M}(\mathcal{C}_{g,\mathcal{D}})=\log(|\mathcal{D}|)/\log(g).\label{proof:Transversality_1}
\end{equation}
For the other direction, we aim to estimate the number of elements in $A_\Sigma$ whose base-$g$ expansion has length at most $i \in \N$. After the block $W_i$, no new words of length at most $i$ appear in the sequence $w$. Therefore, any $n \in A_\Sigma \cap [0, g^i)$ can be generated by a subword with length at most $i$ of the sequence
\[w(i):=w_0^0w_1^0\cdots w_1^{|\mathcal{D}|-1} \cdots w_i^0\cdots w_i^{|\mathcal{D}|^i-1}.\]

In $w(i)$, there are $|w(i)|-(i-1)$ possible starting symbols for words of length $i$, so there are at most $|w(i)|-(i-1)$ words of length $i$. A quick estimation leads to
\[|w(i)|=\sum_{k=0}^{i}(2k+1)|\mathcal{D}|^k\leq (i+1)(2i+1)|\mathcal{D}|^i.\]
Therefore, $i(i+1)(2i+1)|\mathcal{D}|^i$ gives a rough estimation for the number of words of length at most $i$ that can be read in $w(i)$. Then, 

\[\dfrac{\log\left(|A_\Sigma \cap [0,g^i)|\right)}{i\log(g) } \leq \dfrac{\log(i(i+1)(2i+1))+i\log|\mathcal{D}|}{i\log(g)}.\]
Taking $i\to \infty$  and using \cref{lemma:DIM},  $\log(|\mathcal{D}|)/\log(g)\leq\dim_{\text{M}}(A_\Sigma).$ This concludes that $$\dim_{\text{M}}(A_\Sigma)=\log(|\mathcal{D}|)/\log(g).$$

Let $P:=g\N+h$ be an arithmetic progression. Notice that $n\in A_\Sigma\cap P$ if and only if $n\in A_\Sigma$ and the least significant digit of $n$\ in base $g$ is $h$. Thus, let us define $\Sigma(h):=\{v\in \mathcal{L}(\Sigma):\ v_0=h\},$  $\displaystyle A(h):=\left\{(v)_g:\ v\in \Sigma(h)\right\}\subseteq A_\Sigma$ and
 $r=g^2$. Notice that for any $i\in \N$  and $z\in \mathcal{D}^{i-1}$, $(hy^{i-1}z)_g\in A(h)\cap [0,r^i)$. Then, $|\mathcal{D}|^{i-1}\leq |A(h)\cap [0,r^i)|.$ 

On the other hand, after after the appearance of the word $w^0_{2i-1}$ in  the sequence $w$, the only word of length $2i$ that starts with $h$ and can be read is $hy^{2i-1}$. For $j<i$, there are $|\mathcal{D}|^{j}$ appearances of $h$ in the block $j$ of $w$, and each occurrence can be associated with one $2i$-long word. In the block $(w_i^0\cdots w_i^{|\mathcal{D}|^i-1})$, every word of length $2i$ starting with $h$ is given by $hy^iz$ for some $z\in \mathcal{D}^{i-1}$, hence, there are $|\mathcal{D}|^{i-1}$ of those words in $W_i$. More generally, for every $j\geq i$, every word of length $2i$ starting with $h$ in the block $W_j$ has the form $hy^jz$, where $z\in \mathcal{D}^{2i-j-1}$, so there are $|\mathcal{D}|^{2i-j-1}$ of such words in $W_j$. Using these estimates,
\[|\{w\in \Sigma(h):\ |w|=2i\}|\leq \sum_{j=0}^{i-1}|\mathcal{D}|^j+\sum_{j=i}^{2i-1}|\mathcal{D}|^{2i-j-1}\leq 2i|\mathcal{D}|^{i-1}.\]
Using the same estimate for words of length  less than $2i$, it follows that $$|A(h)\cap [0,r^i)|\leq (2i)^2|\mathcal{D}|^{i-1}.$$ Hence,
\[|\mathcal{D}|^{i-1}\leq |A(h)\cap [0,r^i)|\leq (2i)^2|\mathcal{D}|^{i-1}.\]
It follows that
\[\dfrac{i-1}{2i}\dfrac{\log(|\mathcal{D}|)}{\log(g)}\leq \dfrac{\log \left(|A(h)\cap [0,r^i)|\right)}{2i\log(g)}\leq \dfrac{2\log(2i)+i\log(|\mathcal{D}|)}{2i\log(g)},\]
and invoking \cref{lemma:DIM} we conclude
\begin{equation}
 \dim_{\text{M}}(A(h))=\dfrac{1}{2}\dfrac{\log(|\mathcal{D}|)}{\log(g)}.\label{Transversality_}   
\end{equation}
Since the arithmetic progression $P=g\N+h$ satisfies $P\cap A_\Sigma=A(h)$, from \eqref{Transversality_} follows that $$\dim_{\text{M}}(A_\Sigma\cap P)=\dfrac{\dim_{\text{M}}(A_\Sigma)}{2}.$$
\end{proof}
 Notice that for every rational $\alpha\in (0,1)$, it is possible to  modify the construction of the sets $W_j$ to construct a transitive subshift $\Sigma$ such that $\dim_{\text{M}}(A_\Sigma\cap P)=\alpha\dim_{\text{M}}(A_\Sigma)>0$.

\subsection{Positive answer for transitive sofic subshifts}

In this section, we will prove that the question \cite[Question 5.6]{GMR} has a positive answer when the multiplicatively invariant set can be represented by a subshift that is transitive and sofic.
\begin{remark}
    For the rest of this section, consider an integer $g\geq 2$ and $a\in \N$. Let $\Sigma\subseteq \{0,\ldots,g-1\}^{\N_0}$ be a transitive and sofic subshift, where $\mathcal{G}=(\mathcal{V},\mathcal{E},\lambda)$ is the Fischer cover of $\Sigma$. Let $p\in \N$ be as given by \cref{prop:periodicity}, so that $(p,a-1)$ forms an eventual period for the function $\mathbf{f}=\id$ with respect to $a$.\end{remark}
We utilize the sets $E_i$ from \cref{def:sets_E} to define a non-normalized version of the matrix $M$ and the measure $\mu$ introduced in \cref{Section3}. 
\begin{definition}
    For any $i\in \N$,  $b,b'\in \Z_a$ and $F,F'\in \mathcal{V}$, we define
\[\mathcal{M}_{i}\Big((b,F),(b',F')\Big)=|E_{i}(b'-b,F,F')|,\]
    and \[\nu_{i}(b,F)=|\{w\in \mathcal{L}^{i}(\Sigma):\  (w)_g\equiv b\Mod{a},\ F(w)=F\}|.\] 
\end{definition}
In a straightforward way, we can rewrite \cref{prop:keymarkov} in terms of $\mathcal{M}_i$ and $\nu_i$. 
\begin{corollary}
   If $i\geq a-1$ and $n\in\N$, $$\mathcal{M}_i^n\Big((b_0,F_0),(b_1,F_1)\Big)=|\{w\in \mathcal{L}^{np}(\Sigma):\ g^{i}(w)_g\equiv b'-b \Mod{a} \text{ and } F\overset{w}{\to}F'\}|.$$
 for any $b,b'\in \Z_a$ and $F,F'\in \mathcal{V}$. Also, $\nu_{i+np}=\nu_{i}\mathcal{M}_i^n.$\label{cor:keymarkov_2} 
\end{corollary} 
The main idea can be summarized as follows: Firstly, we want to find a word $w\in \mathcal{L}(\Sigma)$ such that $i:=|w|\geq a-1$, $(w)_g\equiv b\Mod{a}$ and $F(w)\in \mathcal{V}$. By studying the extensions of $w$ to words of length $i+pn$ for some $n\in\N$, we will see that a significant number of these extensions $\Tilde{w}$ satisfy $(\Tilde{w})_g\equiv b\Mod{a}$. As a conclusion, we aim to show $\dim_{\text{M}}(A_\Sigma\cap (a\N+b))=\dim_{\text{M}}(A_\Sigma)$. If such a word $w$ does not exist, then the intersection will be finite (hence, zero-dimensional).

For a given word $w$, it is easy to compute $(w)_g \Mod{a}$, however, decide whether $F(w)\in \mathcal{V}$ or not could be challenging if we don't know the Fischer cover explicitly. The following lemma shows that, for our purposes, such a check is not required.

\begin{lemma} Let $w\in \mathcal{L}(\Sigma)$ such that $|w|\geq a-1$, and $F(w)\notin \mathcal{V}$. Then, there exists an extension $\Tilde{w}$ of $w$ such that $(w)_g\equiv (\Tilde{w})_g\Mod{a}$, $F(\Tilde{w})\in \mathcal{V}$ and $|\Tilde{w}|=|w|+pn$ for some $n\in\N$. \label{lemma:Extension_to_synchronizing}
\end{lemma}
\begin{proof}
    Since $\Sigma$ is transitive and sofic, there exists a synchronizing word $y\in \mathcal{L}(\Sigma)$, and also there exists $x\in \mathcal{L}(\Sigma)$ such that $wxy\in \mathcal{L}(\Sigma)$. Since $y$ is synchronizing, any extension of $y$ is also synchronizing. Thus, we can suppose  $|xy|=mp$ for some $m\in \N$ (otherwise, we can extend $y$ to meet that condition). Since $y$ is synchronizing, the follower set of any word that ends with the subword $y$ is equal to $F(y)$. In particular,  $F:=F(wxy)=F(xy)=F(y)\in \mathcal{V}.$ Furthermore, every path in $\mathcal{G}$ labelled by $xy$ has $F$ as terminal node, and we can find a path $\pi$ labelled by $xy$ starting at some $F'\in \mathcal{V}$. As $\mathcal{G}$ is an irreducible graph, there is a path from $F$ to $F'$, labelled by some $z\in \mathcal{L}(\Sigma).$ For any $k\in\N$, denote $$(xyz)^k=\underbrace{(xyz)\cdots (xyz)}_{\text{$k$ times}}\in \mathcal{L}(\Sigma).$$
    Let $\Tilde{w}=w(xyz)^{ap}$, and $\ell:=|xyz|$. Using that $g^{p+j}\equiv g^p\Mod{a}$ for all $j\geq a-1$,
    $$(\Tilde{w})_g=(w)_g+g^{|w|}(xyz)^p+\cdots+g^{|w|+(a-1)p\ell}(xyz)^p=(w)_g+ag^{|w|}(xyz)^p\equiv (w)_g\Mod{a}.$$
    As the path associated to $z$ ends at $F'$, it follows that $F(\Tilde{w})=F'\in \mathcal{V}$. Also, notice that $|\Tilde{w}|=|w|+\ell a p$.
\end{proof}

We now establish a sufficient condition to ensure that extensions of a certain word $w$ will visit certain congruence classes and follower sets.
\begin{lemma}
    Let $w\in \mathcal{L}(\Sigma)$ such that $|w|\geq a-1$, $(w)_g\equiv b\Mod{a}$ and $F(w)\in \mathcal{V}$. If $\mathcal{M}^n_{|w|}\Big((b',F'),(b,F(w)\Big)>0$ for some $(b',F')\in \Z_a\times \mathcal{V}$ and $n\in \N$, there exists $\Tilde{w}\in \mathcal{L}(\Sigma)$ such that $(\Tilde{w})_g\equiv b'\Mod{a}$, $F(\Tilde{w})=F'$ and $|\Tilde{w}|=|w| +mp$ for some $m\in \N$.\label{lemma:Returning}
\end{lemma}

\begin{proof}
    As $\mathcal{M}^n_{|w|}\Big((b',F'),(b,F(w)\Big)>0$,  \cref{cor:keymarkov_2} provides the existence of a word $y$ with length $np$ such that $g^{|w|}(y)_g\equiv b-b'\Mod{a}$ and $ F'\overset{y}{\to}F(w).$ Equivalently, we can find a path $\pi_1$ in $\mathcal{G}$ with length $np$, labelled by $y$ and from  $F'$ to $F(w)$. Since the Fischer cover is irreducible, there is a path $\pi_2$ from $F(w)$ to $F'$, thus the path $\Tilde{\pi}:=(\pi_2\pi_1)^{p-1}\pi_2$ goes from $F(w)$ to $F'$, and $|\Tilde{\pi}|=p|\pi_2|+(p-1)|\pi_1|$ is divisible by $p$. If $\Tilde{\pi}$ is labelled by some $x\in \mathcal{L}(\Sigma)$, we can define $\Tilde{w}:=w(xy)^{a-1}x\in \mathcal{L}(\Sigma)$. Since the terminal node for $\Tilde{\pi}$ is $F'$, $F(\Tilde{w})=F'$. As $|x|$ and $|y|$ are divisible by $p$, it follows that $|\Tilde{w}|=|w|\Mod{p}$, and the periodicity implies
    \[(\Tilde{w})_g\equiv (w)_g+(a-1)g^{|w|}(x)_g+(a-1)g^{|w|}(y)_g+g^{|w|}(x)_g\equiv b+(a-1)g^{|w|}(y)_g\Mod{a}.\]
    Since $g^{|w|}(y)_g\equiv b-b'$, it follows $(\Tilde{w})_g\equiv b+(a-1)(b-b')\equiv b'\Mod{a}.$
\end{proof}
By combining \cref{lemma:Extension_to_synchronizing}, \cref{lemma:Returning} and ideas of previous sections, we can show a more general version of \cref{thmC}.

\begin{theorem}
    Let $\Sigma\subseteq \{0,\cdots,g-1\}^{\N_0}$ be a transitive sofic subshift, and $a\in \N$. If there exists $w\in \mathcal{L}(\Sigma)$ such that $|w|\geq a-1$ and $(w)_g\equiv b\Mod{a}$, then \begin{equation}
      \dim_{\text{M}}(A_\Sigma\cap (a\N+b))=\dim_{\text{M}}(A_\Sigma). \label{eq:thmB}  
    \end{equation}
    In particular, either \eqref{eq:thmB} holds, or $|A_\Sigma\cap (a\N+b)|\leq g^a$.\label{thm:Transversality}
\end{theorem}
\begin{proof}
    Consider $p\in \N$ given by \cref{prop:periodicity}. From \cref{lemma:Extension_to_synchronizing}, we can suppose $F(w)\in \mathcal{V}$ (otherwise, use $\Tilde{w}$ instead), and define  $i:=|w|\geq a-1$. In particular, $(p,i)$ is an eventual period. Consider the set \[\mathcal{R}:=\left\{(b',F')\in \Z_a\times \mathcal{V}: \text{ there exists $n\in \N$ s.t. } \mathcal{M}^n_i\Big((b',F'),(b,F)\Big)>0\ \right\}.\]
    From \cref{lemma:Returning}, for each $(b',F')\in \mathcal{R}$, there is $m'\in \N$ such that
$\nu_{i+m'p}(b',F')\neq 0$.

In particular, we can obtain some $m\in \N$ large enough such that
\begin{equation}
  \sum_{j=0}^{m}\nu_{i+jp}(b',F')\neq 0 \label{proof:ThmB_1}  
\end{equation}for every $(b',F')\in \mathcal{R}$. Let  $n\in \N$ and $(b,F)\in \Z_a\times \mathcal{V}$. Using \cref{cor:keymarkov_2},
\begin{align*}
    \sum_{j=0}^m\nu_{i+(j+n)p}(b,F)&=\sum_{j=0}^m \nu_{i+jp}\mathcal{M}^n_i\Big(\cdot, (b,F)\Big)\\
    &=\sum_{j=0}^m\sum_{b'\in \Z_a,F'\in \mathcal{V}}\nu_{i+jp}(b',F')\mathcal{M}_i^n\Big((b',F'),(b,F)\Big)\\
    &= \sum_{b'\in \Z_a,F'\in \mathcal{V}}\Big(\sum_{j=0}^m\nu_{i+jp}(b',F')\Big)\mathcal{M}_i^n\Big((b',F'),(b,F)\Big)\\
    &\geq \sum_{b'\in \Z_a,F'\in \mathcal{V}} \mathcal{M}_i^n\Big((b',F'),(b,F)\Big),
\end{align*}
where the last inequality follows from the definition of $\mathcal{R}$ and \eqref{proof:ThmB_1}. Combining the previous inequality, \cref{prop:E_size} and the characterization of \cref{cor:keymarkov_2}, we can conclude
\begin{equation}
\sum_{j=0}^m\nu_{i+(j+n)p}(b,F) \geq |\{\pi=(\pi_0\cdots \pi_{np-1}):\ \text{$\pi$ is a path of $\mathcal{G}$ and $t(\pi)=F$}\}|.\label{proof:ThmB_2}
\end{equation}

Since $\mathcal{G}$ is an irreducible graph, for each $F'\in \mathcal{V}$ there exists $\ell(F')\in \N$ such that there is a path of length $\ell(F')$ from $F'$ to $F$. Thus, defining $\ell:=\max_{F'\in \mathcal{V}}\ell(F'),$ for every $n$ large enough,
\begin{equation}
\begin{split}
    &|\{\pi=(\pi_0\cdots \pi_{np-1}):\ \text{$\pi$ is a path of $\mathcal{G}$ and $t(\pi)=F$}\}\\
    &\geq \sum_{F'\in \mathcal{V}}|\{\pi=(\pi_0\cdots \pi_{np-1-\ell(F')}):\ \text{$\pi$ is a path of $\mathcal{G}$ and $t(\pi)=F'$}\}|\\
    &\geq \sum_{F'\in \mathcal{V}}|\{\pi=(\pi_0\cdots \pi_{np-1-\ell}):\ \text{$\pi$ is a path of $\mathcal{G}$ and $t(\pi)=F'\}$}|\\
    &=|\{\pi=(\pi_0\cdots \pi_{np-1-\ell}):\ \text{$\pi$ is a path of $\mathcal{G}\}$}|,
\end{split}\label{proof:ThmB_3}
\end{equation}
 Since $\mathcal{G}$ is a cover for $\Sigma$, for every $x\in \mathcal{L}^{np-1-\ell}(\Sigma)$ there exists at least one path in $\mathcal{G}$ labelled by $x$. From this idea, \eqref{proof:ThmB_2} and \eqref{proof:ThmB_3}, it follows
 \begin{equation}
    \sum_{j=0}^m\nu_{i+(j+n)p}(b,F) \geq |\mathcal{L}^{np-\ell-1}(\Sigma)|. \label{proof:ThmB_4}
 \end{equation}

  On the other hand, by definition of $\nu_{i+(j+n)p}$, 
 \begin{equation}
     \sum_{j=0}^m\nu_{i+(j+n)p}(b,F)\leq \left|\left\{w\in \bigcup_{j=0}^{m}\mathcal{L}^{i+(j+n)p}(\Sigma):(w)_g\equiv b\Mod{a}\right\}\right|.\label{proof:ThmB_5}
 \end{equation}
 Since two words of different length can represent the same integer, the RHS of \eqref{proof:ThmB_5}  contains at most $m+1$ representations for each integer, then
 \begin{equation*}
     \left|\left\{w\in \bigcup_{j=0}^{m}\mathcal{L}^{i+(j+n)p}(\Sigma):(w)_g\equiv b\Mod{a}\right\}\right|\leq (m+1)|A_\Sigma\cap (a\N+b)\cap [0,g^{i+(m+n)p})|.
 \end{equation*}

Combining this estimation with \eqref{proof:ThmB_4} and \eqref{proof:ThmB_5}, it follows that
 \begin{equation*}
    \dfrac{1}{m+1}|\mathcal{L}^{np-\ell-1}(\Sigma)|\leq  |A_\Sigma\cap (a\N+b)\cap  [0,g^{i+(m+n)p})|.
 \end{equation*}
Taking $\log(\cdot)$ and rearranging terms,
\begin{align*}
    &\dfrac{np-\ell-1}{(i+mp+np)\log(g)}\dfrac{\log \left(|\mathcal{L}^{np-\ell-1}(\Sigma)|\right)}{np-\ell-1}-\dfrac{\log(m+1)}{(i+mp+np)\log(g)}\\
    &\leq \dfrac{\log\left( |A_\Sigma\cap (a\N+b)\cap  [0,g^{i+(m+n)p})| \right)}{(i+mp+np)\log(g)}.
\end{align*}

  Taking $\liminf$, we conclude that
 \[ h(\Sigma)/\log(g)\leq \liminf_{n\to \infty}\dfrac{\log\left(| A_\Sigma\cap (a\N+b)\cap [0,g^{i+mp+np})|\right)}{(i+mp+np)\log(g)},\]
 where $h(\Sigma)$ is the topological entropy of $\Sigma$.
 Since $A_\Sigma\cap (a\N+b)\subseteq A_\Sigma$, and recalling that $\dim_{\text{M}}(A_\Sigma)=h(\Sigma)/\log(g)$, we can obtain 
 \[\limsup_{n\to \infty}\dfrac{\log\left(| A_\Sigma\cap (a\N+b)\cap [0,g^{i+mp+np})|\right)}{(i+mp+np)\log(g)}\leq h(\Sigma)/\log(g). \]
 Therefore, 
  \[\lim_{n\to \infty}\dfrac{\log\left(| A_\Sigma\cap (a\N+b)\cap [0,g^{i+mp+np})|\right)}{(i+mp+np)\log(g)}=h(\Sigma)/\log(g).\]
From \cref{lemma:DIM}\footnote{In particular, we use a slight variant of this lemma that follows directly from the proof of Lemma 3.2 in \cite{GMR}.},
  \[\dim_{\text{M}}(A_\Sigma\cap (a\N+b))=h(\Sigma)/\log(g).\]
  
  Finally, if there is no word $w\in \mathcal{L}(\Sigma)$ such that $|w|\geq a-1$ and $(w)_g\equiv b\Mod{a}$, $A_\Sigma\cap (a\N+b)\cap [g^{a},\infty)=\emptyset$. Consequently, $|A_\Sigma\cap (a\N+b)|\leq g^a$, and thus $$\dim_{\text{M}}(A_\Sigma\cap (a\N+b))=0.$$
\end{proof}

From the discussion in \cref{Section6_1} and \cref{thm:Transversality}, the conclusion is the \cref{OpenQuestion} is affirmative if the multiplicatively invariant set can be represented by a transitive sofic shift. However, this is not necessarily true if the shift is merely transitive, or sofic. 

Nevertheless, we can extend the positive answer beyond transitive sofic subshifts.
\begin{definition}
Let $S\subseteq \N_0$.  We define the \emph{$S$-gap} subshift as
\[X_S:=\Big\{x\in \{0,1\}^{\N_0}:\ \text{ the number of 0's between two consecutive 1's  is in $S$} \Big\}.\]
\end{definition}
For $S\subseteq \N_0$, we denote $A_S:=A_{\Sigma_S}$.
\begin{proposition}
    Let  $S\subseteq\N_0$. For any arithmetic progression $P$, $\dim_{\text{M}}(A_{S}\cap P)$ is either $\dim_{\text{M}}(A_{S})$  or zero.
\end{proposition}
\begin{proof}
    If $|S|<\infty$, the result follows from \cref{thm:Transversality} as $\Sigma_S$ is transitive and sofic. Let us suppose $|S|=\infty$. Consider $s_n$ as the $n$-th element of $S$ in increasing order, and define $S_n:=\{s_1,\cdots,s_n\}$.  The subshift $X_{S_n}$ is a transitive and sofic shift for every $n\in\N$, and notice that $\mathcal{L}^n(X_{S_n})=\mathcal{L}^n(X_S)$. Consequently, it is not difficult to see that \begin{equation}
        h(X_S)=\sup_{n\in\N} h(X_{S_n}).\label{Transversality_sup}
    \end{equation}
    Suppose that there exists $w\in \mathcal{L}(X_S)$ such that $|w|\geq a-1$ and $(w)_g\equiv b \Mod{a}$. Then, there exists $n_0\in \N$ such that $w\in \mathcal{L}(X_{S_n})$ for every $n\geq n_0$. Since $X_{S_n}$ is sofic and transitive, we can apply \cref{thm:Transversality} to obtain that \begin{equation}
        \dim_{\text{M}}(A_{S_n}\cap (a\N+b))=\dim_{\text{M}}(A_{S_n})=h(X_{S_n})/\log(g).\label{Transversality_1}
    \end{equation}
    On the other hand,  
    \begin{equation}
        h(X_S)/\log(g)=\dim_{\text{M}}(A_{S})\geq\dim_{\text{M}}(A_{S}\cap (a\N+b))\geq \dim_{\text{M}}(A_{S_n}\cap (a\N+b)).\label{Transversality_2}
    \end{equation}
    From \eqref{Transversality_1} and \eqref{Transversality_2}, 
    \[h(X_{S_n})/\log(g)\leq\dim_{\text{M}}( A_{S}\cap (a\N+b))\leq h(X_S)/\log(g).\]
  The conclusion follows from \eqref{Transversality_sup}.
\end{proof}
If $S$ is the set of prime numbers, it is easy to see that $X_S$ is transitive but not sofic. Moreover, note that the property can be extended to subshifts that can be approximated (in terms of entropy) from within by transitive sofic subshifts. In particular, the class of transitive sofic subshifts is not optimal to provide an affirmative answer. 

One possible refinement is to look at subshifts with a unique measure of maximal entropy (which includes all transitive sofic subshifts \cite{Weiss1973}). However, it is not difficult to see that the subshift constructed in \cref{prop:ExampleTransitive} also has a unique measure of maximal entropy. Thus, this class is not sufficient to settle the question either. Seeking a general answer to the open question, we will consider the following class.

\begin{definition}
    A subshift $\Sigma$ is said to be \emph{entropy minimal} if for every subshift $\Sigma'\subsetneq \Sigma$, $h(\Sigma)>h(\Sigma').$\label{EntropyMinimal}
\end{definition}
The class of transitive sofic subshifts and $S$-gap subshifts are examples of entropy minimal subshifts (see \cite[Remark 2.4]{CT} and \cite{GR-P}). On the other hand, the counterexample provided in \cref{prop:ExampleTransitive} is not entropy minimal. Based on the techniques used in this paper, we think it should be possible to extend the affirmative answer to the open question to subshifts having left almost specification with bounded function  (see \cite[Definition 2.14]{PCl}), which is a subclass of the entropy minimal subshifts. More generally, we propose the following conjecture.
\begin{Conjecture}
    Let $A\subseteq \N_0$ be a $\times g$-invariant set. There exists an entropy minimal subshift $\Sigma\subseteq \Big\{0,\cdots,g-1\Big\}^{\N_0}$ such that $A=A_\Sigma$ if and only if, for every arithmetic progression $P$, it holds that $\dim_{\text{M}}(A\cap P)$ is either 0 or $\dim_{\text{M}}(A)$.
\end{Conjecture}

\bibliographystyle{abbrv}
\bibliography{bibliography}

\begin{thebibliography}{10}

\bibitem{A-A}
S.~Akiyama and A.~Peth\H{o}.
\newblock On canonical number systems.
\newblock {\em Theoret. Comput. Sci.}, 270(1-2):921--933, 2002.

\bibitem{Aloui}
K.~Aloui.
\newblock Sur les entiers ellips\'ephiques: somme des chiffres et r\'epartition dans les classes de congruence.
\newblock {\em Period. Math. Hungar.}, 70(2):171--208, 2015.

\bibitem{AMM}
K.~Aloui, C.~Mauduit, and M.~Mkaouar.
\newblock Somme des chiffres et r\'epartition dans les classes de congruence pour les palindromes ellips\'ephiques.
\newblock {\em Acta Math. Hungar.}, 151(2):409--455, 2017.

\bibitem{Austin}
T.~Austin.
\newblock A new dynamical proof of the {S}hmerkin-{W}u theorem.
\newblock {\em J. Mod. Dyn.}, 18:1--11, 2022.

\bibitem{BCS}
W.~D. Banks, A.~Conflitti, and I.~E. Shparlinski.
\newblock Character sums over integers with restricted {$g$}-ary digits.
\newblock {\em Illinois J. Math.}, 46(3):819--836, 2002.

\bibitem{BWS}
W.~D. Banks and I.~E. Shparlinski.
\newblock Arithmetic properties of numbers with restricted digits.
\newblock {\em Acta Arith.}, 112(4):313--332, 2004.

\bibitem{BS}
R.~Bellman and H.~N. Shapiro.
\newblock On a problem in additive number theory.
\newblock {\em Ann. of Math. (2)}, 49:333--340, 1948.

\bibitem{Besineau}
J.~B\'esineau.
\newblock Ind\'ependance statistique d'ensembles li\'es \`a{} la fonction ``somme des chiffres''.
\newblock {\em Acta Arith.}, 20:401--416, 1972.

\bibitem{PCl}
V.~Climenhaga and R.~Pavlov.
\newblock One-sided almost specification and intrinsic ergodicity.
\newblock {\em Ergodic Theory Dynam. Systems}, 39(9):2456--2480, 2019.

\bibitem{CT}
V.~Climenhaga and D.~J. Thompson.
\newblock Beyond {B}owen's specification property.
\newblock In {\em Thermodynamic formalism}, volume 2290 of {\em Lecture Notes in Math.}, pages 3--82. Springer, Cham, [2021] \copyright 2021.

\bibitem{SC}
S.~Col.
\newblock Diviseurs des nombres ellips\'ephiques.
\newblock {\em Period. Math. Hungar.}, 58(1):1--23, 2009.

\bibitem{DM1}
C.~Dartyge and C.~Mauduit.
\newblock Nombres presque premiers dont l'\'ecriture en base {$r$} ne comporte pas certains chiffres.
\newblock {\em J. Number Theory}, 81(2):270--291, 2000.

\bibitem{DM2}
C.~Dartyge and C.~Mauduit.
\newblock Ensembles de densit\'e{} nulle contenant des entiers poss\'edant au plus deux facteurs premiers.
\newblock {\em J. Number Theory}, 91(2):230--255, 2001.

\bibitem{DMS}
C.~Dartyge, C.~Mauduit, and A.~S\'ark\"ozy.
\newblock Polynomial values and generators with missing digits in finite fields.
\newblock {\em Funct. Approx. Comment. Math.}, 52(1):65--74, 2015.

\bibitem{Dartyge_Sarkozy}
C.~Dartyge and A.~S\'ark\"ozy.
\newblock The sum of digits function in finite fields.
\newblock {\em Proc. Amer. Math. Soc.}, 141(12):4119--4124, 2013.

\bibitem{DMC}
M.~Drmota and C.~Mauduit.
\newblock Weyl sums over integers with affine digit restrictions.
\newblock {\em J. Number Theory}, 130(11):2404--2427, 2010.

\bibitem{DMR}
M.~Drmota, C.~Mauduit, and J.~Rivat.
\newblock The sum-of-digits function of polynomial sequences.
\newblock {\em J. Lond. Math. Soc. (2)}, 84(1):81--102, 2011.

\bibitem{EMS1}
P.~Erd\H{o}s, C.~Mauduit, and A.~S\'ark\"ozy.
\newblock On arithmetic properties of integers with missing digits. {I}. {D}istribution in residue classes.
\newblock {\em J. Number Theory}, 70(2):99--120, 1998.

\bibitem{EMS2}
P.~Erd\H{o}s, C.~Mauduit, and A.~S\'ark\"ozy.
\newblock On arithmetic properties of integers with missing digits. {II}. {P}rime factors.
\newblock volume 200, pages 149--164. 1999.
\newblock Paul Erd\H os memorial collection.

\bibitem{Fine}
N.~J. Fine.
\newblock The distribution of the sum of digits {$({\rm mod}\ p)$}.
\newblock {\em Bull. Amer. Math. Soc.}, 71:651--652, 1965.

\bibitem{Fischer1}
R.~Fischer.
\newblock Sofic systems and graphs.
\newblock {\em Monatsh. Math.}, 80(3):179--186, 1975.

\bibitem{Fischer2}
R.~Fischer.
\newblock Graphs and symbolic dynamics.
\newblock In {\em Topics in information theory ({S}econd {C}olloq., {K}eszthely, 1975)}, volume Vol. 16 of {\em Colloq. Math. Soc. J\'anos Bolyai}, pages 229--244. J\'anos Bolyai Math. Soc., Budapest, 1977.

\bibitem{Furst}
H.~Furstenberg.
\newblock Disjointness in ergodic theory, minimal sets, and a problem in {D}iophantine approximation.
\newblock {\em Math. Systems Theory}, 1:1--49, 1967.

\bibitem{GR-P}
F.~Garc\'ia-Ramos and R.~Pavlov.
\newblock Extender sets and measures of maximal entropy for subshifts.
\newblock {\em J. Lond. Math. Soc. (2)}, 100(3):1013--1033, 2019.

\bibitem{Gelfond}
A.~O. Gelfond.
\newblock Sur les nombres qui ont des propri\'et\'es additives et multiplicatives donn\'ees.
\newblock {\em Acta Arith.}, 13:259--265, 1967/68.

\bibitem{G_Thus}
B.~Gittenberger and J.~M. Thuswaldner.
\newblock Asymptotic normality of {$b$}-additive functions on polynomial sequences in the {G}aussian number field.
\newblock {\em J. Number Theory}, 84(2):317--341, 2000.

\bibitem{GMR}
D.~Glasscock, J.~Moreira, and F.~K. Richter.
\newblock Additive and geometric transversality of fractal sets in the integers.
\newblock {\em J. Lond. Math. Soc. (2)}, 109(5):Paper No. e12902, 55, 2024.

\bibitem{Katai}
I.~K\'atai.
\newblock Construction of number systems in algebraic number fields.
\newblock {\em Ann. Univ. Sci. Budapest. Sect. Comput.}, 18:103--107, 1999.

\bibitem{Kempner01021914}
A.~J. Kempner.
\newblock A curious convergent series.
\newblock {\em Amer. Math. Monthly}, 21(2):48--50, 1914.

\bibitem{KDH}
D.-H. Kim.
\newblock On the joint distribution of {$q$}-additive functions in residue classes.
\newblock {\em J. Number Theory}, 74(2):307--336, 1999.

\bibitem{SK}
S.~Konyagin.
\newblock Arithmetic properties of integers with missing digits: distribution in residue classes.
\newblock {\em Period. Math. Hungar.}, 42(1-2):145--162, 2001.

\bibitem{KMS}
S.~Konyagin, C.~Mauduit, and A.~S\'ark\"ozy.
\newblock On the number of prime factors of integers characterized by digit properties.
\newblock {\em Period. Math. Hungar.}, 40(1):37--52, 2000.

\bibitem{Kovacs2}
A.~Kov\'acs.
\newblock On expansions of {G}aussian integers with non-negative digits.
\newblock {\em Math. Pannon.}, 10(2):177--191, 1999.

\bibitem{Kovacs}
B.~Kov\'acs.
\newblock Canonical number systems in algebraic number fields.
\newblock {\em Acta Math. Acad. Sci. Hungar.}, 37(4):405--407, 1981.

\bibitem{L_P}
D.~A. Levin and Y.~Peres.
\newblock {\em Markov chains and mixing times}.
\newblock American Mathematical Society, Providence, RI, second edition, 2017.
\newblock With contributions by Elizabeth L. Wilmer, With a chapter on ``Coupling from the past'' by James G. Propp and David B. Wilson.

\bibitem{Lind_Marcus_2021}
D.~Lind and B.~Marcus.
\newblock {\em An Introduction to Symbolic Dynamics and Coding}.
\newblock Cambridge Mathematical Library. Cambridge University Press, 2 edition, 2021.

\bibitem{MM}
M.~G. Madritsch.
\newblock The sum-of-digits function of canonical number systems: distribution in residue classes.
\newblock {\em J. Number Theory}, 132(12):2756--2772, 2012.

\bibitem{Mahler}
K.~Mahler.
\newblock On the generating functions of integers with a missing digit.
\newblock {\em K'o Hs\"ueh (Science)}, 29:265--267, 1947.

\bibitem{MR}
C.~Mauduit and J.~Rivat.
\newblock Sur un probl\`eme de {G}elfond: la somme des chiffres des nombres premiers.
\newblock {\em Ann. of Math. (2)}, 171(3):1591--1646, 2010.

\bibitem{MS2}
C.~Mauduit and A.~S\'ark\"ozy.
\newblock On the arithmetic structure of the integers whose sum of digits is fixed.
\newblock {\em Acta Arith.}, 81(2):145--173, 1997.

\bibitem{MAY1}
J.~Maynard.
\newblock Primes with restricted digits.
\newblock {\em Invent. Math.}, 217(1):127--218, 2019.

\bibitem{MAY2}
J.~Maynard.
\newblock Primes and polynomials with restricted digits.
\newblock {\em Int. Math. Res. Not. IMRN}, (14):1--23 [10626--10648 on table of contents], 2022.

\bibitem{MS}
N.~G. Moshchevitin and I.~D. Shkredov.
\newblock On the multiplicative properties modulo {$m$} of numbers with missing digits.
\newblock {\em Mat. Zametki}, 81(3):385--404, 2007.

\bibitem{MarkovChain2}
J.~R. Norris.
\newblock {\em Discrete-time Markov chains}, page 1–59.
\newblock Cambridge Series in Statistical and Probabilistic Mathematics. Cambridge University Press, 1997.

\bibitem{MarkovChain1}
J.~S. Rosenthal.
\newblock Convergence rates for markov chains.
\newblock {\em SIAM Review}, 37(3):387--405, 1995.

\bibitem{Sandor_Crstici}
J.~S\'andor and B.~Crstici.
\newblock {\em Handbook of number theory. {II}}.
\newblock Kluwer Academic Publishers, Dordrecht, 2004.

\bibitem{Shm}
P.~Shmerkin.
\newblock On {F}urstenberg's intersection conjecture, self-similar measures, and the {$L^q$} norms of convolutions.
\newblock {\em Ann. of Math. (2)}, 189(2):319--391, 2019.

\bibitem{Spandl}
C.~Spandl.
\newblock Computing the topological entropy of shifts.
\newblock {\em MLQ Math. Log. Q.}, 53(4-5):493--510, 2007.

\bibitem{CS}
C.~Swaenepoel.
\newblock Prescribing digits in finite fields.
\newblock {\em J. Number Theory}, 189:97--114, 2018.

\bibitem{Thus}
J.~M. Thuswaldner.
\newblock The sum of digits function in number fields: distribution in residue classes.
\newblock {\em J. Number Theory}, 74(1):111--125, 1999.

\bibitem{WW}
A.~Walker and A.~Walker.
\newblock Arithmetic progressions with restricted digits.
\newblock {\em Amer. Math. Monthly}, 127(2):140--150, 2020.

\bibitem{Weiss1973}
B.~Weiss.
\newblock Subshifts of finite type and sofic systems.
\newblock {\em Monatshefte für Mathematik}, 77:462--474, 1973.

\bibitem{Wu}
M.~Wu.
\newblock A proof of {F}urstenberg's conjecture on the intersections of {$\times p$}- and {$\times q$}-invariant sets.
\newblock {\em Ann. of Math. (2)}, 189(3):707--751, 2019.

\bibitem{Yu}
H.~Yu.
\newblock Additive properties of numbers with restricted digits.
\newblock {\em Algebra Number Theory}, 15(5):1283--1301, 2021.

\end{thebibliography}

\end{document}